%% file: main.tex
\documentclass[12pt]{amsart}
\usepackage{amssymb,euscript,amsmath, mathrsfs,amscd}
\usepackage{mathabx}
\usepackage{hyperref}
\usepackage{caption,wrapfig}
\usepackage{subcaption}
\usepackage[dvips]{color}
\usepackage{enumitem}
\usepackage{fullpage}
\usepackage[font=small,labelfont=bf]{caption}
\usepackage{tikz}

\usetikzlibrary{calc}
\usetikzlibrary{intersections}
\newcounter{ENUM}

\def\namedlabel#1#2{\begingroup
#2%
\def\@currentlabel{#2}%
\phantomsection\label{#1}\endgroup
}

\newcommand{\cover}{\prec\mathrel{\mkern-5mu}\mathrel{\cdot}}
\newcommand{\beas}{\begin{eqnarray*}}
\newcommand{\eeas}{\end{eqnarray*}}

\newtheorem{thm}{Theorem}[section]

\newtheorem{prop}[thm]{Proposition}
\newtheorem{lem}[thm]{Lemma}

\newtheorem{cor}[thm]{Corollary}

\theoremstyle{definition}
\newtheorem{defn}[thm]{Definition}

\newtheorem{ques}[thm]{Question}
\newtheorem{ex}[thm]{Example}

\newtheorem{notn}[thm]{Notation}
\newtheorem{rem}[thm]{Remark}

\numberwithin{equation}{section}

\makeatletter
\def\moverlay{\mathpalette\mov@rlay}
\def\mov@rlay#1#2{\leavevmode\vtop{%
	\baselineskip\z@skip \lineskiplimit-\maxdimen
	\ialign{\hfil$\m@th#1##$\hfil\cr#2\crcr}}}
\newcommand{\charfusion}[3][\mathord]{
#1{\ifx#1\mathop\vphantom{#2}\fi
	\mathpalette\mov@rlay{#2\cr#3}
}
\ifx#1\mathop\expandafter\displaylimits\fi}
\makeatother

\makeatletter
\def\subsubsection{\@startsection{subsubsection}{3}%
\z@{.5\linespacing\@plus.7\linespacing}{-.5em}%
{\normalfont\bfseries}}
\makeatother

\newcommand{\bigcupdot}{\charfusion[\mathop]{\bigcup}{\cdot}}

\def\Z{{\mathbb Z}}
\def\R{{\mathbb R}}
\def\fS{{\mathfrak S}}

\def\ds{\displaystyle}
\newcommand{\bm}[1]{{\boldsymbol{#1}}}
\def\0{\bm{0}}
\def\a{\bm{a}}

\def\e{\bm{e}}

\def\n{\bm{n}}

\def\v{\bm{v}}

\def\u{\bm{u}}
\def\w{\bm{w}}
\def\x{\bm{x}}
\def\y{\bm{y}}

\def\balpha{\bm{\alpha}}
\def\bbeta{\bm{\beta}}

\def\bdelta{\bm{\delta}}

\newcommand\preceqdot{\mathrel{\ooalign{$\preceq$\cr
		\hidewidth\raise0.225ex\hbox{$\cdot\mkern0.5mu$}\cr}}}

\definecolor{amber}{rgb}{1.0, 0.75, 0.0}

\def\cF{{\mathcal F}}

\def\cK{\mathcal K}

\def\cN{\mathcal N}
\def\cS{\mathcal S}

\def\loday{\mathrm{LodAsso}}
\def\lodayfan{\Lambda}

\newcommand{\ee}{\textbf{e}}
\newcommand{\ff}{\textbf{f}}

\def\sT{\mathscr T}
\def\sO{\mathscr O}
\def\sP{\mathscr P}

\def\sI{\mathscr I}

\def\MM{{\mathscr M}}
\def\KK{{\mathscr K}}
\def\MB{{\mathscr B}}

\newcommand{\OO}{\mathcal{O}}

\def\FL{\mathcal{F}}

\def\x{\mathbf{x}}

\def\ncone{\operatorname{ncone}}
\def\conv{\operatorname{ConvexHull}}

\def\Perm{\operatorname{Perm}}
\def\PA{\operatorname{PermAsso}}

\def\val{\operatorname{val}}
\def\Itv{\operatorname{Itv}}

\newcommand{\hi}{\overline{i}}
\newcommand{\hj}{\overline{j}}

\newcommand{\1}{\mathbf{1}}

\newcommand{\A}{{\mathrm A}}

\newcommand{\DD}{{\mathrm D}}
\newcommand{\GG}{{\mathrm G}}
\newcommand{\II}{{\mathrm I}}

\newcommand{\LL}{{\mathrm L}}

\newcommand{\TT}{{\mathrm T}}

\newcommand{\ZZ}{\mathbb{Z}}

\def\cF{ {\mathcal{F}}}
\def\fS{ {\mathfrak{S}}}

\def\Type{ {\operatorname{Type}}}

\def\Perm{\operatorname{Perm}}
\def\ncone{\operatorname{ncone}}

\def\1{ {\bf{1}}}
\def\balpha{ {\bm \alpha}}
\def\bbeta{ {\bm \beta}}

\numberwithin{equation}{section}

\definecolor{darkgreen}{rgb}{0,0.7,0}
\definecolor{brown}{rgb}{0.7,0.4,0}

\newcommand{\Br}{\textrm{Br}}

\newcommand\commentout[1]{}
\author{Federico Castillo}
\thanks{Federico Castillo is is partially supported by FONDECYT Regular Grant \#1221133.} 
\address{Federico Castillo, Departamento de Matem\'aticas, Pontificia Universidad Cat\'olica de Chile, Santiago, Chile.}
\email{federico.castillo@mat.uc.cl}

\author{Fu Liu}
\thanks{Fu Liu is partially supported by a grant from the Simons Foundation \#426756 and an NSF grant \#2153897-0.} 
\address{Fu Liu, Department of Mathematics, University of California, Davis, One Shields Avenue, Davis, CA 95616 USA.}
\email{fuliu@math.ucdavis.edu}

\begin{document}
\title{The permuto--associahedron Revisited}

\begin{abstract}
	A classic problem connecting algebraic and geometric combinatorics is the realization problem: given a poset, determine whether there exists a polytope whose face lattice is the poset. In 1990s, Kapranov defined a poset as a hybrid between the face poset of a permutohedron and that of an associahedron, and he asked whether this poset is realizable. Shortly after his question was posed, Reiner and Ziegler provided a realization.
	Based on our previous work on the nested braid fan, we provide in this paper a different realization of Kapranov's poset by constructing the vertex set and the normal fan of a permuto-associahedron simultaneously.
	
\end{abstract}	

\maketitle

\section{Introduction}

A \emph{polytope} is the convex hull of finitely many points in an Euclidean space. Equivalently, a polytope can also be defined as a bounded solution set of a finite system of linear inequalities. Either definition provides a direct geometric embedding of a given polytope. However, sometimes people are more interested in combinatorial properties of a polytope, which is captured by its face poset. The \emph{face poset} of a polytope $P$, denote $\FL(P),$ is the poset of all faces of $P$ ordered by inclusion. This leads to a classic question: given a poset $\cF$, determine whether there exists a polytope $P$ such that $\cF= \FL(P)$. 
If the answer is yes, we say $\cF$ is \emph{realizable}, and $P$ is a \emph{realization} of $\cF.$ We say a polytope arises or is defined \emph{abstractly} if it was initially constructed as an answer to a realization problem.

Two of the most studied polytopes in geometric combinatorics are the \emph{permutohedron} and the \emph{associahedron}, where the former was constructed by providing a geometric embedding, and the latter arised abstractly.
Several generalizations have been explored and developed; 
we mention a few: Chapoton and Fomin \cite{chapoton_fomin}, Fomin and Reading \cite{fomin_reading}, Pilaud and Pons \cite{pilaud_pons}, and Lange and Pilaud \cite{lange_pilaud}.
Particularly relevant for the present paper is the work of Gelfand, Kapranov, and Zelevinsky \cite{GKZ}, and Postnikov \cite{bible}.
We also mention the work of Reading \cite{reading}, Stella \cite{stella}, and subsequent generalizations 
by Hohlweg, Lange, and Thomas \cite{hohlweg2011permutahedra}, and by Hohlweg, Pilaud, and Stella \cite{hohlweg2018polytopal}.

\subsection{Motivation: Realizing Kapranov's poset}
The main purpose of this paper is to give a realization of a poset defined by Kapranov \cite{kapranov} as a hybrid between the face poset of a permutohedron and that of an associahedron. (See Definition \ref{defn:kap}.)	
Before giving more details about Kapranov's poset, we give a brief introduction to permutohedra and associahedra.

\subsubsection*{Permutohedron.} 

A $d$-dimensional \emph{permutohedron} or a \emph{$d$-permutohedron} is the convex hull of all coordinate permutations of a fixed (generic) point in $\R^{d+1}$. After being first introduced and studied by Schoute in 1911 \cite{schoute}, the family of permutohedra naturally appeared in many different fields of mathematics.
They can be described as the set of diagonal vectors of hermitian matrices with a fixed spectrum
\cite[Chapter II.6]{barvinokconvex},
as simple zonotopes \cite[Section 7]{zie} or as the Newton polytopes of the Schur polynomials \cite{bayer}.	
It is known that the face poset of a permutohedron is the poset of ordered set partitions \cite[Chapter VI, Proposition 2.2]{barvinokconvex}.

\subsubsection*{Associahedron} 

A $d$-dimensional \emph{associahedron} or a \emph{$d$-associahedron} is a polytope defined by the following property: its vertices $v_B$ are in bijection with full bracketings $B$ on $(\ell_1 \ast \ell_2 \ast \dots \ast \ell_{d+2}),$ and two vertices $v_B$ and $v_{B'}$ form an edge if and only if the bracketings $B$ and $B'$ are related by a single application of the associative law. 
This description of the associahedron can be translated into an equivalent description in terms of complete binary threes, using a connection between full bracketings and complete binary trees. (See Remark \ref{rem:associativity} for this connection.)
We want to also mention another well-known equivalent but different way of defining a $d$-associahedron abstractly: its vertices $v_T$ are in bijection with triangulations $T$ of a $(d+3)$-gon and two vertices $v_T$ and $v_{T'}$ form an edge of if and only if the triangulations $T$ and $T'$ differ by a flip.
All three descriptions above only define vertices and edges of associahedra.
A complete description of the face poset of the associahedron will be given in Definition \ref{def:poset} using the language of trees. 
We remark that it is not a coincidence that full bracketings, complete binary trees, and triangulations of a polygon all belong to Catalan families. 

The associahedron was initially defined abstractly by Stasheff \cite{stasheff_1963} and for a while whether there exists a geometric realization was an open problem.
See \cite{ceballos_santos_ziegler} for an historical account.
After decades of insights by many mathematicians, several realizations of the associahedron have been found. Below we mention two that are relevant to this paper: 

\begin{enumerate}[leftmargin=*, label=(\alph*)]
	\item Gelfand, Kapranov and Zelevinsky \cite[Chapter 7]{GKZ} provided a realization by considering regular triangulations of polytopes of arbitrary dimension.
	This realization was further generalized by Billera and Sturmfels \cite{billera1992fiber} in their construction of fiber polytopes.
	\item \label{itm:Loday} Loday gave a realization by providing explicit coordinates for vertices of associahedra \cite{loday} and showed that this realization is a deformation of the regular permutohedron.
	His construction recovers the realization of Stasheff and Shnider \cite[Appendix B]{stasheff_1995} of associahedra that proceeds by truncating faces of a standard simplex.
	Later Postnikov \cite{bible} showed that Loday's associahedron can be expressed as a Minkowski sum of simplices.

\end{enumerate}

\subsubsection*{Permuto-associahedron}

Motivated by providing a geometric proof for MacLane's coherence theorem for associativities and commutativities in monoidal categories \cite{maclane}, Kapranov constructed a poset whose elements are ordered set partitions with bracketings and ordered by either removing bracketings or merging blocks. (See Definition \ref{defn:kap} for a precise definition for this poset.)
He then showed that his poset can be realized as a CW-ball.
Using this, he provided a short proof for MacLane's coherence theorem.
In the introduction of \cite{kapranov}, Kapranov asked a natural question: does his poset have a geometric realization as a polytope?
Shortly after Kapranov's question was posed, Reiner and Ziegler \cite{reiner_ziegler} gave an affirmative answer with an explicit construction, using Gelfand, Kapranov, and Zelevinsky's realization of the associahedron mentioned above.  
A second realization was obtained by Gaiffi \cite{gaiffi}.

Reiner-Ziegler's \cite{reiner_ziegler} and Gaiffi's \cite{gaiffi} work are both related to the construction that will be given in this paper, but in different ways. Our construction is connected to Reiner and Ziegler's topological proof for the result that Kapranov's poset is a CW-ball, but is very different from their geometric realization. 
Meanwhile, the polytopes we construct have the same normal fan as Gaiffi's.
However, the constructed families of polytopes are different and more importantly our approaches are different:
Whereas his starting point is the work of Stasheff and Shnider \cite{stasheff_1995}, ours is Loday's \cite{loday}.
In Section \ref{sec:comparison}, we will compare our construction with both Reiner-Ziegler's and Gaiffi's, discussing both similarities and differences.

\subsection{Our construction} 

In our previous work \cite{def-cone}, we have defined and studied a family of polytopes called \emph{nested permutohedra},
which interpolate the structures of two permutohedra of consecutive dimensions.
In the present paper we use the tools and ideas developed in \cite{def-cone} to construct a permuto-associahedron as a \emph{deformation} of a nested permutohedron. In other words, we obtain a permuto-associahedron by altering the inequality description of a nested permutohedron without overrunning any vertex.
We call our realization the \emph{nested permuto-associahedron}.

Our main strategy for constructions in both this article and \cite{def-cone} involves a primal/dual argument which we lay out below: 
Suppose we want to construct a polytope whose vertices are in bijection with a certain set $S$ (e.g., if we try to realize a poset $\cF$, then $S$ is the set of rank-$1$ elements of $\cF$). Then we do the following steps:
\begin{enumerate}[leftmargin=*]
	\item Define a point $v_s$ for each $s\in S$. 
	\item Define a cone $\sigma_s$ for each $s\in S$, and let $\cN$ be the fan induced by the set $\{\sigma_s\}.$ 
	(If we want to realize a poset $\cF$, we need to make sure that the face poset of $\cN$ is dual to the poset $\cF$.) 
	\item Verify that the set $\{v_s\}$ and the set $\{\sigma_s\}$ ``match''. (See the hypothesis in Lemma \ref{lem:main_tool} for a precise statement.)
\end{enumerate}
After the above procedure, particularly after the last verification step, we can immediately conclude that the set $\{v_s\}$ constructed in step (1) and the fan $\cN$ constructed in step (2) are the vertex set and the normal fan respectively of the desired polytope. Moreover, we can immediately apply Lemma \ref{lem:det-ineq} to obtain an inequality description of the constructed polytope, which is another benefit of our construction strategy.

In this paper, we first describe a realization of the associahedron using the 3-step method outlined above. Explicit coordinates for Step (1) are given as a generalization of Loday's construction, and the cones in Step (2) are given by a union of \emph{braid cones} \cite{faces}. 
We then combine this realization with our previous results on nested permutohedra \cite{def-cone} to give a realization of Kapranov's poset which we call \emph{nested permuto-associahedra}.

\begin{figure}[ht]
	\centering
	\begin{subfigure}{.5\textwidth}
		\centering
		\input{tikz/nested3.tex}
		
	\end{subfigure}%
	\begin{subfigure}{.5\textwidth}
		\centering
		\input{tikz/permasso3.tex}
		
	\end{subfigure}
	\caption{A side by side comparison of a nested permutohedron (on the left) and a nested permuto-associahedron (on the right).}
	\label{fig:permasso}
\end{figure}
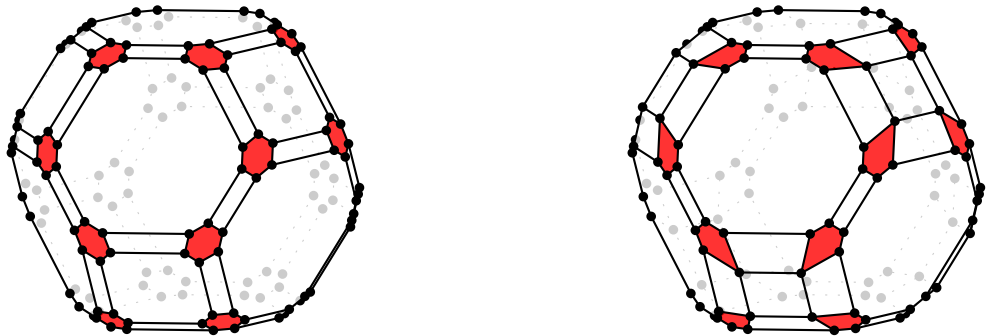

We finish this part by mentioning one more related construction of polytopes. 
Recently Baralic, Ivanovic, and Petric \cite{baralic_ivanovic_petric} constructed a \emph{simple} permuto-associahedron. Since Kapranov's permuto-associahedron which is not simple, these two families of permuto-associahedra clearly have different combinatorial structures.

\subsection*{Organization of the paper} 
After reviewing basic preliminary material in Section \ref{sec:prelim} we proceed, in Sections \ref{sec:trees} and \ref{sec:asso}, to describe a realization of the associahedron.
Note that Section \ref{sec:trees} serves as a preliminary section for Section \ref{sec:asso}. In particular, \S \ref{subsec:sbt}, \S \ref{subsec:asso} and \S \ref{subsec:labeling} introduce basic combinatorial objects that have been used in various other realizations of the associahedron, experts on associahedra may skip these parts.
In Sections \ref{sec:orderedtrees} and \ref{sec:permasso}, we use the results of previous sections together with results on nested permutohedra \cite{def-cone} to realize Kapranov's poset.  

Finally, in Section \ref{sec:comparison}, we compare our construction with Reiner-Ziegler's and Gaiffi's realizations.

\subsection*{Acknowledgements} 
We thank Vic Reiner for many insightful comments and the anonymous referees for suggestions to improve the presentation. 
This project began when the authors were attending the program ``Geometric and Topological Combinatorics'' at the Mathematical Sciences Research Institute in Berkeley, California, during the Fall 2017 semester, and they were partially supported by the NSF grant DMS-1440140.
The first named author is also partially supported by FONDECYT Grant 1221133, and the second named author is partially supported by a grant from the Simons Foundation \#426756 and an NSF grant \#2153897-0.

\section{Preliminaries}
\label{sec:prelim}

Recall that $[n]$ denotes the set $\{1, 2, \dots, n\}$ for any positive integer $n$ and $[a,b]$ denotes the \emph{integer interval} $\{z \in \Z: a \le z \le b\}$ for any integers $a$ and $b$ satisfying $a-1 \le b$. By convention $[a, a-1]$ denotes the \emph{empty integer interval}.

\subsection{Preorder and preposet}\label{subsec:preorder}
Let $\A$ be a finite set. A \emph{preorder} $\preceq$ on $\A$ is a binary relation that is both reflexive and transitive.  
If $i\preceq j$ and $j\preceq i$, we write $i \equiv j$. The relation $\equiv$ is an equivalence relation on $\A$, and thus it decomposes $\A$ into equivalence classes. 
We denote by $\hi$ the equivalence class of $i$ and $\A/_{\equiv}$ the set of equivalence classes. 
A \emph{preposet} is an ordered pair $(\A,\preceq)$ where $\preceq$ is a preorder on $\A$. A \emph{poset} is a preposet such that $i\equiv j$ if and only if $i=j$.
Note that if a preorder $\preceq$ is antisymmetric, which implies that $i\equiv j$ if and only if $i=j$, then $\preceq$ is a partial order on $\A$, and the preposet $(\A, \preceq)$ is a poset. (See \cite[Chapter 3]{ec1} for concepts related to partial orders and posets.)

One sees that a preorder $\preceq$ on $\A$ induces a partial order on $\A/_\equiv$ in which $\hi \preceq \hj$ if $i \preceq j$ in $\A$. The poset  $(\A/_\equiv, \preceq)$ and the preposet $(\A, \preceq)$ are closely related. Hence, we can conveniently extend many concepts for the former to the latter:
A \emph{covering relation} $i\cover j$ in the preposet $(\A, \preceq)$ is a pair of elements $(i,j)$ such that 
$\hi$ is covered by $\hj$ in the poset $(\A/_\equiv, \preceq)$. The \emph{Haase diagram} of a preposet $(\A, \preceq)$ is the Haase diagram of the poset $(\A/_\equiv, \preceq)$ except that for convenience when we mark vertices with equivalence classes $\hi$, we omit the parentheses around sets, see Figure \ref{fig:partitionorder} for an example. 

Suppose $\preceq_1$ and $\preceq_2$ are two preorders on $\A$. We say the preorder $\preceq_1$ is a \emph{contraction} of the preposet $\preceq_2$ if the Haase diagram of the former is obtained by contracting some edges of the Haase diagram of the latter and merging the corresponding equivalence classes.

An \emph{order-preserving} map from a preposet $(\A_1, \preceq_1)$ to another preposet $(\A_2, \preceq_2)$ is a bijection $f: \A_1 \to \A_2$ such that $f(i) \preceq_2 f(j)$ whenever $i \preceq_1 j$ for any $i,j \in \A_1.$ 
An order-preserving map is an \emph{isomorphism} if it is invertible and its inversion is order-preserving as well. 
Two preposets are \emph{isomorphic} if there exists an isomorphism between them.

Note that any subset $\mathrm{C}$ of $\Z$ (or $\R$) as is totally ordered with respect to $\le$ and thus can be considered as a preposet (or a poset); we use the letter $\mathrm{C}$ alone to indicate the preposet $(\mathrm{C}, \le)$ for simplicity.

Suppose $(\A, \preceq)$ is a poset.
The \emph{dual} of $(\A, \preceq)$ is the poset $(\A, \preceq^*)$ where $i\preceq^* j$ if and only if $j\preceq i$.
An order-preserving map from the poset $(\A, \preceq)$ to the set $\{1, 2, \dots, |\A|\}$ is called a \emph{linear extension} of $(\A, \preceq).$ We denote by $\LL[\A, \preceq]$ the set of linear extensions of $(\A, \preceq)$.

A poset $(\A,\preceq)$ is \emph{graded} if there exists a function $\rho:\A\to\Z_{\geq 0}$ such that $\rho(i)=0$ for every minimal element of the poset and $\rho(j)=\rho(i)+1$ whenever $i\cover j$ is a covering relation. 
We call $\rho$ (which is uniquely defined) the \emph{rank function} of $(\A, \preceq)$, and $\rho(i)$ the \emph{rank} of $i$ for each element $i.$ The \emph{rank} of a graded poset $(\A, \preceq)$ is defined as $r(\A,\preceq):=\max_{i\in \A} \rho(i)$.

We denote by $\hat{0}$ and $\hat{1}$ the minimum and maximum of a poset (if they exist). 

\subsection*{Our setup:} In our paper, we will mostly fix $\A =[n]$ where $n$ is either $d$ or $d+1$. Hence, when $n$ is fixed, the preorder $\preceq$ on $[n]$ is the only variable that changes. Whenever it is clear that $\preceq$ is a partial order on $\A=[n],$ we will omit $\A$ and just write $\LL[\preceq]$ for the set of linear extensions of $([n], \preceq)$. Note that $\LL[\preceq]$ is a subset of the symmetric group $\fS_{n}.$

\subsection{Polytopes and fans}

Let $V \subseteq \R^D$ be a $d$-dimensional vector space in the $D$-dimensional Euclidean space and $W$ is the dual space of $V$ which consists of all linear functionals on $V.$ Thus, we may consider $W$ is a quotient space of $\R^D$, and the perfect pairing between $V$ and $W:$ $\langle \cdot, \cdot \rangle: W \times V \to \R$ is just the dot product on $\R^D.$ 

Let $U \subseteq \R^D$ be an affine space that is a translation of $V.$  A \emph{polyhedron} $P \subseteq U$ is the solution set of a finite set of linear inequalities:
\begin{equation}\label{ineq}
	P = \{ \x \in U \ : \ \langle \a_i, \x \rangle \le b_i, \ i\in I \},
\end{equation}
A \emph{face} of a polyhedron is a subset $F\subseteq P$ such that there exists $\w\in W$ such that \[F=\left\{\x\in P~:~ \langle \w, \x \rangle \geq \langle \w, \y \rangle, \quad \forall \y \in P \right\}.\]
An inequality $\langle \a,\x\rangle \leq b$ is \emph{facet-defining} on $P$ if the corresponding equality defines a facet of $P$, i.e., $\{\x\in P : \langle \a,\x\rangle = b\}$ is face of $P$ of dimension $\dim(P)-1.$

Suppose a polyhedron $P$ is defined by \eqref{ineq}. We say \eqref{ineq} is a \emph{facet-defining inequality description} for $P$ if each inequality in \eqref{ineq} is facet-defining. (However, it is possible multiple inequalities determine a same facet.) We say \eqref{ineq} is a \emph{minimal inequality description} if $P$ has exactly $|I|$ facets. Thus, when \eqref{ineq} is minimal, the equality obtained for each $i \in I$ determines exactly one facet of $P.$ 
A \emph{polytope} is a bounded polyhedron. A $k$-dimensional polytope is \emph{simple} if each vertex 
is incident to exactly $k$ edges.
A \emph{(polyehdral) cone} is a polyhedron defined by homogeneous linear inequalities. A cone is \emph{pointed} if it does not contain a line. A $k$-dimensional cone is \emph{simplicial} if it is spanned by exactly $k$ (linearly independent) rays. Note that any simplicial cone is pointed.

By convention we always consider $\emptyset$ to be a face of a polyhedron $P$.
The set of all faces of $P$ partially ordered by inclusion forms the \emph{face poset} $\FL(P)$ of $P$. 
A \emph{fan} in $W$ is a collection $\Sigma$ of cones that is a simplicial complex. The collection together with the partial order given by inclusion forms a poset $\mathcal{F}(\Sigma)$ called its \emph{face poset}.
A fan $\Sigma$ is \emph{simplicial} if every cone in it is simplicial.  A fan $\Sigma$ in $W$ is \emph{complete} if the union of its cones is $W$.
The following definition gives a standard example of complete fans that arises from polytopes. 

\begin{defn}\label{defn:normal}
	Suppose $V, W$ and $U$ are given as above, and $P \subset U$ is a polytope.
	Given a nonempty face $F$ of $P$, the \emph{normal cone} of $P$ at $F$ 
	is defined to be
	\[
	\ncone(F, P) := \left\{ \w \in W \ : \ \langle \w, \y \rangle \geq \langle \w, \y' \rangle, \quad \forall \y \in F,\quad \forall \y' \in P \right\}.
	\]
	Therefore, $\ncone(F,P)$ is the collection of linear functionals $\w$ in $W$ such that $\w$ attains maximum value at $F$ over all points in $P.$ 
	The \emph{normal fan} of $P$, 
	denoted by $\Sigma(P)$, is the collection of all normal cones of $P$ as we range over all nonempty faces of $P$.

\end{defn}

\begin{lem}\label{lem:anti}
	The map $F\mapsto \ncone(F,P)$ for nonempty faces $F$ induces a poset isomorphism from $\FL(P)\setminus\{\emptyset\}$ to the dual
	poset of $\mathcal{F}(\Sigma(P))$.
\end{lem}

If $Q$ is a polytope such that $\Sigma(Q)$ is a coarsening of $\Sigma(P)$, i.e., if every cone in the former is a union of cones in the latter, we say that $Q$ is a \emph{deformation} of $P$.

As we mentioned above, $\Sigma(P)$ is always a complete fan in $W$. Moreover, any fan that is a normal fan of a polytope is called a \emph{projective} fan. Once we know that a projective fan $\Sigma$ is the normal fan of a polytope, one can check that the polytope is full-dimensional if and only if $0 \in \Sigma,$ i.e., all cones in $\Sigma$ are pointed.

Given a fan $\Sigma$ in $W,$ the set $\MM \subseteq \Sigma$ of maximal cones (in terms of dimension) determines $\Sigma$. More precisely, the set of cones in $\MM$, together with all their faces, forms the fan $\Sigma.$ In this case, we say $\Sigma$ is \emph{induced} by $\MM.$ Therefore, we often focus on the description of the maximal cones of a fan, which has the property of being the conic dissection of $W.$ 

\begin{defn}\label{def:dissection}
	A \emph{conic dissection} of $W$ is a set $\MM$ of full-dimensional cones such that the union of the cones in $\MM$ is equal to $W$, and for any distinct $\sigma_1, \sigma_2 \in \MM,$ their relative interiors $\sigma_1^\circ$ and $\sigma_2^\circ$ have no intersection.
	We say a conic dissection $\MM$ is \emph{pointed} (and \emph{simplicial} resp.) if all the cones in $\MM$ are \emph{pointed} (and \emph{simplicial} resp.)
\end{defn}

We remark that a conic dissection does not necessarily induce a fan, since cones in the dissection may not intersect in proper faces.

The primal/dual argument in the following lemma was used in the proof of Proposition 3.5 of our previous work \cite{def-cone}. We summarize it here since it will be our main tool in verifying the constructions of associahedra and permuto-associahedra.  
\begin{lem}\label{lem:main_tool}
	Let $\MM = \{\sigma_1,\dots,\sigma_k\}$ be a conic dissection of $W$ and $\{\v_1,\dots,\v_k\}\subseteq  U$ a set of points such that for each $i=1,\dots,k$ we have
	\begin{equation}\label{eq:max}
		\langle \w,\v_i\rangle > \langle \w,\v_j\rangle\qquad \forall \w\in \sigma_i^\circ \text{ and } j \neq i.
	\end{equation} 
	Let $P$ be the polytope defined by $P:=\conv\{\v_1,\dots,\v_k\}$. Then the followings are true: 
	\begin{enumerate}
		\item The set $\{\v_1, \dots, \v_k\}$ is the vertex set of $P.$
		\item For each $i=1,2,\dots,k$, we have $\sigma_i=\ncone(\v_i,P).$
	\end{enumerate}
	As a consequence, the conic dissection $\MM$ induces the normal fan $\Sigma(P)$ of $P$, which is a complete projective fan in $W$.
	Moreover, if $\MM$ is pointed, then $0 \in \Sigma(P)$ and thus $P$ is full-dimensional in $U.$
\end{lem}

\begin{proof}
	For each $i=1,2,\dots, k,$ it follows from condition \eqref{eq:max} that $\v_i$ does not lie in $\conv(\v_j \ : \ j \neq i),$ and thus $\v_i$ must be a vertex of $P.$ Hence, (1) follows. 
	Next, condition \eqref{eq:max} also implies that for each $i$  we have $\sigma_i \subseteq  \ncone(\v_i, P)$.
	However, since both $\{\sigma_i\}$ and $\{\ncone(v_i,P)\}$ are conic dissections of $W$, we must have $\sigma_i=\ncone(\v_i,P).$ 
\end{proof}

\begin{lem}\label{lem:det-ineq}
	Suppose $P$ is a full-dimensional polytope in $U$ with vertex set $\{\v_1, \dots, \v_k\},$ and $\sigma_i=\ncone(\v_i,P)$ for each $1 \le i \le k.$
	Let $\{\rho_1, \rho_2, \dots, \rho_m\}$ be the set of one dimensional cones in $\Sigma(P)$ and for each $1 \le j \le m,$ let $\n_j$ be a nonzero vector in the cone $\rho_j$ (or equivalently a generator for $\rho_j$). Then the polytope $P$ has the following minimal inequality description: 
	\[ 
	\setlength{\abovedisplayskip}{0pt}
	P = \left\{\x \in U \ :  \ \langle \n_j,\x \rangle \leq \max_{1 \le i \le k} \langle \n_j ,\v_i \rangle, \quad 1 \le j \le m \right\}. \]
	%
	Moreover, for each $1 \le j \le m,$ if we choose $i_j$ such that $\rho_j \subseteq  \sigma_{i_j}$, then \[\max_{1 \le i \le k} \langle \n_j ,\v_i \rangle = \langle \n_j ,\v_{i_j} \rangle.\]
\end{lem}

\subsection{Permutohedra and braid cones} In this paper, we always have $D=d+1$ where $D$ is the dimension of the ambient space $\R^D$ and $d$ is the dimension of the polytopes we consider. The $d$-dimensional vector space we use is $V_d = \{ \x \in \R^{d+1} \ : \langle\1, \x\rangle = 0\} \subseteq \R^{d+1}$ and its dual space is $W_d = \R^{d+1}/\1$, where $\1 = (1, 1, \dots, 1)$ denotes the all-one vector in $\R^{d+1}.$ 
%
For $\balpha \in \R^{d+1}$, let 
\begin{equation}\label{eq:defnU}	
	\setlength{\abovedisplayskip}{0pt}
	U_d^\balpha :=  \left\{ \x \in \R^{d+1} \ : \langle\1, \x\rangle = \sum_{i=1}^{d+1} \alpha_i\right\}
\end{equation}
be a translation of $V_d.$ Our polytopes will be defined in these affine spaces. 

Given a strictly increasing sequence $\balpha= (\alpha_1,\alpha_2,\dots,\alpha_{d+1}) \in \R^{d+1}$, for any $\pi\in \fS_{d+1}$, we use the notation below, following \cite{def-cone}: 
\begin{equation}\label{eq:permuto_vertex}
	\setlength{\abovedisplayskip}{0pt}
	\setlength{\belowdisplayskip}{0pt}
	v_\pi^\balpha := \left(\alpha_{\pi(1)},\alpha_{\pi(2)},\dots, \alpha_{\pi({d+1})}\right) = \sum_{i=1}^{d+1} \alpha_{i} \e_{\pi^{-1}(i)}.
\end{equation}
Then we define the \emph{usual permutohedron} 
\[ \Perm(\balpha) = \conv\left(v_\pi^{\balpha} \ : \ \pi\in \fS_{d+1}\right) \subseteq U_d^\balpha.\]
It is well-known that $\Perm(\balpha)$ is a full-dimensional polytope in $U_d^\balpha,$ and it has the following minimal inequality description, writing $\e_S := \sum_{i \in S} \e_i$: 
\begin{equation}\label{eq:perm_ineqs}
	\setlength{\abovedisplayskip}{0pt}
	\Perm(\balpha) = \left\{\x\in U_d^\balpha \ : \ \langle \e_S, \x \rangle \leq \sum_{i=d+2-|S|}^{d+1} \alpha_i,\quad \forall \emptyset \neq S\subsetneq [d+1] \right\}.
\end{equation}

A \emph{generalized permtuhohedron} is a deformation of a usual permutohedron $\Perm(\balpha)$ for some $\balpha$.

\begin{defn}\label{def:braid_cone}
	
	For each $\pi\in \fS_{d+1}$, we define the \emph{braid cone associated to $\pi$} to be:
	\begin{equation}\label{eq:braid_closed_defi}
		\sigma(\pi):=\{\w\in W_d \ :\ w_{\pi^{-1}(1)}\le w_{\pi^{-1}(2)} \le \dots \le w_{\pi^{-1}(d+1)}\}.
	\end{equation}
	Let $\MB_d := \{\sigma(\pi): \pi\in \fS_{d+1}\}$ be the collection of braid cones in $W_d.$
\end{defn}

One checks that the relative interior of $\sigma(\pi)$ is
\begin{equation}\label{eq:braid_open_defi}
	\sigma^\circ(\pi):=\{\w\in W_d \ :\ w_{\pi^{-1}(1)} < w_{\pi^{-1}(2)} < \dots < w_{\pi^{-1}(d+1)}\}.
\end{equation}
Thus, we clearly have the following:

\begin{lem}\label{lem:braid_dissect}
	The collection of braid cones $\MB_d = \{\sigma(\pi): \pi\in \fS_{d+1}\}$ forms a simplicial conic dissection of $W_d.$
\end{lem}

As an example of the usefulness of Lemma \ref{lem:main_tool}, we verified in \cite{def-cone} that the collection of braid cones $\MB_d$ and the set of points $\{v_\pi^\balpha: \pi \in \fS_{d+1}\}$ satisfy the hypotheses of the lemma. 
Consequently, we proved that $\MB_d$ induces the well-known \emph{braid fan} $\Br_d$, and that the braid fan $\Br_d$ is the normal fan of the usual permutohedron $\Perm(\balpha).$
Thus a polytope is a generalized permutohedron if and only if its normal fan coarsens the braid fan $\Br_d$ for some $d$. 

Finally, the face poset of a permutohedron has a nice combinatorial description.

\begin{defn}\label{defn:osp}
	We say the ordered tuple $\cS=(S_1,\dots,S_k)$ is an \emph{ordered (set) partition} of $[d+1]$ with \emph{$k$ blocks} if $S_1, \dots, S_k$ are $k$ disjoint sets whose union is $[d+1]$.
	We denote by $\sO_{d+1}$ the set of all ordered partitions of $[d+1]$ and by $\sO_{d+1,k}$ the set of all ordered partitions of $[d+1]$ with $k$ parts.
	
	We define a partial order $\trianglelefteq$ on the set $\sO_{d+1}\cup\{\hat{0}\}$ by declaring $\cS_1\trianglelefteq \cS_2$ if $\cS_1 \in \sO_{d+1}$ refines $\cS_2 \in \sO_{d+1}$ and $\hat{0}\trianglelefteq \cS$ for all $\cS\in\sO_{d+1}$. 
	We denote the poset $(\sO_{d+1}\cup\{\hat{0}\},\trianglelefteq)$ by $\OO_{d+1}$. 
\end{defn}

The set $\sO_{d+1,d+1}$ of rank $1$ elements of $\OO_{d+1}$ is in bijection with $\fS_{d+1}.$ More precisely, for each permutation $\pi \in \fS_{d+1}$, we let 
\begin{equation}\label{equ:cSpi}
	\cS(\pi) := ( \{\pi^{-1}(1)\}, \{\pi^{-1}(2)\}, \ldots, \{\pi^{-1}(d+1)\}).
\end{equation}
to be the ordered set partition that corresponds to $\pi.$ Clearly, $\pi \mapsto \cS(\pi)$ is a bijection from $\fS_{d+1}$ to $\sO_{d+1,1}.$

\begin{notn}
	
	When we write examples of ordered partitions we often omit commas and brackets for convenience. For example, $(\{1,2,3\},\{4\},\{5,6\})$ will be written as $(123,4,56).$
\end{notn}	

\begin{thm}{\cite[Section VI.2]{barvinokconvex}}
	Suppose $\balpha \in \R^{d+1}$ is a strictly increasing sequence. Then the face poset of the usual permutohedron $\Perm(\balpha)$ is isomorphic to $\OO_{d+1}.$ 
\end{thm}

Note that by Lemma \ref{lem:anti}, there is a poset isomorphism from the poset $\OO_{d+1} \setminus \{\hat{0}\}$ to the dual of the face poset $\cF(\Br_d)$ of the braid fan $\Br_d$. See Remark \ref{rem:order2cone} below.

\subsection{Preorder cones}
In \cite[Section 3.4]{faces}, the authors defined a ``braid cone'' as a polyhedral cone in $W_{n-1}$ whose defining inequalities are of the form $w_i\leq w_j$ for some $i,j\in[n]$. To avoid confusion with Definition \ref{def:braid_cone}, we will refer to these cones as \emph{preorder cones}.

\begin{defn}\label{def:posetcone}
	For each preorder $\preceq$ on the set $[n]$, we define the \emph{preorder cone} associated to $\preceq$ to be
	\[
	\setlength{\abovedisplayskip}{0pt}
	\setlength{\belowdisplayskip}{0pt}
	\sigma_\preceq := \{\w\in W_{n-1} \ : \ w_i\leq w_j \textrm{ if }i\preceq j\}.\]
\end{defn}

It is clear from the definition 
that every face of a preorder cone is itself a preorder cone.
\begin{rem} \label{rem:order2cone}
	For any $\cS=(S_1,\dots,S_k) \in \sO_{d+1}$, it determines a unique preorder $\preceq_\cS$ on $[d+1]$ by letting
	\begin{equation}
		\text{$i\leq_\cS j$ if $i\in S_a, j\in S_b$ with $a\leq b$.}
		\label{equ:ord_preorder}
	\end{equation}
	Then we define the cone $\sigma(\cS):=\sigma_{\leq_\cS}$.
	The set $\{\sigma(\cS)~:~\cS\in\sO_{d+1}\}$ consists of all cones in $\Br_d$.
	In particular, for $\pi\in\fS_{d}$ the cone $\sigma(\cS(\pi))$ is the braid cone $\sigma(\pi)$.
	
	The map $\cS \to \sigma(\cS)$ induces the poset isomorphism from $\OO_{d+1} \setminus \{\hat{0}\}$ to $\cF(\Br_d)$ that is asserted by Lemma \ref{lem:anti} with $P$ being the usual permutohedron $\Perm(\balpha)$.
\end{rem}

We state the following facts from \cite{faces} that will be useful in our constructions.

\begin{lem}[Proposition 3.5 in \cite{faces}]\label{lem:poset_facts} 

	Let $\preceq$ be a (fixed) preorder on the set $[n].$ Then the following statements hold. 
	\begin{enumerate}
		\item \label{itm:interior} 
		The preorder cone has the following facet-defining inequality description: 
		\begin{equation*}
			\sigma_\preceq = \{\w\in W_\preceq \ : \ w_i \leq w_j \textrm{ if }i\cover j\}.
		\end{equation*}
		Thus, the relative interior of $\sigma_\preceq$ is 
		\begin{equation*}
			\sigma_\preceq^\circ = \{\w\in W_\preceq \ : \ w_i < w_j \textrm{ if }i\cover j\}.
		\end{equation*}
		\item\label{itm:simplicial} The preorder cone $\sigma_{\preceq}$ is simplicial if and only if the Haase diagram of $\preceq$ is a tree. (Recall that a tree is a connected acyclic graph.)

		\item\label{itm:faces} 
		The preorder cone $\sigma_{\preceq'}$ is a face of $\sigma_\preceq$ if and only if  $\preceq'$ is a contraction of $\preceq$. 
		
		\item\label{itm:union} Suppose $\preceq$ is a partial order on $[n]$. Then its associated cone $\sigma_\preceq$ is a union of braid cones. More precisely, 
		$\sigma_\preceq =\bigcup_{\pi \in \LL[\preceq]} \sigma(\pi).$
		
	\end{enumerate}
\end{lem}

\section{Trees and the associahedron}\label{sec:trees}

In this section, we will review the combinatorics of the associahedron and develop results that will be needed in our constructions of associahedra and permuto-associahedra.
As mentioned in the introduction, instead of using bracketings, we will define the face poset of the associahedron in terms of trees.

\subsection{Strictly branching trees}\label{subsec:sbt}
We assume the readers are familiar with terminologies on graphs as presented in \cite[Appendix]{ec1}, and review ones that are relevant to our paper.

A \emph{rooted tree} is tree with a special vertex which is called the \emph{root} of the tree. For any edge $\{i,j\}$ of a rooted tree $\TT$, if $i$ is closer to the root of $\TT$ than $j,$ then we say $i$ is the \emph{parent} of $j$ and $j$ is a child of $i.$ We call a vertex of a rooted tree a \emph{leaf} if it has no children, and an \emph{internal vertex} otherwise. An edge of a rooted tree is \emph{internal} if it connects two internal vertices. 
A rooted tree is \emph{strictly branching} if each of its internal vertex has at least two children. 
An \emph{unlabeled plane rooted tree} $\TT$ is a rooted tree whose vertices are considered to be indistinguishable, but the subtrees at any vertex are linearly ordered.

For $n \in \ZZ_{\ge 0},$ let $\sT_{n}$ be the set of unlabeled plane rooted trees that are strictly branching and have $n+1$ leaves. For $0 \le k \le n,$ let $\sT_{n,k}$ be the set of trees in $\sT_{n}$ that has $k$ internal vertices. Note that $\sT_{n,0} = \emptyset$ unless $n=0,$ and $\sT_{0} = \sT_{0,0}$ consists of the only rooted tree on one vertex. It is easy to see that for any positive integer $n$ the followings are true:
\begin{enumerate}[leftmargin=*]
	\item 
	$\displaystyle \sT_{n} = \bigcupdot_{k=1}^{n} \sT_{n,k}.$ 
	\item $\sT_{n,1}$ consists of only one tree in which the root is the only internal vertex.
	
	\item $\sT_{n,n}$ consists of all the complete binary trees with $n+1$ leaves.
	Recall that a \emph{complete binary tree} is an unlabeled plane rooted tree whose internal vertices all have exactly two children. For each internal vertex of a complete binary tree, we call its first child its \emph{Left} child, and second child its \emph{Right} child. As a consequence, we call the two corresponding subtrees its \emph{Left} subtree and \emph{Right} subtree, and the two connecting edges a \emph{Left} internal edge and a \emph{Right} internal edge.
\end{enumerate}

See Figure \ref{fig:planetrees} for examples: the tree on the left is a complete binary tree in $\sT_{8,8}$ and the tree on the right is a plane rooted tree in $\sT_{8,5}.$ 
The leaves of both trees were enumerated by left-to-right order.

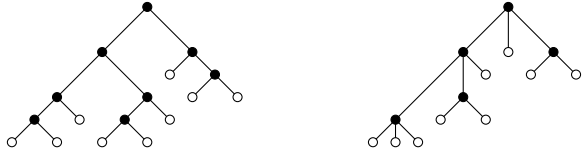
\begin{figure}[htp]
	\begin{tikzpicture}[scale=0.6]
		
		\begin{scope}[xshift=-4cm]
			\draw (0,0)--(3,3)--(4.5,1.5)--(5,1) (1,1)--(1.5,0.5) (0.5,0.5)--(1,0) (2,2)--(3,1)--(2,0)
			(3,1)--(3.5,0.5) (2.5,0.5)--(3,0) (4,2)--(3.5,1.5) (4.5,1.5)--(4,1);
			\draw[fill=white] (0,0) circle (.1);
			\draw[fill=white] (1,0) circle (.1);
			\draw[fill=black] (0.5,0.5) circle (.1);
			\draw[fill=black] (1,1) circle (.1);
			\draw[fill=white] (1.5,0.5) circle (.1);
			\draw[fill=black] (2,2) circle (.1);
			\draw[fill=white] (2,0) circle (.1);
			\draw[fill=white] (3,0) circle (.1);
			\draw[fill=black] (2.5,0.5) circle (.1);
			\draw[fill=black] (3,1) circle (.1);
			\draw[fill=white] (3.5,0.5) circle (.1);
			\draw[fill=black] (3,3) circle (.1);
			\draw[fill=black] (4,2) circle (.1);
			\draw[fill=white] (3.5,1.5) circle (.1);
			\draw[fill=black] (4.5,1.5) circle (.1);
			\draw[fill=white] (4,1) circle (.1);
			\draw[fill=white] (5,1) circle (.1);
			
		\end{scope}
		
		\begin{scope}[xshift=4cm]
			\draw (0,0)--(3,3)--(4.5,1.5) (0.5,0)--(0.5,0.5)--(1,0) (4,2)--(3.5,1.5) (3,3)--(3,2) (1.5,0.5)--(2,1)--(2,2) (2.5,0.5)--(2,1) (2,2)--(2.5,1.5);
			\draw[fill=white] (0,0) circle (.1);
			\draw[fill=white] (0.5,0) circle (.1);
			\draw[fill=black] (0.5,0.5) circle (.1);
			\draw[fill=white] (1,0) circle (.1);
			\draw[fill=black] (2,2) circle (.1);
			\draw[fill=white] (1.5,0.5) circle (.1);
			\draw[fill=white] (2.5,0.5) circle (.1);
			\draw[fill=black] (2,1) circle (.1);
			\draw[fill=white] (2.5,1.5) circle (.1);
			\draw[fill=black] (3,3) circle (.1);
			\draw[fill=black] (4,2) circle (.1);
			\draw[fill=white] (3.5,1.5) circle (.1);
			\draw[fill=white] (4.5,1.5) circle (.1);
			\draw[fill=white] (3,2) circle (.1);
		\end{scope}
	\end{tikzpicture}
	\caption{Example of unlabeled plane rooted trees that are strictly branching.}
	\label{fig:planetrees}
\end{figure}

\begin{defn}\label{defn:flips}	Let $\TT, \TT' \in \sT_{n,n}$. 
	We say $\TT$ is obtained from $\TT'$ by a \emph{flip (of an internal edge)} and $\TT'$ is obtained from $\TT$ by a \emph{flip (of an internal edge)} if there exist a Left internal edge $e$ of $\TT$ and a Right internal edge $e'$ of $\TT'$ such that after contracting $e$ in $\TT$ and contracting $e'$ in $\TT'$, we obtain exactly the same tree. 
	
\end{defn}	
See Figure \ref{fig:flip} for a demonstration of how a flip of an internal edge works.

\begin{figure}[htp]

	\begin{tikzpicture}[scale=0.6]
		
		\begin{scope}[xshift=-5cm]
			\node[draw] at (2,0) (1){$A$};
			\node[draw] at (4,0) (2){$B$};
			\node[draw] at (6,0) (3){$C$};
			
			\node at (3,1) (4){}; \draw[fill=black] (4) circle (.1);
			\node at (4,2) (5){}; \draw[fill=black] (5) circle (.1);

			\node[draw] at (4,3) (6){$D$};

			\draw (1)--(4) (5)--(3) (4)--(2) (6)--(5);
			\draw[ultra thick, red] (4)--(5);

		\end{scope}
		
		\begin{scope}[scale=0.6, xshift=3cm, yshift=4cm]
			\node[draw] at (2,0) (1){$A$};
			\node[draw] at (4,0) (2){$B$};
			\node[draw] at (6,0) (3){$C$};
			
			\node at (4,2) (5){}; \draw[fill=red] (5) circle (.1);

			\node[draw] at (4,3.2) (6){$D$};

			\draw (1)--(5) (5)--(3) (5)--(2) (4,2)--(4,2.5);

		\end{scope}

		\begin{scope}[xshift=5cm]
			
			\node[draw] at (2,0) (1){$A$};
			\node[draw] at (4,0) (2){$B$};
			\node[draw] at (6,0) (3){$C$};
			
			\node at (5,1) (4){}; \draw[fill=black] (4) circle (.1);
			\node at (4,2) (5){}; \draw[fill=black] (5) circle (.1);
			
			\node[draw] at (4,3) (6){$D$};

			\draw (1)--(5) (4)--(3) (4)--(2) (6)--(5);
			\draw[ultra thick, red] (4)--(5);
		\end{scope}

		\begin{scope}[yshift=1cm, xshift=4cm]
			\draw[->] (-2,-.5)--(2,-.5);
			\draw[->] (2,0.5)--(-2,0.5);
			\node at (0,0) {\small Flipping an internal edge};
		\end{scope}

	\end{tikzpicture}
	\caption{Flips on complete binary trees}
	\label{fig:flip}
\end{figure}
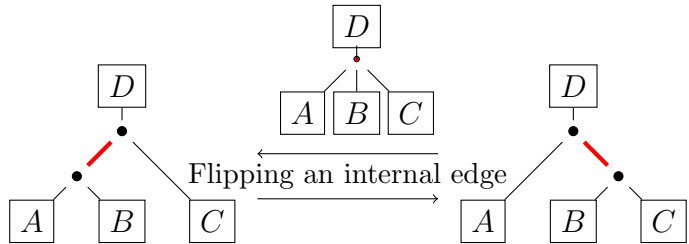

\subsection{The associahedron}\label{subsec:asso}
We can now define associahedron abstractly. Here we replace $n$ with $d+1$ in $\sT_{n}$ or $\sT_{n,k}.$  

\begin{defn}\label{def:asso}
	A \emph{$d$-associahedron} is a $d$-dimensional polytope such that its vertices are in bijection with complete binary trees in $\sT_{d+1,d+1}$, and two vertices form an edge if and only if their corresponding complete binary trees are obtained from one another by a flip.
\end{defn}

Every complete binary tree $\TT \in \sT_{d+1,d+1}$ has exactly $d$ internal edges, and thus is connected with exactly $d$ complete binary trees in $\sT_{d+1,d+1}$ via flips.
This means that an associahedron is a simple polytope. A classical result of Blind and Mani \cite{blind-mani} 
(with a very short alternative proof given by G.Kalai in \cite{kalai}) states that the graph\footnote{The graph $G(P)=(V,E)$ of a polytope $P$ is defined by taking $V$ be the vertex set of $P$ and two vertices are joined by an edge if they form a one dimensional face of $P$.} 
of a simple polytope determines its face poset. \footnote{By duality we have that the facet-ridge graph of a simplicial polytope determines the rest of the face poset. In \cite[Question 1]{blind-mani} it is asked whether the same property hold for simplicial spheres. To our best knowledge this generalization remains open to this day.} 
As a consequence, all realizations of $d$-associahedra share the same face poset which we describe below.

\begin{defn}\label{def:poset}
	We define a partial order $\leq_{K}$ on the set $\sT_{d+1}\cup\{\hat{0}\}$ by declaring $\TT_1\leq_K \TT_2$ whenever $\TT_2 \in \sT_{d+1}$ is obtained from $\TT_1 \in \sT_{d+1}$ by contracting \emph{internal edges}, and $\hat{0}\leq_K \TT$ for all $\TT\in\sT_{d+1}$. 
	We denote the poset $(\sT_{d+1}\cup\{\hat{0}\},\leq_K)$ by $\cK_{d}$. 
\end{defn}

It is easy to verify that the poset $\cK_{d}$ is graded of rank $d+1$ with a unique minimal element $\hat{0}$. For each $1 \le k \le d+1,$ the set $\sT_{d+1,k}$ consists of all elements of $\cK_d$ of rank $d+2-k$. In particular, the only tree in $\sT_{d+1,1}$ is the unique maximal element of $\cK_d.$

The following lemma is a well-known result, and often is taken as the definition of associahedra. See \cite[Proof of Lemma 3.3]{revisited-arxiv-v1} for a proof.
\begin{lem}\label{lem:FLdefineAsso}

	A polytope is a $d$-associahedron if and only if its face poset is $\cK_d$. 
\end{lem}

\subsection{Enumerating leaves and labeling internal vertices}\label{subsec:labeling}

For each tree in $\sT_{n}$, we canonically enumerate its leaves by left-to-right order, denoted by $\ell_1, \ell_2, \ell_3, \dots, \ell_{n+1}.$ 	For each tree $\TT\in \sT_{n}$, we also assign labels $1,2,\dots,n$ to its internal vertices: If an internal vertex $v$ is the closest common ancestor of $\ell_i$ and $\ell_{i+1}$ we label $v$ with $i$.

Suppose $\TT'$ is a subtree of $\TT.$ We denote by $\II_\TT(\TT')$ the collection of labels on the internal vertices of $\TT'$ (as a substree of  $\TT'$).
\begin{ex}
	In Figure \ref{fig:preorders} we depict the internal vertices of the trees of Figure \ref{fig:planetrees} together with the labelings with the set $[8]$.
	Let $\TT$ be the tree on the left of Figure \ref{fig:preorders}, and $\TT'$ the right subtree of $\TT.$ Then $\II_\TT(\TT') =\{7,8\}.$
	
	\begin{figure}[htp]
		\begin{tikzpicture}[scale=0.6]
			
			\begin{scope}[xshift=-4cm]

				\draw (0,0)--(3,3)--(4.5,1.5)--(5,1) (1,1)--(1.5,0.5) (0.5,0.5)--(1,0) (2,2)--(3,1)--(2,0)
				(3,1)--(3.5,0.5) (2.5,0.5)--(3,0) (4,2)--(3.5,1.5) (4.5,1.5)--(4,1);
				
				\draw[fill=black] (0.5,0.5) circle (.1);
				\draw[fill=black] (1,1) circle (.1);
				\draw[fill=black] (2,2) circle (.1);
				\draw[fill=black] (2.5,0.5) circle (.1);
				\draw[fill=black] (3,1) circle (.1);
				\draw[fill=black] (3,3) circle (.1);
				\draw[fill=black] (4,2) circle (.1);
				\draw[fill=black] (4.5,1.5) circle (.1);
				\draw[fill=white] (0,0) circle (.1);
				\draw[fill=white] (1,0) circle (.1);
				\draw[fill=white] (1.5,0.5) circle (.1);
				\draw[fill=white] (2,0) circle (.1);
				\draw[fill=white] (3,0) circle (.1);
				\draw[fill=white] (3.5,0.5) circle (.1);
				\draw[fill=white] (3.5,1.5) circle (.1);
				\draw[fill=white] (4,1) circle (.1);
				\draw[fill=white] (5,1) circle (.1);

				\node[below] at (0,0) {\scriptsize $\ell_1$};
				\node[below] at (1,0) {\scriptsize $\ell_2$};
				\node[above] at (0.5,0.5) {\footnotesize 1};
				\node[above] at (1,1) {\footnotesize 2};
				\node[below] at (1.5,0.5) {\scriptsize $\ell_3$};
				\node[above] at (2,2) {\footnotesize 3};
				\node[below] at (2,0) {\scriptsize $\ell_4$};
				\node[below] at (3,0) {\scriptsize $\ell_5$};
				\node[above] at (2.5,0.5) {\footnotesize 4};
				\node[above] at (3,1) {\footnotesize 5};
				\node[below] at (3.5,0.5) {\scriptsize $\ell_6$};
				\node[above] at (3,3) {\footnotesize 6};
				\node[above] at (4,2) {\footnotesize 7};
				\node[below] at (3.5,1.5) {\scriptsize $\ell_7$};
				\node[above] at (4.5,1.5) {\footnotesize 8};
				\node[below] at (4,1) {\scriptsize $\ell_8$};
				\node[below] at (5,1) {\scriptsize $\ell_9$};

			\end{scope}
			
			\begin{scope}[xshift=4cm]
				
				\draw (0,0)--(3,3)--(4.5,1.5) (0.5,0)--(0.5,0.5)--(1,0) (4,2)--(3.5,1.5) (3,3)--(3,2) (1.5,0.5)--(2,1)--(2,2) (2.5,0.5)--(2,1) (2,2)--(2.5,1.5);
				\draw[fill=white] (0,0) circle (.1);
				\draw[fill=white] (0.5,0) circle (.1);
				\draw[fill=black] (0.5,0.5) circle (.1);
				\draw[fill=white] (1,0) circle (.1);
				\draw[fill=black] (2,2) circle (.1);
				\draw[fill=white] (1.5,0.5) circle (.1);
				\draw[fill=white] (2.5,0.5) circle (.1);
				\draw[fill=black] (2,1) circle (.1);
				\draw[fill=white] (2.5,1.5) circle (.1);
				\draw[fill=black] (3,3) circle (.1);
				\draw[fill=black] (4,2) circle (.1);
				\draw[fill=white] (3.5,1.5) circle (.1);
				\draw[fill=white] (4.5,1.5) circle (.1);
				\draw[fill=white] (3,2) circle (.1);
				
				\node[above] at (0.5,0.5) {\footnotesize 1,2};
				\node[above] at (2,2) {\footnotesize 3,5};
				\node[anchor=south east, inner sep=2] at (2,1) {\footnotesize 4};
				\node[above] at (3,3) {\footnotesize 6,7};
				\node[above] at (4,2) {\footnotesize 8};
				\node[below] at (0,0) {\scriptsize $\ell_1$};
				\node[below] at (0.5,0) {\scriptsize $\ell_2$};
				\node[below] at (1,0) {\scriptsize $\ell_3$};
				\node[below] at (1.5,0.5) {\scriptsize $\ell_4$};
				\node[below] at (2.5,0.5)  {\scriptsize $\ell_5$};			
				\node[below] at (2.5,1.5) {\scriptsize $\ell_6$};
				\node[below] at (3,2) {\scriptsize $\ell_7$};
				\node[below] at (3.5,1.5) {\scriptsize $\ell_8$};
				\node[below] at (4.5,1.5) {\scriptsize $\ell_9$};	
				
			\end{scope}

		\end{tikzpicture}
		\caption{Internal labelings associated to Figure \ref{fig:planetrees}.}\label{fig:preorders}
	\end{figure}
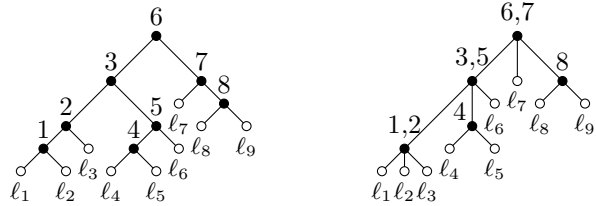
\end{ex}

The name associahedron historically comes from the interpretation of vertices as bracketings and flips as applications of associativity, see following remark.

\begin{rem}\label{rem:associativity}		
	The labeling on the leaves gives a canonical bijection between $\sT_{n}$ and the set of all possible ``bracketings'' on $\left(\ell_1 \ast \ell_2 \ast\dots \ast \ell_{n+1}\right),$ which induces a bijection between complete binary trees in $\sT_{n,n}$ and ``full bracketings'' on $\left(\ell_1 \ast \ell_2 \ast\dots \ast \ell_{n+1}\right).$
	Furthermore, flipping internal edges on complete binary trees correspond to applying ``associative law'' on full bracketings on $\left(\ell_1 \ast \ell_2 \ast\dots \ast \ell_{n+1}\right).$
	
\end{rem}

The following lemma states properties of the labels on the internal vertices of trees in $\sT_{n}.$ The proof of it is straightforward, and thus is omitted. 

\begin{lem}\label{lem:label}
	Let $\TT \in \sT_{n}.$ 
	\begin{enumerate}
		\item\label{itm:nonempty} Each internal vertex of $\TT$ is labeled by a nonempty set of numbers, and each number in $[n]$ appears exactly once. 
		Thus, if $\TT\in \sT_{n,n}$, each internal vertex has a unique label.
		
		\item \label{itm:label-substract}
		Let $\TT'$ be a subtree of $\TT$. 
		Then $\II_\TT(\TT')$ 
		is an integer interval, and it is empty if and only if $\TT'$ is a leave of $\TT$. 
		Furthermore, suppose $\TT'$ is rooted at an internal vertex of $\TT,$ thus $\II_\TT(\TT') = [a, b]$ for some integers $a\le b$. Then the labels of internal vertices of $\TT'$ when we treat it as a rooted tree by itself are obtained from subtracting $(a-1)$ from those when we treat it as a subtree of $\TT.$
		
		\item \label{itm:label-increase} Suppose $\TT_1$ and $\TT_2$ are two subtrees below an internal vertex of $\TT.$ If $\TT_1$ is to the left of $\TT_2,$ then any label appearing on an internal vertex of $\TT_1$ is smaller than any label appearing on an internal vertex of $\TT_2.$
		
		\item \label{itm:label-compatible} Suppose $\TT'$ is obtained from $\TT$ by contracting one internal edge $\{x,y\}$, and we call the new vertex $z$. Then the internal vertex labeling of $\TT'$ can be obtained from that of $\TT$ by labeling $z$ with the union of the labels for $x$ and $y$ and keeping labels for all the other internal vertices.

	\end{enumerate}
\end{lem}

We want to note that given a tree in $\sT_{n}$ or $\sT_{n,n}$, the labels we put on its leaves (and internal vertices) are uniquely determined. Hence, those labels do not carry extra information.

\subsection{From trees to cones}\label{subsec:tree2cone}
We finish this section by defining a cone for each of our trees. 

	Let $\TT \in \sT_{n}$. We define $\GG(\TT)$ to be the induced subtree of $\TT$ on its internal vertices, together with the labeling for internal vertices we described above. One sees that $\GG(\TT)$ is the Hasse diagram of a preorder on $[n],$ which we denote by $\preceq_\TT$. Recall that we have defined linear extensions and preorder cones in \S \ref{subsec:preorder}. For convenience and brevity, we denote
	\begin{equation}\label{eq:sigmaT}
		\LL[\TT]:=\LL[[n],\preceq_\TT]  \text{\quad and \quad}   \sigma(\TT):=\sigma_{\preceq_\TT}.
	\end{equation}
	
%
The following lemma follows immediately from definitions and Lemma \ref{lem:poset_facts}~\eqref{itm:faces}.
\begin{lem}\label{lem:treecompatible} 
	Let $\TT,\TT' \in \sT_n.$ Then the following statements are equivalent:
	\begin{enumerate}
		\item \label{itm:tree_contraction} $\TT'$ is a contraction of $\TT$ (i.e., $\TT \le_K \TT'$). 
		\item \label{itm:preord_contraction}The preorder $\preceq_{\TT'}$ is a contraction of $\preceq_{\TT}$.
		\item \label{itm:tree_face}The cone $\sigma(\TT')$ is a face of $\sigma(\TT)$.
	\end{enumerate}
	
\end{lem}


The next lemma gives a connection between permutations and complete binary trees. 

\begin{lem}\label{lem:unique}
	Let $n \in \ZZ_{\ge 0}.$ For every $\pi\in\fS_{n}$, there exists a \emph{unique} complete binary tree $\TT\in\sT_{n,n}$ such that 
	$\pi \in \LL[\TT].$
\end{lem}
The key idea of the proof is to construct a tree $\TT$ by inserting the permutation $\pi$ in a binary search tree. 
Such a construction was given as the map $\psi$ by Loday \cite[Page 2]{loday} and a proof can be found in Loday and Ronco's paper \cite[Section 2]{loday1998hopf}.

	\section{Realization of the associahedron}\label{sec:asso}
	
	In this section, we follow the method outlined in the introduction to construct Loday's associahedron as in \cite{loday}.
	In particular we explicitly describe its normal fan, a fan we called Loday fan.
	This fan can be seen as an example of a \emph{Cambrian} fan, studied by Reading and Speyer \cite{reading2009cambrian} and it is also an example of a \emph{permutreehedral} fan studied by Pilaud and Pons \cite{pilaud_pons}.
	We redo the construction here because particular details of it are relevant for our realization of the permuto-associahedron in Section \ref{sec:permasso}.
	
	\subsection{Vertices of Loday Associahedra} We begin by defining a set of points that are the candidates for vertices of the polytope.
	\begin{defn}\label{def:vector}
		Suppose $\balpha \in \R^{n}$ and $\TT \in \sT_{m,m}$ where $1 \le m \le n.$ Define
		\begin{equation}\label{eq:v_def}
			\setlength{\abovedisplayskip}{0pt}
			\setlength{\belowdisplayskip}{1pt}
			\val\left(\balpha, \TT\right) :=  \sum_{k=1}^{t} \alpha_k - \sum_{k=1}^{a} \alpha_k - \sum_{k=1}^{b} \alpha_k,
		\end{equation}
		where $t$, $a$, and $b$ are the number of total internal vertices in $\TT,$ the Left subtree of $\TT$, and the Right subtree of $\TT$, respectively. (Note that we have $t=a+b+1.$) 
		
		Given a (strictly increasing) sequence $\balpha \in \R^n$ and $\TT \in \sT_{n,n}$, we define 
		\[ 
		\setlength{\abovedisplayskip}{2pt}
		\setlength{\belowdisplayskip}{2pt}
		\v^\balpha_\TT := \sum_{i=1}^{n}\val\left(\balpha, \TT_{(i)}\right)\e_i = \left( \val\left(\balpha, \TT_{(1)}\right), \val\left(\balpha, \TT_{(2)}\right), \dots, \val\left(\balpha, \TT_{(n)}\right) \right),\]	
		where $\TT_{(i)}$ denotes the subtree of $\TT$ that is rooted at the internal vertex labeled by $i.$	
	\end{defn}
	
	The definition given above is a generalization of Loday: Note that if we choose $\balpha=(1, 2, \dots, n),$ then the right hand side of \eqref{eq:v_def} becomes $(a+1)(b+1)$. This recovers the vertex coordinates of the associahedron constructed by Loday \cite{loday}. Another generalization of Loday's coordinates was given by Masuda, Thomas, Tonks, and Vallete \cite{masuda2021diagonal} using weights on the leaves (rather than on internal vertices as we do).
	The reason we consider our generalization is twofold: On the one hand we want a polytope related to $\Perm(\balpha)$, see Corollary \ref{cor:remove} below. On the other hand, we will need some flexibility on $\balpha$ later when we construct the permuto-associahedron.
	
	We remark that some of the definitions and results in this section will be stated using the variable $n$ in  which case the readers can assume $n=d+1$ (where $d$ is the dimension of the constructed associahedron).
	The reason to do this is that in later sections we used these ideas with $n=d$. 
	
	\begin{ex}
		Let $\balpha = (2,5,6,14,17,21,22,24)$ and $\TT$ be the complete binary tree on the left of Figure \ref{fig:preorders}. Then $\TT_{(3)}$ is the Left subtree of the root of $\TT$. It has $5$ internal vertices, and both of its Left subtree and Right subtree have $2$ internal vertices. Hence,
		\[ 
		\setlength{\abovedisplayskip}{2pt}
		\setlength{\belowdisplayskip}{2pt}
		\val\left( \balpha, \TT_{(3)} \right) =  \sum_{k=1}^{5} \alpha_k - \sum_{k=1}^{2} \alpha_k - \sum_{k=1}^{2} \alpha_k = 44-7-7 = 30,\]
		which gives the $3$rd entry of $\v^\balpha_\TT.$ We can compute the other entries similarly, and get $v_\TT^\balpha=(2,5,30,2,5,60,5,2)$. 
		
		We can check that the sum of coordinates of $v_\TT^\balpha$ is $111$ which is the same as that of $\balpha.$ Hence, $v_\TT^\balpha \in U_7^\balpha.$ This is true in general, as we state in the lemma below.
	\end{ex}

	\begin{lem}\label{lem:sumval}
		Suppose $\balpha \in \R^{n}$ and $\TT \in \sT_{n,n}$. 
		Let $\TT'$ be a subtree of $\TT$ with $t$ internal vertices, and suppose (by Lemma \ref{lem:label}~\eqref{itm:label-substract}) that $\II_\TT(\TT') = [c, c+t-1]$ for some integer $c.$ Then 
		\begin{equation}
			\setlength{\abovedisplayskip}{0pt}
			\setlength{\belowdisplayskip}{0pt}
			\sum_{i=c}^{c+t-1} \val(\balpha, \TT_{(i)}) = \sum_{k=1}^{t} \alpha_k. 
			\label{eqn:sumval}
		\end{equation}
		In particular, if $\TT' = \TT,$ we obtain that sum of the coordinates of $\v^\balpha_\TT$ is $\ds \sum_{k=1}^{n} \alpha_k$. Hence, $\v^\balpha_\TT$ is a point in $U_{n-1}^\balpha$. (Recall the affine space $U_d^\balpha$ is defined in \eqref{eq:defnU}.)
	\end{lem}
	
	\begin{rem}\label{rem:sumval}
		If we let $I = \II_{\TT}(\TT'),$ then \eqref{eqn:sumval} can be rewritten as 
		$\langle \e_{I}, \v_\TT^\balpha \rangle  = \sum_{k=1}^{|I|} \alpha_k.$
	\end{rem}
	
	\begin{proof}[Proof of Lemma \ref{lem:sumval}] We prove by induction on $t,$ the number of internal vertices in $\TT'.$ If $t=0,$ we have that $\II_\TT(\TT') = \emptyset$ or $[c, c-1]$. So \eqref{eqn:sumval} clearly holds.
		
		Suppose $t \ge 1$ and \eqref{eqn:sumval} holds for any $\TT'$ with less than $t$ internal vertices. Now we consider $\TT'$ has $t$ internal vertices. Let $\TT_L$ and $\TT_R$ be the Left and Right subtrees of $\TT'$ respectively, and suppose $a$ and $b$ are the number of internal vertices of $\TT_L$ and $\TT_R.$ (Note we have $t=a+b+1.$) Then by Lemma \ref{lem:label} (specifically numerals \eqref{itm:nonempty},\eqref{itm:label-substract}, and\eqref{itm:label-increase}), we must have that $\TT' = \TT_{(c+a)}$ is the subtree of $\TT$ rooted at the internal vertex with label $c+a,$ and  
		\[ \II_\TT(\TT_L) = [c, c+a-1] \quad \text{and} \quad \II_\TT(\TT_R) = [c+a+1, c+a+b]=[c+a+1,c+t-1].\]
		Since both $\TT_L$ and $\TT_R$ have less than $t$ internal vertices, by induction hypothesis,
		\[  
		\sum_{i=c}^{c+a-1} \val(\balpha, \TT_{(i)}) = \sum_{k=1}^{a} \alpha_k \quad \text{and} \quad  \sum_{i=c+a+1}^{c+t-1} \val(\balpha, \TT_{(i)}) = \sum_{k=1}^{b} \alpha_k. \]
		Summing these two equations and that 
		$\val(\balpha, \TT_{(c+a)}) = \val(\balpha, \TT')=\sum_{k=1}^{t} \alpha_k - \sum_{k=1}^{a} \alpha_k - \sum_{k=1}^{b} \alpha_k,$ we obtain \eqref{eqn:sumval}.  
	\end{proof}

	\begin{defn} Let $\balpha\in\R^{d+1}$ be a strictly increasing sequence. We define the following polytope 
		\begin{equation}
			\setlength{\abovedisplayskip}{0pt}
			\loday(\balpha) := \conv\{\v^\balpha_\TT \ : \ \TT\in\sT_{d+1,d+1}\}\subseteq U_d^\balpha.
		\end{equation}
	\end{defn}		
	
	The polytope $\loday(\balpha)$ is called the \emph{Loday associahedron} in reference of Jean-Louis Loday who first study it in the case $\balpha = (1, 2, \dots, d+1)$.
	The following is the main theorem of this section.

	\begin{thm}\label{thm:loday}
		Let $\balpha\in\R^{d+1}$ be a strictly increasing sequence.
		Then the face poset of the Loday associahedron $\loday(\balpha)$ is $\cK_d$. (Recall that $\cK_d$ is defined in Definition \ref{def:poset}.) Moreover, $\loday(\balpha)$ is a $d$-associahedron, and is a generalized permutohedron as well.
	\end{thm}

	\subsection{Normal fan of Loday associahedra}\label{subsec:lodayfan}
	
	In this part we construct a conic dissection of $W_d$ and then apply Lemma \ref{lem:main_tool} to show that the conic dissection induces the normal fan of $\loday(\balpha).$
	
	Recall that in 
	\S \ref{subsec:tree2cone}
	we define a cone $\sigma(\TT)\subset W_{n-1}$ for each tree $\TT\in\sT_n$.
	The following result gives us the conic dissection we need.
	
	\begin{lem}\label{lem:treefan}
		Each cone $\sigma(\TT)$ for $\TT \in \sT_{n,n}$ is a union of braid cones. Furthermore the collection of cones $\MM_{n-1}:=\{\sigma(\TT): \TT \in \sT_{n,n}\}$ is a simplicial conic dissection of $W_{n-1}$. 	
	\end{lem}
	
	\begin{proof} 
		Note that for $\TT \in \sT_{n,n}$, the preorder $\preceq_{\TT}$ on $[n]$ is a partial order. Hence, the first statement follows from Lemma \ref{lem:poset_facts}~\eqref{itm:union}.  
		Then the second statement follows from part \eqref{itm:simplicial} of Lemma \ref{lem:poset_facts}, and Lemmas \ref{lem:unique} and \ref{lem:braid_dissect}. 
	\end{proof}

	We need one preliminary lemma before establishing the condition we need on $\MM_{n-1}$ and $\{ v_\TT^\balpha\}$ in order to apply Lemma \ref{lem:main_tool}.

	\begin{lem}\label{lem:comparedotprod}
		Let $\TT \in \sT_{n,n}$ and $\w \in \sigma^\circ_{\TT}$. For every $\TT' \in \sT_{n,n}$ that is different from $\TT,$ there exists $\TT''$ obtained from $\TT'$ by a flip of an internal edge such that
		$\langle \w,\v^\balpha_{\TT'} \rangle < \langle\w,\v^\balpha_{\TT''} \rangle.$
	\end{lem}
	
	\begin{proof}
		Let $\w\in \sigma^\circ(\TT)$. Lemma \ref{lem:treefan} states that $\MM_{n-1}$ is a conic dissection of $W_{n-1}.$ Hence, given that $\TT' \neq \TT,$ we must have that $\w \not\in \sigma(\TT').$ Using the description of a preorder cone given in Lemma \ref{lem:poset_facts}~\eqref{itm:interior}, we conclude that there exists a covering relation $j \cover_{\TT'} i$ in $\preceq_{\TT'}$ such that $x_i < x_j.$ 
		Let $e$ be the internal edge of $\TT'$ corresponding to this covering relation, and let $\TT''$ be the tree obtained from $\TT'$ by flipping $e.$ 
		It can be checked (see \cite[Lemma 4.8]{revisited-arxiv-v1}) that $\v^\alpha_{\TT'}-\v^{\alpha}_{\TT''}=\lambda (\e_i-\e_j)$ for some $\lambda>0$. Thus,
		\begin{equation*}
			\langle \w,\v^\balpha_{\TT'}- \v^\balpha_{\TT''} \rangle = \langle\w,\lambda (\e_i-\e_j) \rangle = \lambda (x_i - x_j) < 0. \qedhere
			\label{eq:computedotprod}
		\end{equation*}
	\end{proof}
	
	\begin{cor}\label{cor:aux}
		Let $\TT, \TT' \in \sT_{n,n}$ be two complete binary trees. Then for every $\w\in \sigma^\circ(\TT)$, we have $\langle \w,\v^\balpha_\TT\rangle \ge \langle\w,\v^\balpha_{\TT'} \rangle$, where the equality holds if and only if $\TT = \TT'$.
	\end{cor}
	
	\begin{proof}
		It follows from Lemma \ref{lem:comparedotprod} that we can construct a sequence of complete binary trees in $\sT_{n,n}:$
		$\TT_0 = \TT', \TT_1, \TT_2, \dots,$
		such that for each $i,$ 
		$\langle \w,\v^\balpha_{\TT_i} \rangle < \langle\w,\v^\balpha_{\TT_{i+1}} \rangle.$
		The construction can continue as long as $\TT_i \neq \TT.$ Since $\sT_{n,n}$ is finite, this procedure must ends with a tree $\TT_k = \TT.$ Then the conclusion follows. 
	\end{proof}
	
	The following proposition is the key result of this section, characterizing the vertex set and the normal fan of the Loday associahedron. It also provides the main ingredients we need for proving Theorem \ref{thm:loday}. 
	
	\begin{prop}\label{prop:loday}
		Let $\balpha\in\R^{d+1}$ be a strictly increasing sequence. 
		\begin{enumerate}
			\item \label{itm:lodayfulldim} The Loday associahedron $\loday(\balpha)$ is full-dimensional in $U_d^\balpha$ and its vertex set is $\{v^\balpha_\TT : \TT\in\sT_{d+1,d+1}\}$.
			\item For each $\TT \in \sT_{d+1,d+1}$, we have $\sigma(\TT)=\ncone(v^\alpha_\TT,\loday(\balpha)).$
			
			\item\label{itm:lodaynormal} The normal fan of $\loday(\balpha)$ is $\lodayfan_d:=\{\sigma(\TT):\TT\in\sT_{d+1}\}.$
			Hence, $\lodayfan_d$ is a complete projective fan in $W_d.$
		\end{enumerate}
	\end{prop}
	
	\begin{proof}
		It follows from Lemma \ref{lem:treefan} and Corollary \ref{cor:aux} that the set of cones $\MM_d=\{\sigma(\TT) : \TT\in\sT_{d+1,d+1}\}$ in $W_d$ and the set of points $\{ v^\balpha_\TT : \TT\in\sT_{d+1,d+1}\}$ in $U_d^\balpha$ satisfy the hypothesis of Lemma \ref{lem:main_tool}. Hence, we conclude that (1) and (2) are true, and that the $\MM_d$ induces the normal fan of $\loday(\balpha).$   
		Therefore, it is left to show that $\lodayfan_d$ is induced by $\MM_d$. However, this easily follows from Lemma \ref{lem:poset_facts}~\eqref{itm:faces}, Lemma \ref{lem:treecompatible}, and the fact that $\sT_{d+1}$ consists of all contractions of $\sT_{d+1,d+1}$.
	\end{proof}
	
	Because of the results given above, we call \begin{equation}\label{eq:lodayfan}
		\lodayfan_d:=\{\sigma(\TT):\TT\in\sT_{d+1}\}.
	\end{equation}  the \emph{Loday fan}. 
	The following lemma gives a characterization of the Loday fan.
	\begin{lem}\label{lem:lodayfan} The Loday fan $\lodayfan_d$ has the following properties:
		\begin{enumerate}
			\item\label{itm:lodayfandim} 
			The set $\{ \sigma(\TT): \TT \in \sT_{d+1,k}\}$ consists of all the $(k-1)$-dimensional cones in $\lodayfan_d.$
			\item\label{itm:lodayface} For any $\TT, \TT' \in \sT_{d+1},$ we have that $\sigma_{\TT'}$ is a face of $\sigma_{\TT}$ if and only if $\TT'$ is obtained from $\TT$ by contracting internal edges of $\TT.$ 
			\item\label{itm:lodayfaceposet} The face poset of $\lodayfan_d$ is isomorphic to the poset dual to  $\cK_d\setminus\hat{0}$. 
			\item\label{itm:lodaysimpli} The Loday fan $\lodayfan_d$ is a simplicial fan. 
			\item\label{itm:lodaycoarsen} The Loday fan $\lodayfan_d$ is a coarsening of the braid fan $\Br_d.$
			
		\end{enumerate}
	\end{lem}
	
	\begin{proof}
		One sees that \eqref{itm:lodayfandim}, \eqref{itm:lodayface}, and \eqref{itm:lodaysimpli} follow from Lemma \ref{lem:poset_facts} and the definition of $\sigma(\TT)$. 
		Item \eqref{itm:lodayfaceposet} follows from \eqref{itm:lodayface} and the definition of $\cK_d$ in Definition \ref{def:poset}.
		By Lemma \ref{lem:treefan}, each cone in $\MM_d$ is a union of braid cones. Since $\lodayfan_d$ is induced by $\MM_d,$ it must coarsen $\Br_d.$
	\end{proof}

	We are now ready to prove the main theorem of this section.	
	\begin{proof}[Proof of Theorem \ref{thm:loday}]
		
		By Proposition \ref{prop:loday}~\eqref{itm:lodaynormal} and Lemmas \ref{lem:lodayfan}~\eqref{itm:lodayfaceposet} and \ref{lem:anti}, we have that the face poset of $\loday(\balpha)$ is $\cK_d$.
		Moreover, by Lemma \ref{lem:FLdefineAsso}, we conclude that $\loday(\balpha)$ is a $d$-associahedron.
		Finally, it follows from Proposition \ref{prop:loday}~\eqref{itm:lodaynormal} and Lemma \ref{lem:lodayfan}~\eqref{itm:lodaycoarsen}  that $\loday(\balpha)$ is a generalized permutohedron.
	\end{proof}
	
	\subsection{Inequality description for Loday associahedra}
	
	It follows from Proposition \ref{prop:loday} that we can apply Lemma \ref{lem:det-ineq} to find an inequality description for the Loday associahedron $\loday(\balpha)$.
	Note that by Lemma \ref{lem:lodayfan}~\eqref{itm:lodayfandim}, the set $\{\sigma(\TT) : \TT \in \sT_{d+1,2}\}$ consists of all one dimensional cones of $\lodayfan_d$. 
	Thus, our first step is to choose a generator for each of the cones in this set.

	\begin{defn}	
		Let $\TT \in \sT_{d+1,2}.$ Since $\TT$ has two internal vertices, it has a unique non-root internal vertex, say $v.$ Then by Lemma \ref{lem:label}~\eqref{itm:label-substract}, the labels of $v$ is a nonempty integer interval, denoted by $\Itv(\TT)$, and called the \emph{defining interval} of $\TT.$
	\end{defn}
	
	The following proposition is straightforward.
	\begin{prop}\label{prop:rays_asso} 	
		The map $\TT \mapsto \Itv(\TT)$ gives a one-to-one correspondence between the set $\sT_{d+1,2}$ and the set 
		\begin{equation}
			\setlength{\abovedisplayskip}{1pt}
			\sI_d :=\{ I \ : \ I \text{ is an integer interval and } \emptyset \neq I \subsetneq [d+1]\}
			\label{eq:defnsI}
		\end{equation}
		
		Moreover, for each $\TT\in\sT_{d+1,2}$, if let $I = \Itv(\TT)$, then the one-dimensional cone $\sigma_{\TT} \in\lodayfan_d$ is generated by $\e_{\bar{I}}$, where $\bar{I} := [d+1] \setminus I$ is the complement of $I$.
		
	\end{prop}

	\begin{thm}\label{thm:ineqs} Let $\balpha\in\R^{d+1}$ be a strictly increasing sequence.
		Then we have the following minimal inequality description for $\loday(\balpha):$ 
		\begin{equation}\label{eq:inequalities}
			\setlength{\abovedisplayskip}{1pt}
			\setlength{\belowdisplayskip}{0pt}
			\loday(\balpha) = \left\{\x \in U^\balpha_d \ : \ \langle \e_{\bar{I}}, \x \rangle \le \sum_{i=|I|+1}^{d+1} \alpha_i,\quad \forall I \in \sI_d \right\}.
		\end{equation}
	\end{thm}
	We remark that the inequality $\langle \e_{\bar{I}}, \x \rangle \le \sum_{i=|I|+1}^{d+1} \alpha_i$ is equivalent to $\langle \e_{I}, \x \rangle \ge \sum_{i=|1}^{|I|} \alpha_i$ since the sum of all coordinates is fixed.
	We present this way so that it is consistent with using outer normal vectors for the normal fan.
	
	\begin{proof}
		Applying the first part of Lemma \ref{lem:det-ineq} together with Proposition \ref{prop:rays_asso}, one sees that it is left to show for any $\TT' \in \sT_{d+1,2}$, if we let $I = \Itv(\TT'),$ then 
		$\ds \max_{\TT\in\sT_{d+1,d+1}}\langle \e_{\bar{I}},\v^\balpha_\TT\rangle = \sum_{i=|I|+1}^{d+1} \alpha_i.$
		We choose $\TT \in \sT_{d+1,d+1}$ such that $\TT \le_K \TT'$ in $\cK_d,$ equivalently, $\TT'$ can be obtained from $\TT$ by contracting internal edges. Thus, by Lemma \ref{lem:lodayfan}~\eqref{itm:lodayface}, we have $\sigma_{\TT'}$ is a face of $\sigma_{\TT},$ where the latter is the normal cone of $\loday(\balpha)$ at $\v_\TT^\balpha.$ Therefore, by the second part of Lemma \ref{lem:det-ineq}, we just need to show 
		$\langle \e_{\bar{I}},\v^\balpha_\TT\rangle = \sum_{i=|I|+1}^{d+1} \alpha_i.$

	By the definition of $I$ and 
	how the internal vertices of trees in $\sT_{n}$ are labeled, 
	one sees that
	there exists a subtree $\TT_0$ of $\TT$ such that $I = \II_\TT(\TT_0).$
	Therefore, by Lemma \ref{lem:sumval} and Remark \ref{rem:sumval}, we have
	$\langle \e_{I}, \v_\TT^\balpha \rangle  = \sum_{k=1}^{|I|} \alpha_k,$ and $\langle \1, \v_\TT^\balpha \rangle  = \sum_{k=1}^{d+1} \alpha_k.$
	Thus, the conclusion follows.
\end{proof}

Following terminology from V.Pilaud \cite{pilaud},
if a polytope is defined by a subset of a system of linear inequalities that defines a permutohedron $\Perm(\balpha)$, then we call it an $\balpha$-\emph{removohedron}.

\begin{cor}\label{cor:remove}
	Let $\balpha \in\R^{d+1}$ be a strictly increasing sequence.
	Then the associahedron $\loday(\balpha)$ is an $\balpha$-removohedron. 
\end{cor}

\begin{proof}
	Comparing the inequality description obtained in Theorem \ref{thm:ineqs} with the inequality description for $\Perm(\balpha)$ given in Equation \eqref{eq:perm_ineqs}, we see the conclusion follows.
\end{proof}

\begin{cor}\label{cor:minkowski}
	Let $\balpha=(\alpha_1,\dots,\alpha_{d+1}) \in\R^{d+1}$
	be a strictly increasing sequence and $\bdelta=(\delta_1,\dots,\delta_d)$, where $\delta_k=\alpha_{k}-\alpha_{k-1}$ for $k=1,2,\ldots,d+1$, the vector of first differences (and setting $\alpha_0=0$). We have the following description using Mikowski sums
	\begin{equation}\label{eq:minkowski}
		\loday(\balpha) = \delta_{d+1} \Delta_{[d+1]} + \sum_{I\in \sI_d}  \delta_{|I|}\Delta_I,
		\setlength{\belowdisplayskip}{0pt}
	\end{equation}
	where $\Delta_I:=\conv\{\e_i: i\in I\}\subset \R^{d+1}$. 
\end{cor}

The corollary can be proved by applying \cite[Proposition 6.3]{bible} to give an inequality description for the Minkowski sum of simplices in the right hand side of Equation \eqref{eq:minkowski} and verifying that it is the same as \eqref{eq:inequalities}.

\begin{rem}
	If one takes Equation \eqref{eq:minkowski} as the definition for $\loday(\balpha)$, most of the constructions and results can be derived using results of \cite[Section 7]{bible}.
	In particular, \cite[Theorem 7.9]{bible} provides explicit coordinates for vertices of nestohedra and \cite[Theorem 7.10]{bible} describes the normal cone at each vertex.
	We write this section in the order as presented because this is how we come up with our construction and we want to use it as an example to demonstrate our methods of using Lemmas \ref{lem:main_tool} and \ref{lem:det-ineq}.
\end{rem}

In \cite{ceballos_santos_ziegler} the authors called any polytope of the form $\sum_{I} a_I\Delta_I$, where $a_I>0$ summed over all integer intervals, a \emph{Postnikov associahedron}. 
This family of polytopes contains Loday associahedra defined in this paper, but is strictly larger, since in Equation \eqref{eq:minkowski} the scalar factors are the same for integer intervals of the same size.
Furthermore, Corollary \ref{cor:remove} is not valid for Postnikov associahedra.

\section{Kapranov poset, partition labeled trees and their associated cones}
\label{sec:orderedtrees}

In this section we define the Kapranov poset in terms of trees labeled by partitions.
Then we associate cones to these trees by using our ideas from \cite{def-cone}.
Finally, we prove some basic properties of these cones which are fundamental for our construction of the permuto-associahedron in the next section.

\subsection{The Kapranov poset}
In this part, we will introduce the poset defined by Kapranov as a hybrid of the face poset of a permutohedron and the face poset of an associahedron, and define the permuto-associahedron abstractly. 

Recall the face poset of a usual $d$-permutohedron is the poset $\OO_{d+1}$ on ordered set partitions on $[d+1]$ (defined in Definition \ref{defn:osp}) and the face poset of a $d$-associahedron is the poset $\cK_d$ on strictly branching trees in $\sT_{d+1}$ (defined in Definition \ref{def:poset}). 
Below we construct Kapranov's poset as a poset on pairs of ordered set partitions and these trees.

\begin{defn}\label{defn:kap}
	Let $\sT_{\leq d}:=\bigcup_{i=0}^{d}\sT_i$ be the set of strictly branching trees with at most $d+1$ leaves.
	A \emph{($[d+1]$-)partition labeled tree} consists of a pair $(\cS,\TT)\in \OO_{d+1}\times\sT_{\leq d}$, where the tree $\TT$ has $k\leq d+1$ leaves and the partition $\cS$ has $k$ blocks $S_1,\ldots,S_k$ that we use to label the leaves of $\TT$ from left to right. The set of all $[d+1]$-partition labeled trees is denoted $\sP_d$. 
	
	For any $(\cS_1,\TT_1), (\cS_2,\TT_2)$ in $\sP_d$, we define $(\cS_1,\TT_1)\precdot_{KP}(\cS_2,\TT_2)$ if one of the following two conditions hold:
	\begin{enumerate}[label=\roman*.]
		\item \namedlabel{condition1} The tree $\TT_2$ is obtained from $\TT_1$ by contracting a single internal edge of $\TT_1$, and $\cS_1=\cS_2$.
		\item \namedlabel{condition2} There exists a minimal\footnote{An internal vertex is a \emph{minimal} internal vertex if all of its children are leaves.} internal vertex $p$ of $\TT_1$ such that $\TT_2$ is obtained from $\TT_1$ by contracting all edges between $p$ and its children in $\TT_1$ while labeling this new leaf with the union of the labels of the children of $p$ in $\TT_1$.
	\end{enumerate}
	
	This relation $\precdot_{KP}$ extends to a partial order $\le_{KP}$ on $\sP_d$ where $\precdot_{KP}$ is the covering relation. We also add a minimum element $\{ \hat{0} \}$ by declaring $\hat{0} \le_{KP} (\cS, \TT)$ for any $(\cS, \TT) \in \sP_d.$ Finally, we denote the poset $(\sP_d \cup \{\hat{0}\}, \le_{KP})$ by $\cK\Pi_d$ and call it the \emph{Kapranov poset}.
\end{defn}

\begin{ex}\label{ex:compatible}
	
	In Figure \ref{fig:permutree1} we present two elements of $\sP_8$ with partition labelings\\ $(3,7,4,8,9,6,1,5,2)$ and $(374,89,6,1,5,2)$ respectively.
	The one on the left is smaller than the one on the right in the Kapranov poset.

	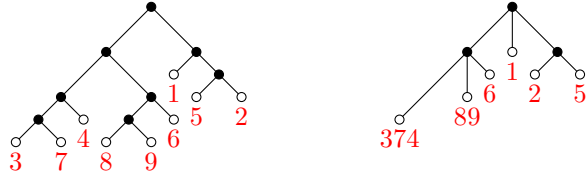
\begin{figure}[ht]
		\begin{tikzpicture}[scale=0.6]
			\draw (0,0)--(3,3)--(4.5,1.5)--(5,1) (1,1)--(1.5,0.5) (0.5,0.5)--(1,0) (2,2)--(3,1)--(2,0)
			(3,1)--(3.5,0.5) (2.5,0.5)--(3,0) (4,2)--(3.5,1.5) (4.5,1.5)--(4,1);
			
			\draw[fill=black] (0.5,0.5) circle (.1);
			\draw[fill=black] (1,1) circle (.1);
			\draw[fill=black] (2,2) circle (.1);
			\draw[fill=black] (2.5,0.5) circle (.1);
			\draw[fill=black] (3,1) circle (.1);
			\draw[fill=black] (3,3) circle (.1);
			\draw[fill=black] (4,2) circle (.1);
			\draw[fill=black] (4.5,1.5) circle (.1);
			\draw[fill=white] (0,0) circle (.1);
			\draw[fill=white] (1,0) circle (.1);
			\draw[fill=white] (1.5,0.5) circle (.1);
			\draw[fill=white] (2,0) circle (.1);
			\draw[fill=white] (3,0) circle (.1);
			\draw[fill=white] (3.5,0.5) circle (.1);
			\draw[fill=white] (3.5,1.5) circle (.1);
			\draw[fill=white] (4,1) circle (.1);
			\draw[fill=white] (5,1) circle (.1);

			\node[below] at (0,0) {\footnotesize \textcolor{red}{3}};
			\node[below] at (1,0) {\footnotesize \textcolor{red}{7}};
			\node[below] at (1.5,0.5) {\footnotesize \textcolor{red}{4}};
			\node[below] at (2,0) {\footnotesize \textcolor{red}{8}};
			\node[below] at (3,0) {\footnotesize \textcolor{red}{9}};
			\node[below] at (3.5,0.5) {\footnotesize \textcolor{red}{6}};
			\node[below] at (3.5,1.5) {\footnotesize \textcolor{red}{1}};
			\node[below] at (4,1) {\footnotesize \textcolor{red}{5}};
			\node[below] at (5,1) {\footnotesize \textcolor{red}{2}};

			\begin{scope}[xshift=8cm]
				\draw (0.5,0.5)--(3,3)--(4.5,1.5) (4,2)--(3.5,1.5) (3,3)--(3,2) (2,1)--(2,2) (2,2)--(2.5,1.5);
				
				\draw[fill=white] (0.5,0.5) circle (.1);
				
				\draw[fill=black] (2,2) circle (.1);
				\draw[fill=white] (2,1) circle (.1);
				\draw[fill=white] (2.5,1.5) circle (.1);
				\draw[fill=black] (3,3) circle (.1);
				\draw[fill=black] (4,2) circle (.1);
				\draw[fill=white] (3.5,1.5) circle (.1);
				\draw[fill=white] (4.5,1.5) circle (.1);
				\draw[fill=white] (3,2) circle (.1);
				
				\node[below] at (0.5,0.5) {\footnotesize $\textcolor{red}{374}$};
				\node[below] at (2,1) {\footnotesize $\textcolor{red}{89}$};

				\node[below] at (2.5,1.5) {\footnotesize $\textcolor{red}{6}$};
				\node[below] at (3,2) {\footnotesize $\textcolor{red}{1}$};
				\node[below] at (3.5,1.5) {\footnotesize $\textcolor{red}{2}$};
				\node[below] at (4.5,1.5) {\footnotesize $\textcolor{red}{5}$};

			\end{scope}
		\end{tikzpicture}
		\caption{Two $[9]$-partition labeled trees}
		\label{fig:permutree1}
	\end{figure}
\end{ex}

The next lemma collects useful facts about the Kapranov poset.
Some are proven in \cite{kapranov} in a different language and some are straightforward, so we omit proofs.

\begin{lem}\label{lem:kp_poset} The Kapranov poset $\cK\Pi_d$ is graded of rank $d+1.$ Furtheremore, the following facts hold for its elements:
	\begin{enumerate}
		\item The rank of a $[d+1]$-partition labeled tree $(\cS,\TT)$ is $d-i(\TT)+1$, where $i(\TT)$ is the number of internal vertices of the tree. 
		\item It has a unique maximum element $(\cS_0,\TT_0)$ where $\cS_0$ is the trivial partition $([d+1])$, and
		$\TT_0$ is the only rooted tree on a single vertex which is the unique element in $\sT_0$. 
		\item\label{itm:rank1} The elements of rank $1$ are in bijection with $\fS_{d+1} \times \sT_{d,d}$. More precisely, each $(\pi,\TT) \in \fS_{d+1} \times \sT_{d,d}$ defines a rank-$1$ element $(\cS(\pi),\TT)$ of $\cK\Pi_d$, and every rank-$1$ element arises this way. (Recall that $\cS(\pi)$ is defined in \eqref{equ:cSpi}.) 
		\item\label{itm:rankd}The elements of rank $d$ are in bijection with $\overline{\sO_{d+1}}: = \bigcup_{k=2}^{d+1} \sO_{d+1,k}$, the set of all non-trivial ordered partitions of $[d+1]$. More precisely, each $\cS \in {\sO_{d+1,k}}$ (for some $2 \le k \le d+1$) defines a rank-$d$ element $(\cS, \TT_k)$ of $\cK\Pi_d,$ where $\TT_k$ is the unique tree with one internal vertex and $k$ leaves, and every rank-$d$ element of $\cK\Pi_d$ arises this way (for some $k$).
	\end{enumerate}
\end{lem}

\begin{defn}\label{def:permasso}
	A \emph{$d$-permuto-associahedron} is a $d$-dimensional polytope whose face poset is isomorphic to $\cK\Pi_d$.
\end{defn}

\subsection{Nested combinatorics}\label{sec:nested} In this part, we will review results on nested permutohedra and nested braid fans obtained in \cite{def-cone}. 
Recall that $\{\e_1,\dots,\e_{d+1}\}$ is the standard basis for $\mathbb{R}^{d+1}$. For any permutation $\pi \in \fS_{d+1}$ and integer $i\in[d]$, we define the vectors \[\ff^\pi_i := \e_{\pi^{-1}(i+1)} - \e_{\pi^{-1}(i)},\]  and the linear transformations $\DD^\pi: \R^{d+1} \to \R^{d}$ as
\[ \DD^\pi \w := ( \DD^\pi_1 \w, \DD^\pi_2 \w, \dots, \DD^\pi_d \w),\]
where the $i$th coordinate is
\[ 
\setlength{\abovedisplayskip}{0pt}
\DD^\pi_i \w := \langle \w, \ff^\pi_i \rangle = w_{\pi^{-1}(i+1)} - w_{\pi^{-1}(i)}.\]
If $\pi$ is the identity permutation, we may omit $\pi$ and write $\DD \w$ instead.

It is easy to verify that $\DD^\pi(\w) = \DD^\pi(\w + k \1)$ for any $k \in \R.$ Hence, we may consider the domain of the map $\DD^\pi$ is $W_d$, and $\DD^\pi$ is a map from $W_d$ to $\R^{d}$. 

We note that with the above definition, the braid cone can be expressed as
\[ \sigma(\pi) = \{ \w \in W_d \ : \ \DD^\pi \w \ge 0 \}.\]

\begin{ex}\label{ex:Dpi} Let $\pi = 791386245$. Then $\pi^{-1} = 374896152.$ Thus $\w \in \sigma(\pi)$ means that 
	\[
	\begin{array}{cccc}
		\DD^\pi_1\w=w_7-w_3 \ge 0, &\DD^\pi_3\w=w_8-w_4 \ge 0, &\DD^\pi_5\w=w_6-w_9 \ge0, &\DD^\pi_7\w=w_2-w_1 \ge0,\\
		\DD^\pi_2\w=w_4-x_7\ge0, &\DD^\pi_4\w=w_9-w_8\ge0, &\DD^\pi_6\w=w_1-w_6 \ge0,  &\DD^\pi_8\w=w_5-w_2 \ge0,
	\end{array}
	\]
	or equivalently,
	\[ 
	\setlength{\abovedisplayskip}{0pt}
	w_3 \le w_7 \le w_4 \le w_8 \le w_9 \le w_6 \le w_1 \le w_5 \le w_2.\] 
\end{ex}

For convenience, for $\w \in \sigma(\pi),$ we often let $u_i = w_{\pi^{-1}(i)}$ for each $i$, which allows us to express $\w$ as
\begin{equation}\label{equ:wexp0}
	\setlength{\abovedisplayskip}{0pt}
	\w = \sum_{i=1}^{d+1} u_{i}\e_{\pi^{-1}(i)}.
\end{equation}
Then one sees that $\DD^\pi \w = \DD \u$ (where $\u=(u_1,\dots, u_{d+1})$).

Given strictly increasing sequences $\balpha = (\alpha_1, \alpha_2, \dots, \alpha_{d+1}) \in \R^{d+1}$ and $\bbeta = (\beta_1, \beta_2, \dots, \beta_{d}) \in \R^{d}$,
for any $(\pi, \tau) \in \fS_{d+1} \times \fS_d$, we define
\begin{equation}\label{equ:defnv}
	\v_{\pi,\tau}^{(\balpha,\bbeta)} := \sum_{i=1}^{d+1} \alpha_i \e_{\pi^{-1}(i)} + \sum_{i=1}^d \beta_i \ff_{\tau^{-1}(i)}^\pi\in U^\balpha_d.
\end{equation}

We call $(\balpha,\bbeta)$ an \emph{appropriate pair} (of strictly increasing sequences) if for any $(\pi,\tau),$ when we write $\v_{\pi,\tau}^{(\balpha,\bbeta)}=\sum \gamma_i \e_{\pi^{-1}(i)}$ we have $\gamma_1<\gamma_2<\dots<\gamma_{d+1}$. It is not hard to see, by scaling the vector $\balpha$ if necessary, that appropriate pairs exist.
Then for any appropriate pair $(\balpha,\bbeta)$, we define the \emph{usual nested permutohedron} (associated with the pair $(\balpha, \bbeta)$) to be
\begin{equation}
	\Perm(\balpha, \bbeta) := \conv\left(v_{\pi,\tau}^{(\balpha,\bbeta)}~:~(\pi,\tau) \in \fS_{d+1}\times \fS_d\right).
	\label{equ:defnusual2}
\end{equation}

A \emph{generalized nested permtuhohedron} is a deformation of a usual nested permutohedron. 
We need the following definition before defining \emph{nested braid cones}.	
\begin{defn}\label{def:check}
	Let $\tau \in \fS_d.$ We define $\check{\sigma}(\tau):=\{\x\in \R^{d} \ :\ x_{\tau^{-1}(1)}\le x_{\tau^{-1}(2)} \le \dots \le x_{\tau^{-1}(d)}\}$.
\end{defn}

Clearly $\check{\sigma}(\tau)$ maps to $\sigma(\tau)$ under the quotient map $\R^{d}\rightarrow W_{d-1}$.	
Notice that $\check{\sigma}(\tau)$ is not pointed as it contains the line spanned by $\1$.
Similar to Lemma \ref{lem:braid_dissect}, the collection of cones $\{\check{\sigma}(\tau): \pi\in \fS_{d}\}$ forms a conic dissection of $\R^{d}.$

\begin{defn}\label{defn:NBF}
	For each $(\pi,\tau)\in \fS_{d+1}\times \fS_{d}$, we define the \emph{nested braid cone} $\sigma(\pi,\tau)$ to be
	\begin{equation} \label{eq:nested_cone}
		\sigma(\pi, \tau) :=\left\{ \w \in W_d \ : \ 
		\begin{array}{c}
			\w \in \sigma(\pi),  \\
			\DD^\pi \w \in \check{\sigma}(\tau). 
		\end{array}
		\right\}
	\end{equation}
	Note here $\sigma(\pi)$ is in $W_d$ and $\check{\sigma}(\tau)$ is in $\R^d$. 
	
\end{defn}
One can check that $\sigma(\pi,\tau)$ is a well-defined $d$-dimensional cone in $W_d$, and has a minimal inequality description:

\[ \sigma(\pi, \tau)  = \left\{ \w \in W_d \ : \ 
0 \le \DD^\pi_{\tau^{-1}(1)}\w \leq \DD^\pi_{\tau^{-1}(2)}\w \leq \dots \leq  \DD^\pi_{\tau^{-1}(d)}\w
\right\}.
\]
Therefore, the relative interior of $\sigma(\pi,\tau)$ is given by
\[ \sigma^\circ(\pi, \tau)  = \left\{ \w \in W_d \ : \ 
0 < \DD^\pi_{\tau^{-1}(1)}\w < \DD^\pi_{\tau^{-1}(2)}\w < \dots <  \DD^\pi_{\tau^{-1}(d)}\w \right\}.
\]

Let $\MB^2_d := \{\sigma(\pi,\tau): (\pi, \tau) \in \fS_{d+1}\times \fS_{d}\}$ be the collection of nested braid cones in $W_d.$  In \cite{def-cone} we proved that $\MB^2_d$ induces a projective fan $\Br^2_d$ which we call the \emph{nested braid fan}. More precisely, in \cite[Proposition 4.6]{def-cone} we prove that for any appropriate pair $(\balpha, \bbeta)$, the normal fan of $\Perm(\balpha, \bbeta)$ is equal to $\Br_d^2$. Thus a polytope $P$ is a generalized nested permutohedron if and only if its normal fan coarsens the nested braid fan $\Br^2_d$ for some $d$. 

Analogous to Lemma \ref{lem:braid_dissect} we have the following.

\begin{lem}\label{lem:nested_braid_dissect}
	The collection of nested braid cones $\MB^2_d$ forms a conic dissection of $W_d.$
\end{lem}

Recall we define $\e_S = \sum_{i \in S} \e_i$ for each $S\subseteq[d+1]$. For each element $\cS = (S_1,\dots,S_k) \in \OO_{d+1}$, we define
\begin{equation}
	\e_{\cS} := \sum_{i} i \e_{S_i}.
	\label{equ:defneT}
\end{equation}
We also define the \emph{type} of $\cS$, denoted by $\Type(\cS)$, 
to be the sequence $(t_1, t_2, \dots, t_{k-1}),$ where
\[ t_i = \sum_{j=1}^i |S_j|, \quad \text{for $0 \le i \le k$}.\]
The following theorem gives inequality descriptions for usual nested permutohedra, recalling that $\overline{\sO_{d+1}}= \bigcup_{i=2}^{d+1} \sO_{d+1,i}$ is the set of all non-trivial ordered partitions of $[d+1]:$

\begin{thm}[Theorem 4.20 in \cite{def-cone}]\label{thm:facetdes}
	Suppose $(\balpha, \bbeta) \in \R^{d+1} \times \R^d$ is a pair of strictly increasing sequences $(\balpha, \bbeta) \in \R^{d+1} \times \R^d$ that is appropriate.  Suppose $\mathbf{b} \in \R^{\overline{\sO_{d+1}}}$ is defined as follows: for each $\cS \in \overline{\sO_{d+1}},$ 
	if $\Type(\cS)=(t_1,t_2,\dots, t_{k-1}),$ 
	let
	\begin{equation}
		b_{\cS} =  \left( \sum_{i=1}^{k} i \sum_{j=t_{i-1}+1}^{t_i} \alpha_j\right) + \sum_{j=d+2-k}^d \beta_j, 
		\label{equ:usualb}
	\end{equation} 
	where by convention we set $t_0=0$ and $t_k = d+1.$ 
	Then we have the following minimal inequality description for $\Perm(\balpha, \bbeta)$: 
	\begin{equation}\label{eq:nested_perm_ineqs}
		\Perm(\balpha, \bbeta) = \left\{\x\in U_d^\balpha \ : \ \langle \e_{\cS}, \x \rangle \leq b_{\cS},\quad \forall  \cS\in\overline{\sO_{d+1}}\right\}.
	\end{equation}
\end{thm}

\subsection{Cones associated to partition labeled trees} \label{subsec:newlabeling}

Let $(\cS, \TT) \in \sP_{d}$. Suppose $\cS = (S_1, S_2, \dots, S_k)$ and $\Type(\cS) = (t_1, \dots, t_{k-1}).$ Hence, $\TT$ has $k$ leaves. 
Recall in Section \ref{subsec:labeling}, we describe a way to label internal vertices of $\TT$ with the set $[k-1].$ We apply the same procedure on $\TT$ first, and then replace each label $i$ with $t_i.$ Hence, we obtain a labeling on internal vertices of $\TT$ with the set $\{t_1, t_2, \dots, t_{k-1}\}$.
Let $\GG(\cS,\TT)$ be the induced subtree of $\TT$ on its internal vertices together with the new labeling. One sees that $\GG(\cS, \TT)$ is the Hasse diagram of a preorder on $\{t_1,t_2,\dots,t_{k-1}\},$ which we denote by $\preceq_{(\cS, \TT)}$. For convenience, we also treat $\Type(\cS)$ as a set, and thus we can say that $\preceq_{(\cS,\TT)}$ is a preorder on $\Type(\cS).$

\begin{ex}\label{ex:partitionorder}
	Suppose $(\cS, \TT)$ is the $[9]$-partition labeled tree on the right of Figure \ref{fig:permutree1}. Then $\cS= (374,89,6,1,2,5)$ and $\Type(\cS) = (3,5,6,7,8).$ In Figure \ref{fig:partitionorder}, on the left we show $\TT$ together with its old internal vertex labeling considered in Section \ref{subsec:labeling}, and in the middle we show $\TT$ with its internal vertices labeled by $\Type(\cS)$, and finally the right side is $\GG(\cS,\TT)$, which defines the preorder $\preceq_{(\cS,\TT)}.$

	\begin{figure}[ht]
		\begin{tikzpicture}[scale=0.6]
			\begin{scope}
				\draw (0.5,0.5)--(3,3)--(4.5,1.5) (4,2)--(3.5,1.5) (3,3)--(3,2) (2,1)--(2,2) (2,2)--(2.5,1.5);
				
				\draw[fill=white] (0.5,0.5) circle (.1);
				
				\draw[fill=black] (2,2) circle (.1);
				\draw[fill=white] (2,1) circle (.1);
				\draw[fill=white] (2.5,1.5) circle (.1);
				\draw[fill=black] (3,3) circle (.1);
				\draw[fill=black] (4,2) circle (.1);
				\draw[fill=white] (3.5,1.5) circle (.1);
				\draw[fill=white] (4.5,1.5) circle (.1);
				\draw[fill=white] (3,2) circle (.1);
				\node[above] at (2,2) {\footnotesize 1,2};
				\node[above] at (3,3) {\footnotesize 3,4};
				\node[above] at (4,2) {\footnotesize 5};	
			\end{scope}
			
			\begin{scope}[xshift=6cm]
				\draw (0.5,0.5)--(3,3)--(4.5,1.5) (4,2)--(3.5,1.5) (3,3)--(3,2) (2,1)--(2,2) (2,2)--(2.5,1.5);
				
				\draw[fill=white] (0.5,0.5) circle (.1);
				
				\draw[fill=black] (2,2) circle (.1);
				\draw[fill=white] (2,1) circle (.1);
				\draw[fill=white] (2.5,1.5) circle (.1);
				\draw[fill=black] (3,3) circle (.1);
				\draw[fill=black] (4,2) circle (.1);
				\draw[fill=white] (3.5,1.5) circle (.1);
				\draw[fill=white] (4.5,1.5) circle (.1);
				\draw[fill=white] (3,2) circle (.1);
				\node[above] at (2,2) {\footnotesize 3,5};
				\node[above] at (3,3) {\footnotesize 6,7};
				\node[above] at (4,2) {\footnotesize 8};	
			\end{scope}
			
			\begin{scope}[xshift=12cm, yshift=-0.8cm]
				\draw (2,2)--(3,3)--(4,2);
				
				\draw[fill=black] (2,2) circle (.1);
				\draw[fill=black] (3,3) circle (.1);
				\draw[fill=black] (4,2) circle (.1);
				
				\node[above] at (2,2) {\footnotesize 3,5};
				\node[above] at (3,3) {\footnotesize 6,7};
				\node[above] at (4,2) {\footnotesize 8};
				
			\end{scope}
			
		\end{tikzpicture}
		\caption{An example of the construction of $\GG(\cS,\TT)$.} 
		\label{fig:partitionorder}
	\end{figure}
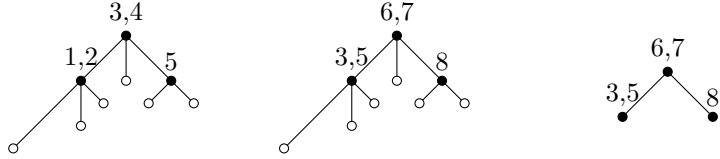
	
\end{ex}

\begin{rem}\label{rem:suits}
	Given the way we construct the labeling of $\TT$ using the set $\Type(\cS),$ one sees that a lot of properties we discussed in Section \ref{subsec:labeling} on internal vertex labelings of $\TT$, e.g., Lemma \ref{lem:label}, can be modified to a version that works for the current version of labeling. 
\end{rem}

For each $(\cS, \TT) \in \sP_{d}$, we will define a preorder cone associated to it using the preorder $\preceq_{(\cS,\TT)}.$ We need a preliminary lemma before giving such a definition.

Recall that for any $\cS \in \sO_{d+1}$, we associate a preorder $\preceq_\cS$ on $[d+1]$ to it (see \eqref{equ:ord_preorder}).
Let 
\begin{equation}\label{eq:Ws}
	W_\cS := \{ \w \in W_d \ : \ w_i = w_j, \text{if $i \equiv_\cS j$}\}.
\end{equation}
Recall $\cS(\pi)$ is defined in \eqref{equ:cSpi}. 

\begin{lem}\label{lem:indeppi}
	Suppose $\cS \in \sO_{d+1}$ and $\pi,\pi' \in \fS_{d+1}.$ If $\cS(\pi)$ and $\cS(\pi')$ both refine $\cS,$ then for any $\w \in W_\cS$, we have
	\[ \DD_i^\pi \w = \DD_i^{\pi'} \w, \text{ for all $i$}.\]
\end{lem}
We won't give a proof for the above lemma, which is straightforward by checking the definition. Instead, we will give an example to demonstrate why it is true.

\begin{ex}
	Let $\cS= (374,89,6,1,2,5)$. Then $\w \in W_{\cS}$ if and only if
	\begin{equation}\label{equ:worderex}		
		w_3= w_7= w_4 \quad \text{ and } \quad w_8= w_9.
	\end{equation}
	Let $\pi=791386245$ and $\pi'=793286145.$ Then $\cS(\pi) = (3,7,4,8,9,6,1,5,2)$ and $\cS(\pi') = (7,4,3,8,9,6,1,5,2),$ both of which refine $\cS.$ 
	Clearly, for each $4 \le i \le 8$, we have $\DD_i^\pi \w = \DD_i^{\pi'} \w$ for any $\w \in W_d.$ Now if $\w \in W_{\cS},$ we have $\DD_i^\pi \w =0=  \DD_i^{\pi'} \w$ for $i=1,2$, and 
	\[ \DD_3^\pi \w = w_8-w_4 =w_8-w_3 =  \DD_3^{\pi'} \w.\]
\end{ex}

Lemma \ref{lem:indeppi} allows us to give the following definition.

\begin{defn}\label{def:cone_general}
	Let $(\cS,\TT) \in \sP_d$. Choose $\pi \in \fS_{d+1}$ such that $\cS(\pi)$ refines $\cS.$ Then we define 
	\begin{equation}
		\sigma(\cS,\TT) 
		:= \left\{ \w \in W_\cS \ :  
		\begin{array}{c}
			
			\DD^\pi_i \w \le \DD^\pi_j \w, \text{ if $i \preceq_{(\cS,\TT)} j$} \\ 
			\DD^\pi_\ell \w \ge 0, \text{ if $\ell$ is a minimal element of $\preceq_{(\cS,\TT)}$} 
		\end{array}
		\right\}.
	\end{equation}
\end{defn}
Note that by Lemma \ref{lem:indeppi}, the definition of $\sigma(\cS,\TT)$ does not depend on the choice of $\pi$ as long as $\cS(\pi)$ refines $\cS.$

\begin{ex}\label{ex:running}
	Let $(\cS, \TT)$ be the $[9]$-partition labeled tree on the right of Figure \ref{fig:permutree1}. Then $\cS= (374,89,6,1,2,5)$ and its preorder $\preceq_{(\cS,\TT)}$ has been discussed in Example \ref{ex:partitionorder}. 
	We choose $\pi=791386245$ which we have shown that $\cS(\pi)$ refines $\cS.$ Thus, we have that $\w \in \sigma(\cS, \TT)$ if and only if 
	both condition in \eqref{equ:worderex} and the condition below hold: 
	\begin{equation}\label{equ:picond} 
		0 \le \DD^\pi_3\w = \DD^\pi_5\w \le  \DD^\pi_6\w = \DD^\pi_7\w \text{ and } 0 \le \DD^\pi_8\w \le \DD^\pi_6\w.  
	\end{equation}
\end{ex}

\subsection{Face structure of \texorpdfstring{$\sigma(\cS,\TT)$}{our cones}} \label{subsec:facestructure}

In Section \ref{sec:permasso}, we will show (in Proposition \ref{prop:nestedpermasso}) that the collection of cones 
\begin{equation}\label{eq:fan_candidate}
	\Xi_d:=\{ \sigma(\cS,\TT) : (\cS,\TT) \in \sP_d\}
\end{equation}
is the normal fan of the permuto-associahedron that we construct. Therefore, we call $\Xi_d$ the \emph{nested Loday fan}. In the remaining of this section, we will explore properties of $\sigma(\cS,\TT)$ and  $\Xi_d.$ 
The main goal of this subsection is to prove the following proposition, which establishes the connection between $\Xi_d$ and the Kapranov poset.

\begin{prop}\label{prop:kapface}
	The map $(\cS,\TT) \mapsto \sigma(\cS, \TT)$ gives a bijection from $\sP_d$ to $\Xi_d.$
	Furthermore, for any $(\cS_1,\TT_1), (\cS_2,\TT_2) \in \sP_d$, we have that $(\cS_1,\TT_1) \le_{KP} (\cS_2,\TT_2)$ in the Kapranov poset $\cK\Pi_d$ if and only if the cone $\sigma(\cS_2,\TT_2)$ is a face of $\sigma(\cS_1,\TT_1)$.
\end{prop}

First we prove the following lemma with various technical but basic facts about the cones $\sigma(\cS, \TT)$.

\begin{lem}\label{lem:dimension_fancycones}
	Let $(\cS, \TT) \in \sP_d$ and choose $\pi \in \fS_{d+1}$ such that $\cS(\pi)$ refines $\cS.$ Recall $W_\cS$ is defined in \eqref{eq:Ws}. 
	We define the affine space 
	\begin{equation}\label{eq:lineality} W_{(\cS,\TT)} := \left\{ \w \in W_{\cS}  :   \DD_i^\pi \w = \DD_j^\pi \w, \text{ if } i \equiv_{(\cS,\TT)} j \right\}
		= \left\{ \w \in W_{d} \ :  
		\begin{array}{c}
			w_i = w_j, \text{ if } i \equiv_\cS j \\
			\DD_i^\pi \w = \DD_j^\pi \w, \text{ if } i \equiv_{(\cS,\TT)} j 
		\end{array}
		\right\}.
	\end{equation}
	Then the cone $\sigma(\cS,\TT)$ has the following inequality description:
	\begin{equation}\label{equ:markedint} 
		\sigma(\cS, \TT) 
		= \left\{ \w \in W_{(\cS, \TT)} \ :  
		\begin{array}{c}
			\DD^\pi_i \w \le \DD^\pi_j \w, \text{ if $i \cover_{(\cS,\TT)} j$} \\
			\DD^\pi_\ell \w \ge 0, \text{ if $\ell$ is a minimal element of $\preceq_{(\cS,\TT)}$} 
		\end{array}
		\right\}.
	\end{equation} 
	Moreover, the following statements hold:
	\begin{enumerate}
		\item\label{itm:dimension} The cone $\sigma(\cS,\TT)$ is full-dimensional in $W_{(\cS,\TT)}$, whose dimension is equal to the number of internal vertices of $\TT.$ 
		
		\item \label{itm:facetdefining}
		The inequality description \eqref{equ:markedint} for $\sigma(\cS,\TT)$ is facet-defining.
		
		\item \label{itm:markedface} 
		For every facet $F$ of $\sigma(\cS,\TT)$, there exists $(\cS',\TT')$ covering $(\cS,\TT)$ in $\cK\Pi_d$ such that $F = \sigma(\cS',\TT').$ Conversely, for every $(\cS',\TT')$ in $\cK\Pi_d$ that covers $(\cS,\TT)$, its associated cone $\sigma(\cS',\TT')$ is a facet of $\sigma(\cS,\TT).$  
	\end{enumerate}
\end{lem}

\begin{proof}
	First, we observe that given the definition of $W_{(\cS,\TT)},$ the inequaliy description \eqref{equ:markedint} clearly defines $\sigma(\cS,\TT)$. Hence, it is left to verify statements (1)--(3).
	
	In order to make notation easy, we let $V(\cS,\TT)$ and $K(\cS,\TT)$ be the images of $W_{(\cS,\TT)}$ and $\sigma(\cS,\TT)$ under the map $\DD^\pi: W_d\rightarrow \R^d$. 
	Then 
	\[ V(\cS,\TT) =\left\{\u \in \R^d: 	
	\begin{array}{c}
		u_\ell = 0, \text{ if } \ell \not\in \Type(\cS) \\
		u_i = u_j, \text{ if } i \equiv_{(\cS,\TT)} j 
	\end{array}\right\},\]
	and
	\begin{equation} \label{eq:aux0}
		K(\cS,\TT) = \left\{\u \in  V(\cS,\TT) : 	
		\begin{array}{c} u_i\leq u_j,  \text{ if $i \cover_{(\cS,\TT)} j$}\\ u_\ell \ge 0, \text{ if $\ell$ is a minimal element of $\preceq_{(\cS,\TT)}$}\end{array}\right\}.
	\end{equation} 
	One important observation is that the linear map $\DD^\pi$ is an isomorphism from $W_d$ to $\R^d$. Hence, it suffices to show the following three corresponding statements: 
	\begin{enumerate}[label=(C\arabic*)]
		\item The cone $K(\cS,\TT)$ is full-dimensional in $V(\cS,\TT)$, whose dimension is equal to the number of internal vertices of $\TT.$

		\item 
		The inequality description \eqref{eq:aux0} for $K(\cS,\TT)$ is facet-defining.
		
		\item 
		For every facet $K$ of $K(\cS,\TT)$, there exists $(\cS',\TT')$ covering $(\cS,\TT)$ in $\cK\Pi_d$ such that $K = K(\cS',\TT').$ Conversely, for every $(\cS',\TT')$ in $\cK\Pi_d$ that covers $(\cS,\TT)$, the cone $K(\cS',\TT')$ is a facet of $K(\cS,\TT).$ 
	\end{enumerate}
	
	We prove (C1) first.
	Recall that $\preceq_{(\cS,\TT)}$ is a preorder on $\Type(\cS).$
	One sees that $V(\cS,\TT) \cong \{ \u \in \R^{\Type(\cS)} \ : \ u_i = u_j, \text{ if } i \equiv_{(\cS,\TT)} j \}.$ Clearly, the dimension of the latter is the number of equivalence classes in $\Type(\cS)/\equiv_{(\cS,\TT)},$ which is exactly the number of internal vertices of $\TT.$ 
	
	We now show that $K(\cS,\TT)$ is full-dimensional in $V(\cS,\TT).$
	Given that $K(\cS,\TT)$ is described by \eqref{eq:aux0}, it is enough to show there exists $\u \in V(\cS,\TT)$ such that 
	\[ \text{(a) $u_i< u_j,$ if $i \cover_{(\cS,\TT)} j$ \quad and \quad (b) $u_k > 0$ if $k$ is a minimal element of $\preceq_{(\cS,\TT)}$.}\] 
	We can construct such a $\u$ easily: First set $u_k = 0$ for each 
	$k \in [d] \setminus \Type(\cS).$
	Next, noticing that the Haase diagram $\GG(\cS,\TT)$ for the preorder $\preceq_{(\cS,\TT)}$ without labels is just $\TT$ which is a rooted tree, we set $u_i = d -k$ for each
	$i \in \Type(\cS)$
	that is a label for an internal vertex that is $k$-distance away from the root of $\TT$. 
	It is easy to see a vector $\u$ constructed this way satisfies the desired conditions. Thus, $K(\cS,\TT)$ is full-dimensional in $V(\cS,\TT)$ (whose dimension is the number of internal vertices of $\TT$.) Hence, statement (C1) holds.

	Next we prove (C2). There are two kinds of inequalities in \eqref{eq:aux0}. We treat them separately.
	\begin{enumerate}[label=\roman*., leftmargin=*]
		\item Suppose $i \cover_{(\cS,\TT)} j$, and assume that $i$ and $j$ are labels of the internal vertices $x$ and $y$ of $\TT.$ Let $\TT'$ be the tree obtained from $\TT$ by contracting the internal edge $\{x,y\}$ and let $\cS':= \cS.$ Clearly, we have $\Type(\cS') = \Type(\cS)$, and thus $\preceq_{(\cS,\TT)}$ and $\preceq_{(\cS',\TT')}$ are two preorders on the same set. By Lemma \ref{lem:label}~\eqref{itm:label-compatible} and Remark \ref{rem:suits}, we conclude that $\preceq_{(\cS',\TT')}$ is obtained from $\preceq_{(\cS,\TT)}$ by merging the equivalence classes $\bar{i}$ and $\bar{j}$. Hence, one sees that 
		\[ K(\cS,\TT) \cap \{ \u : u_i = u_j\} = K(\cS',\TT').\]
		Since $\TT$ has one less internal vertex than $\TT$, it follows from property (C1) that $K(\cS',\TT')$ is a facet of $K(\cS,\TT).$ Therefore, the inequality $u_i \le u_j$ if facet-defining.
		
		\item Suppose $\ell$ is a minimal element of $\preceq_{(\cS,\TT)}$, and assume $\ell$ is a label of the internal vertex $p$ of $(\cS,\TT).$ Clearly, $p$ is a minimal internal vertex of $\TT.$ 
		
		Let $(\cS', \TT')$ be the partition labeled tree obtained from $(\cS,\TT)$ by contracting all edges between $p$ and its children in $\TT$ while labeling this new leaf with the union of the labels of the children of $p$ in $\TT$. More precisely, suppose $\cS = (S_1, S_2,\dots, S_k)$, and assume the children of $p$ are labeled by $S_{a}, S_{a+1}, \dots, S_{b},$ then 
		\[\cS' = (S_1, S_2, \dots, S_{a-1}, \bigcup_{i=a}^b S_i, S_{b+1}, S_{b+2}, \dots, S_{k}).\]
		Hence, if $\Type(\cS) = (t_1, t_2, \dots, t_{k-1})$ then $\Type(\cS') = (t_1, t_2, \dots, t_{a-1}, t_b, t_{b+1}, \dots, t_{k-1}).$ One checks that $\GG(\cS', \TT')$ is obtained from $\GG(\cS,\TT)$ by removing $p$ together with its adjacent edge and its labeling (which is $\{t_a, t_{a+1},\dots, t_{b-1}\})$. Finally, using all these, we can verify that 
		\[ K(\cS,\TT) \cap \{ \u : u_\ell = 0\} = K(\cS',\TT').\]
		Similarly, since $\TT'$ has one less internal vertex than $\TT$, we have that $K(\cS',\TT')$ is a facet of $K(\cS,\TT),$ and thus the inequality $u_\ell \ge 0$ if facet-defining.
	\end{enumerate}
	
	Finally, we prove (C3). Note that in the proof of (C2), the partition labeled tree $(\cS',\TT')$ we constructed in each case covers $(\cS,\TT)$ in $\cK\Pi_d,$ and every $(\cS',\TT')$ covers $(\cS,\TT)$ arises from one of the inequality discussion. Therefore, (C3) follows.
\end{proof}

\begin{ex}\label{ex:nonsimplicial}
	Let $(\cS,\TT)$ the $[9]$-partition labeled tree on the left of Figure \ref{fig:permutree1}.
	Since $\TT$ has $8$ internal vertices, by Lemma \ref{lem:dimension_fancycones}~\eqref{itm:dimension}, we have that $\dim\sigma(\cS,\TT)=8$.
	Recall there are two types, type i and type ii, of covering relations in $\cK\Pi_d$ defined in Definition \ref{defn:kap}. In the poset $\cK\Pi_8$, our partition labeled tree $(\cS,\TT)$ are covered by $10$ elements, $7$ of which are obtained from $(\cS,\TT)$ by contracting a single internal edge of $\TT$ and the remaining $3$ are obtained by contracting the pairs of leaves $\{3,7\},\{8,9\},\{5,2\}$ in $\TT$, respectively.
	By Lemma \ref{lem:dimension_fancycones}~\eqref{itm:markedface}, the cone $\sigma(\cS,\TT)$ has $10$ facets, hence it is not simplicial.
\end{ex}

\begin{lem}\label{lem:face_bijection}
	If $(\cS,\TT), (\cS',\TT') \in \sP_d$ are distinct, then $\sigma(\cS,\TT)\neq \sigma(\cS',\TT')$ as cones in $W_d$.
\end{lem}

In order to prove the the above lemma, we need a 
modified version of preorder cones that are also in bijection with preorders.

For every precorder $\preceq$ on $[n]$,
we define \begin{equation}\label{eq:tildecone}
	\tilde{\sigma}_\preceq := \{\x\in \R_{\ge 0}^n \ : \ x_i\leq x_j \textrm{ if }i\preceq j\}.
\end{equation}

The following result is straightforward to check. 

\begin{lem}\label{lem:bijection}

	The map $\preceq \mapsto \tilde{\sigma}_\preceq$ is an injection from the set of all preorders on $[n]$ to the set of cones in $\R^n_{\geq 0}$.
\end{lem}

\begin{proof}[Proof of Lemma \ref{lem:face_bijection}]
	Suppose $\sigma(\cS,\TT)=\sigma(\cS',\TT')$. We need to show $(\cS,\TT)=(\cS',\TT').$
	
	We first prove that $\cS = \cS'.$
	Suppose $\cS=(S_1, \dots, S_k)$. It follows from Lemma \ref{lem:dimension_fancycones} that for any $\w$ in the interior of $\sigma(\cS,\TT),$ we have $w_i = w_j$ if $i, j \in S_a$ for some $a$, and $w_i < w_j$ if $i \in S_a$ and $j \in S_b$ for some $a < b.$ Since $\sigma(\cS, \TT) = \sigma(\cS', \TT')$, they have exactly the same interior. Thus, we must have that $\cS = \cS'.$

	Next, we show $\TT=\TT'$. We choose $\pi$ such that $\cS(\pi)$ refines $\cS=\cS'$, and let $K(\cS, \TT) = \DD^\pi(\sigma(\cS,\TT))$ and $K(\cS', \TT') = \DD^\pi(\sigma(\cS',\TT'))$ as defined in the proof of Lemma \ref{lem:dimension_fancycones}. Since $\sigma(\cS,\TT)=\sigma(\cS',\TT')$, we must have that $K(\cS, \TT) = K(\cS',\TT')$.
	Recall that $\preceq_1 := \preceq_{(\cS,\TT)}$ and $\preceq_2 := \preceq_{(\cS',\TT')}$ are preorders on $\Type(\cS)=\Type(\cS').$ We consider the modified version of preorder cones
	\[ \tilde{\sigma}_{\preceq_\ell} =  \{ \u \in \R_{\ge 0}^{\Type(\cS)} \ : \ u_i \le u_j, \text{ if } i \preceq_\ell j \}, \quad \ell=1,2,\]
	defined in \eqref{eq:tildecone}. 
	By description of $K(\cS, \TT)$ given in the proof of Lemma \ref{lem:dimension_fancycones}, one sees that $K(\cS, \TT) \cong  \tilde{\sigma}_{\preceq_1}$ and $K(\cS', \TT') \cong  \tilde{\sigma}_{\preceq_2}$. Thus, $\tilde{\sigma}_{\preceq_1} =  \tilde{\sigma}_{\preceq_2}$. Then it follows from Lemma \ref{lem:bijection} that $\preceq_1 = \preceq_{(\cS,\TT)}$ and $\preceq_2 = \preceq_{(\cS',\TT')}$ are the same preorder on $\Type(\cS).$ Since we can recover the tree from the Hasse diagram of the preorder, we must have that $\TT = \TT'$, completing the proof. 
\end{proof}

\begin{proof}[Proof of Proposition \ref{prop:kapface}]
	By the definition of $\Xi_d$, the map $(\cS,\TT) \mapsto \sigma(\cS, \TT)$ clearly is a surjection from $\sP_d$ to $\Xi_d.$ Moreover, by Lemma \ref{lem:face_bijection}, the map is injective. Therefore, the first conclusion of the proposition follows. 
	
	The second conclusion follows from Lemma \ref{lem:dimension_fancycones} \eqref{itm:markedface} and the transitivity of the poset. 
\end{proof}

\subsection{Nested Loday fan}\label{subsec:nestedloday} 
Recall the nested Loday fan $\Xi_d$ is defined by Equation \eqref{eq:fan_candidate}. 
(As we mentioned before, the proof for that $\Xi_d$ is a fan will be completed in the next section. However, for convenience we still refer to $\Xi_d$ as the nested Loday fan.) 
In this subsection, we explore further properties of $\Xi_d$,  
and summarize useful results that will be used in Section \ref{sec:permasso}.

Let $\KK_d \subseteq \Xi_d$ be the set of maximal cones (in terms of dimension). It then follows from Proposition \ref{prop:kapface} that the set $\Xi_d$ is induced by $\KK_d$. By Lemma \ref{lem:dimension_fancycones}~\eqref{itm:dimension} and Lemma \ref{lem:kp_poset}~\eqref{itm:rank1}, we see that 
\[ \KK_d=\{\sigma(\cS, \TT):(\cS, \TT) \text{ is a rank-$1$ element of $\cK\Pi_d$} \} = \{\sigma(\cS(\pi), \TT):(\pi, \TT) \in \fS_{d+1} \times \sT_{d,d}\}, \]
and each cone has dimension $d$ (and thus is full-dimensional in $W_d$).
For ease of notation we define for any pair $(\pi, \TT) \in \fS_{d+1} \times \sT_{d,d}$,
\begin{equation}
	\sigma(\pi,\TT):=\sigma(\cS(\pi), \TT).
\end{equation} 
Then for these maximal dimensional cones in $\KK_d,$ their descriptions in Definition \ref{def:cone_general} and expressions for their interiors can be simplified.

Recall $\check{\sigma}(\tau)$ is defined in Defintion \ref{def:check}. Similar to how braid cones $\sigma(\pi)$ are generalized to precorder cones $\sigma_\preceq,$ for any preorder $\preceq$ on $[n],$ we define

	\begin{equation}\label{eq:posetcone_extended}
		\check{\sigma}_\preceq:= \{\x\in \R^n \ : \ x_i\leq x_j \textrm{ if }i\preceq j\},
	\end{equation}
	and for any $\TT \in \sT_{n}$, we set $\check{\sigma}(\TT):=\check{\sigma}_{\preceq_\TT}$.

With these notations, we have	

\begin{equation}\label{eq:top_dimensional_cones}
	\sigma(\pi, \TT) 
	=\left\{ \w \in W_d \ :  
	\begin{array}{c}
		\w \in \sigma(\pi) \\
		\DD^\pi \w \in \check{\sigma}(\TT)
	\end{array}
	\right\}, \text{ and } \sigma^\circ(\pi, \TT) 
	=\left\{ \w \in W_d \ :  
	\begin{array}{c}
		\w \in \sigma^\circ(\pi) \\
		\DD^\pi \w \in \check{\sigma}^\circ(\TT)
	\end{array}
	\right\},
\end{equation}

where 
\begin{equation}
	\check{\sigma}^\circ(\TT) =  \{\x\in \R^{d} \ : \ x_i < x_j \textrm{ if }i\cover_\TT j\}. 
	\label{eq:sigmaTinterior}
\end{equation} 
\begin{lem}\label{lem:treefancy}
	Each cone $\sigma \in \KK_d$ is a union of nested braid cones. Furthermore, the collection of cones $\KK_{d}$ is a pointed conic dissection of $W_{d}$. 
\end{lem}

\begin{proof}
	Let $\sigma \in \KK_d.$ Then $\sigma =\sigma(\pi,\TT)$ for some $(\pi, \TT) \in \fS_{d+1} \times \sT_{d,d}$. 
	By an analogous proof of Lemma \ref{lem:poset_facts}~\eqref{itm:faces}, we have that $\check{\sigma}(\TT)=\bigcup_{\tau\in\LL[\TT]}\check{\sigma}(\tau)$,
	so combining Equation \eqref{eq:top_dimensional_cones} together with the definition of nested braid cone as in Equation \eqref{eq:nested_cone} we obtain
	\begin{equation}\label{eq:dissection_2}
		\sigma(\pi,\TT)=\bigcup_{\tau\in\LL[\TT]}\sigma(\pi,\tau).
	\end{equation}
	Furthermore, for each fixed $\pi \in \fS_{d+1}$, it follows from Lemma \ref{lem:unique} that each permutation $\tau \in \fS_d$ appears in the above expression for exactly one $\TT \in \sT_{d,d}$. By Lemma \ref{lem:nested_braid_dissect} the collection $\{\sigma(\pi,\tau): (\pi, \tau) \in \fS_{d+1}\times \fS_{d}\}$ is a conic dissection of $W_d,$ so we conclude that so is $\KK_d$. Finally, by \eqref{eq:top_dimensional_cones}, we see that for any $(\pi, \TT) \in \fS_{d+1} \times \sT_{d,d}$, the cone $\sigma(\pi, \TT) \subseteq \sigma(\pi).$ Since the latter is pointed, so is the former. 
\end{proof}

Lemma \ref{lem:nestedlodayfan} below summarizes properties of $\Xi_d$, recalling $\overline{\sO_{d+1}} = \bigcup_{k=2}^{d+1} \sO_{d+1,k}$ is the set of all non-trivial ordered partitions of $[d+1]$.

\begin{notn}\label{notn:ST}
	For any $\cS \in \overline{\sO_{d+1}},$ if $\cS$ has $k$ blocks, we define $\TT_\cS$ to be the unique tree with one internal vertex and $k$ leaves.
\end{notn}

\begin{lem}\label{lem:nestedlodayfan} The nested Loday fan $\Xi_d$ has the following properties:
	\begin{enumerate}
		\item\label{itm:nestedlodayfantop} 
		The set $\KK_d = \{ \sigma(\pi,\TT): (\pi,\TT) \in \fS_{d+1} \times \sT_{d,d}\}$ consists of all the $d$-dimensional cones in $\Xi_d.$
		\item\label{itm:nestedlodayray} 
		The set $\{ \sigma(\cS, \TT_\cS):  \cS\in\overline{\sO_{d+1}}\}$ consists of all the 1-dimensional cones of $\Xi_d$. 
		Moreover, for each $\cS\in\overline{\sO_{d+1}}$, the vector $\e_\cS\in W_d$ is a generator for the $1$-dimensional cone $\sigma(\cS, \TT_\cS)$.
		
		\item\label{itm:nestedfaceposet} The face poset of $\Xi_d$ is isomorphic to the poset dual to  $\cK\Pi_d\setminus\hat{0}$. 
		\item\label{itm:nonsimplicial} The nested Loday fan $\Xi_d$ is not simplicial for $d \ge 3$. 
		\item\label{itm:nestedlodaycoarsen} The nested Loday fan $\Xi_d$ is a coarsening of the nested braid fan $\Br^2_d.$
	\end{enumerate}
\end{lem}

\begin{proof}
	Condition \eqref{itm:nestedlodayfantop} follows from the our discussion on maximal dimensional cones, condition \eqref{itm:nestedfaceposet} follows from Proposition \ref{prop:kapface}, 
	and condition \eqref{itm:nestedlodaycoarsen} follows from Lemma \ref{lem:treefancy}.

	The first assertion in \eqref{itm:nestedlodayray} follows from Lemma \ref{lem:dimension_fancycones}~\eqref{itm:dimension} and Lemma \ref{lem:kp_poset}~\eqref{itm:rankd}.
	Suppose $\cS=(S_1, S_2, \dots, S_k) \in \sO_{d+1,k}$. Then it is easy to check that
	\[ \sigma(\cS, \TT_\cS) = \left\{ \w \in W_d \ : \ 
	\begin{array}{c}
		w_i = w_j \text{ if $i,j \in S_a$ for some $a$} \\	  
		w_i \le w_j \text{ if $i \in S_a, j \in S_{a+1}$ for some $a$} \\	  
		w_j - w_i = w_\ell -w_k \text{ if $i \in S_a, j \in S_{a+1}, k \in S_b, \ell \in S_{b+1}$ for some $a,b$}
	\end{array}
	\right\}.
	\]
	One can verify that the vector $\e_\cS$ is a nonzero vector in the above cone. Therefore, the second assertion in \eqref{itm:nestedlodayray} follows.
	
	Example \ref{ex:nonsimplicial} shows a particular example which is not simplicial. In general it follows from Lemma \ref{lem:dimension_fancycones}~\eqref{itm:markedface} that the number of facets of a $d$-dimensional cone $\sigma(\pi,\TT)\in\KK_d$ is $a+b$ where $a$ is the number of internal edges of $\TT$ and $b$ is the number of vertices adjacent to exactly two leaves. One checks that any tree $\TT \in \sT_{d,d}$ has $a=2d-(d+1)=d-1$, and if $d \ge 3,$ there exists $\TT \in \sT_{d,d}$ such that $b\ge2$. Hence, there exist non simplicial cones for every $d\geq 3$. 
\end{proof}

\section{Realization of the permuto-associahedron} \label{sec:permasso}
Recall that a $d$-permuto-associahedron is a $d$-dimensional polytope whose face poset is isomorphic to $\cK\Pi_d$.
Similar to Section \ref{sec:asso}, we follow the method outlined in the introduction to construct a realization for the permuto-associahedron, which is the majority of the content of this section. 
In the last part of this section, we compare our realization with the one given by Reiner and Ziegler in \cite{reiner_ziegler}. 

\subsection{Vertices of the nested permuto-associahedron}
\label{subsec:vertices_nested}
In this part, we will give a set of points that are vertex candidates of our realization. We start by adapting notation from Definition \ref{def:vector}.

\begin{notn}\label{notn:vtbeta}
	For any $\bbeta = (\beta_1, \dots, \beta_d) \in \R^d$ and $\TT \in \sT_{d,d},$ we let
	\[ \beta_{\TT,i} := \val(\bbeta, \TT_{(i)}).\]
	Therefore, $\v_{\TT}^\bbeta = \sum_{i=1}^d \beta_{\TT,i} \ \e_i.$
\end{notn}

Given strictly increasing sequences $\balpha = (\alpha_1, \alpha_2, \dots, \alpha_{d+1}) \in \R^{d+1}$ and $\bbeta = (\beta_1, \beta_2, \dots, \beta_{d}) \in \R^{d}$,
for any $(\pi, \TT) \in \fS_{d+1} \times \sT_{d,d}$, we define
\begin{equation}\label{equ:defnpav}
	\v_{\pi,\TT}^{(\balpha,\bbeta)} := \sum_{i=1}^{d+1} \alpha_i \e_{\pi^{-1}(i)} + \sum_{i=1}^d \beta_{\TT,i} \ff_{i}^\pi. 
\end{equation}
It is easy to see that $\v_{\pi,\TT}^{(\balpha,\bbeta)}$ lies in $U^\balpha_d$ since the sum of its coordinates is $\sum_{i=1}^{d+1} \alpha_i$.

After rearranging coordinates in \eqref{equ:defnpav}, we get the following expression:
\begin{equation}
	\v_{\pi,\TT}^{(\balpha,\bbeta)} = \sum_{i=1}^{d+1} \left(\alpha_i +\left( \beta_{\TT,i-1} - \beta_{\TT,i} \right) \right) \ \e_{\pi^{-1}(i)},
	\label{equ:expansion} 
\end{equation}
where by convention we let $\beta_{\TT,0} = \beta_{\TT, d+1} =0.$ 

Parallel to Section \ref{sec:nested}, we say that $(\balpha, \bbeta)$ is a \emph{$\sT$-appropriate pair} (of strictly increasing sequences), if for any complete binary tree $\TT \in \sT_{d,d},$ the coefficients of $\e_{\pi^{-1}(i)}$ in the above expansion increase strictly as $i$ increases.

\begin{rem}\label{rem:appropriate}
	We remark that being an ``appropriate pair'' and being a ``$\sT$-appropriate pair'' are not equivalent. We can show for $d =3$ that any $\sT$-appropriate pair $(\balpha,\bbeta)$ is an appropriate pair. We suspect that this implication is true in general; however, we do not have a proof. Since this is not relevant to the discussion of this paper, we leave it to interested readers. In any case, it is not hard to see that by scaling $\balpha$ with a sufficiently large factor, we can make $(\balpha,\bbeta)$ both ``appropriate'' and ``$\sT$-appropriate''.
	
\end{rem}
\begin{defn}\label{defn:permasso}
	Suppose $(\balpha, \bbeta) \in \R^{d+1} \times \R^d$ is a $\sT$-appropriate pair of strictly increasing sequences. We define the \emph{nested permuto-associahedron}
	\begin{equation}
		\PA(\balpha, \bbeta) := \conv\left(v_{\pi,\TT}^{(\balpha,\bbeta)}\ :\ (\pi,\TT) \in \fS_{d+1}\times \sT_{d,d}\right).
		\label{equ:defnpermasso}
	\end{equation}
\end{defn}

The next result is the main theorem of this paper.

\begin{thm}\label{thm:permassoc}
	Let $(\balpha, \bbeta) \in \R^{d+1} \times \R^d$ be a $\sT$-appropriate pair of strictly increasing sequences.
	
	Then the face poset of the nested permuto-associahedron $\PA(\balpha,\bbeta)$ is the Kapranov poset $\cK\Pi_d$. 
	Moreover, $\PA(\balpha,\bbeta)$ is a $d$-dimensional permuto-associahedron, and is a generalized nested permutohedron as well.
\end{thm}

\subsection{Normal fan of nested permuto-associahedra}
Recall that in \S\ref{subsec:nestedloday} we have defined $\KK_d$ to be the set of maximal cones in $\Xi_d$, and have shown that $\KK_d$ is a conic dissection of $W_d$ (see Lemma \ref{lem:treefancy}). 
The goal of this part is to use Lemma \ref{lem:main_tool} to show that $\KK_d$ induces the normal fan of $\PA(\balpha, \bbeta)$, as well as confirm the set of points $\left\{ v_{\pi,\TT}^{(\balpha,\bbeta)} \right\}$ defined above is indeed the vertex set of $\PA(\balpha,\bbeta).$

\begin{lem}\label{lem:aux}
	Suppose $(\balpha, \bbeta) \in \R^{d+1} \times \R^d$ is a $\sT$-appropriate pair of strictly increasing sequences. Let $(\pi,\TT), (\pi', \TT') \in \fS_{d+1} \times \sT_{d,d}$. Then for every $\w \in \sigma^\circ(\pi,\TT)$, we have 
	\begin{equation}
		\left\langle \w,\v^{(\balpha,\bbeta)}_{\pi,\TT} \right\rangle \ge \left\langle\w,\v^{(\balpha,\bbeta)}_{\pi',\TT'} \right\rangle
	\end{equation}
	where the equality holds if and only if $(\pi, \TT) = (\pi', \TT').$
\end{lem}

\begin{proof}
	We will prove the inequality by introducing an intermediate product and showing
	\begin{equation}
		\left\langle \w,\v^{(\balpha,\bbeta)}_{\pi,\TT} \right\rangle \ge
		\left\langle \w,\v^{(\balpha,\bbeta)}_{\pi,\TT'} \right\rangle \ge
		\left\langle\w,\v^{(\balpha,\bbeta)}_{\pi',\TT'} \right\rangle,
		\label{equ:intermprod}
	\end{equation}
	where the first equality holds if and only if $\TT = \TT'$ and the second equality holds if and only $\pi = \pi'.$
	
	We let $u_i = w_{\pi^{-1}(i)}$ for each $i$, which allows us to express $\w$ as in \eqref{equ:wexp0} and have $\DD^\pi \w = \DD \u.$
	Then because $\w \in \sigma^\circ(\pi, \TT)$, we have the following conditions from Equation \eqref{eq:top_dimensional_cones}: 
	\begin{enumerate}
		\item $\DD\u >0,$ which is equivalent to $u_{1} < u_{2} < \dots < u_{d+1}$, and
		\item $\DD_i \u < \DD_j \u$ if $i \cover_\TT j$. 
	\end{enumerate}
	
	Expression \eqref{equ:wexp0}, together with \eqref{equ:expansion}, allows us to compute products in \eqref{equ:intermprod} easily. 
	Since the pair $(\balpha, \bbeta)$ is $\sT$-appropriate, we have that $\left(\alpha_i +\left( \beta_{\TT',i-1} - \beta_{\TT',i} \right) \right)$ strictly increases as $i$ increases. This, together with condition (1) above and the Rearrangement Inequality \cite[Theorem 368]{hardy} gives us the second inequality in \eqref{equ:intermprod} and that the equality holds if and only if $\pi = \pi'.$

	Next we see that the first inequality in \eqref{equ:intermprod} is equivalent to
	\begin{equation*}
		\sum_{i=1}^{d+1}\left( \beta_{\TT,i-1} - \beta_{\TT,i} \right) u_i \ge
		\sum_{i=1}^{d+1} \left( \beta_{\TT',i-1} - \beta_{\TT',i} \right) u_i. 
	\end{equation*}
	After rearranging summations, the above inequality becomes
	\begin{equation}
		\left\langle \DD \u, \v_\TT^\bbeta \right\rangle = \sum_{i=1}^{d+1} (u_{i+1}-u_i) \beta_{\TT,i}  \ge
		\sum_{i=1}^{d+1} (u_{i+1}-u_i) \beta_{\TT',i} = \left\langle \DD \u, \v_{\TT'}^\bbeta \right\rangle. 
		\label{eq:intermprod1}
	\end{equation}
	Then because $\DD \u \in \sigma^\circ(\TT)$, it follows from Corollary \ref{cor:aux} that the inequality \eqref{eq:intermprod1} holds and its equality holds if and only if $\TT = \TT'.$
\end{proof}

The following proposition is the key result of this subsection, characterizing the vertex set and the normal fan of the nested permuto-associahedron. It also provides the main ingredients we need for proving Theorem \ref{thm:permassoc}. 

\begin{prop}\label{prop:nestedpermasso}
	Let $(\balpha, \bbeta) \in \R^{d+1} \times \R^d$ be a $\sT$-appropriate pair of strictly increasing sequences.	
	\begin{enumerate}
		\item \label{itm:nestedfulldim} The nested permuto-associahedron $\PA(\balpha,\bbeta)$ is full-dimensional in $U_d^\balpha$ and its vertex set is $\left\{v_{\pi,\TT}^{(\balpha,\bbeta)}\ :\ (\pi,\TT) \in \fS_{d+1}\times \sT_{d,d}\right\}$.
		\item \label{itm:nestedvertexcone} For each $(\pi,\TT) \in \fS_{d+1}\times \sT_{d,d}$, we have $\sigma(\pi, \TT)=\ncone\left(v_{\pi,\TT}^{(\balpha,\bbeta)},\PA(\balpha,\bbeta)\right).$
		
		\item\label{itm:nestednormal} The normal fan of $\PA(\balpha,\bbeta)$ is $\Xi_d=\{ \sigma(\cS,\TT) : (\cS,\TT) \in \sP_d\}.$
		Hence, $\Xi_d$ is a complete projective fan in $W_d.$
	\end{enumerate}
\end{prop}

\begin{proof}
	It follows from Lemmas \ref{lem:treefancy} and \ref{lem:aux} that the set of cones $\KK_d=\{\sigma(\pi, \TT):(\pi, \TT) \in \fS_{d+1} \times \sT_{d,d}\}$ in $W_d$ and the set of points $\left\{v_{\pi,\TT}^{(\balpha,\bbeta)}\ :\ (\pi,\TT) \in \fS_{d+1}\times \sT_{d,d}\right\}$ in $U_d^\balpha$ satisfy the hypothesis of Lemma \ref{lem:main_tool}. Hence, we conclude that the first two statements are true, and that the $\KK_d$ induces the normal fan of $\PA(\balpha,\bbeta).$  
	However, since $\KK_d$ contains all the maximal cones in $\Xi_d,$ by Proposition \ref{prop:kapface}, 
	we have that $\Xi_d$ is induced by $\KK_d.$ Therefore, \eqref{itm:nestednormal} follows. 
\end{proof}

We can now prove our main theorem.

\begin{proof}[Proof of Theorem \ref{thm:permassoc}]
	By Proposition \ref{prop:nestedpermasso}~\eqref{itm:nestednormal} and Lemmas \ref{lem:nestedlodayfan}~\eqref{itm:nestedfaceposet} and \ref{lem:anti}, we have that the face poset of $\PA(\balpha,\bbeta)$ is the Kapranov poset $\cK\Pi_d$.
	Hence, we conclude that $\PA(\balpha, \bbeta)$ is a $d$-permuto-associahedron.
	Finally, it follows from Proposition \ref{prop:nestedpermasso}~\eqref{itm:nestednormal} and Lemma \ref{lem:nestedlodayfan}~\eqref{itm:nestedlodaycoarsen} that $\PA(\balpha,\bbeta)$ is a generalized nested permutohedron.
\end{proof}

\subsection{Inequality description for nested permuto-associahedra} 
It follows from Proposition \ref{prop:nestedpermasso} that we can apply Lemma \ref{lem:det-ineq} to find an inequality description for the nested permuto-associahedron $\PA(\balpha,\bbeta)$.

\begin{thm} \label{thm:ineqs_permasso}
	Let $(\balpha, \bbeta) \in \R^{d+1} \times \R^d$ be a $\sT$-appropriate pair of strictly increasing sequences $(\balpha, \bbeta) \in \R^{d+1} \times \R^d$. Suppose $\mathbf{b} \in \R^{\overline{\sO_{d+1}}}$ is defined as follows: for each $\cS = (S_1, S_2,\dots, S_k) \in \overline{\sO_{d+1}},$ if $\Type(\cS)=(t_0, t_1,t_2,\dots, t_k),$ let
	\begin{equation}
		b_{\cS} =  \left( \sum_{i=1}^{k} i \sum_{j=t_{i-1}+1}^{t_i} \alpha_j\right) + \left( \sum_{j=1}^d \beta_j -\sum_{i=1}^k \sum_{j=1}^{|S_i|}\beta_j\right). 
		\label{equ:new_usualb}
	\end{equation} 
	Then we have the following facet-defining inequality description for $\PA(\balpha, \bbeta)$:
	\begin{equation}\label{eq:permasso_ineqs}
		\PA(\balpha, \bbeta) = \left\{\x\in U_d^\balpha \ : \ \langle \e_{\cS}, \x \rangle \leq b_{\cS},\quad \forall  \cS\in\overline{\sO_{d+1}}\right\}.
	\end{equation}
	
\end{thm}

\begin{proof}
	
	Recall $\cS(\pi)$ and $\TT_\cS$ is defined in \eqref{equ:cSpi} and Notation \ref{notn:ST}, respectively. Applying Lemma \ref{lem:det-ineq} together with Proposition \ref{prop:kapface}, Lemma \ref{lem:nestedlodayfan}~\eqref{itm:nestedlodayfantop}\eqref{itm:nestedlodayray}, and Proposition \ref{prop:nestedpermasso}~\eqref{itm:nestedvertexcone} one sees it is left to show that for any $\cS\in\overline{\sO_{d+1}}$, if we choose $(\pi,\TT) \in \fS_{d+1}\times \sT_{d,d}$ such that $(\cS(\pi),\TT) \le_{KP} (\cS, \TT_\cS)$ in the Kapranov poset, then 
	\begin{equation}\label{eq:to_show}
		\left\langle \e_\cS, \v^{(\balpha,\bbeta)}_{\pi,\TT}\right\rangle = b_\cS,
	\end{equation}
	where $b_\cS$ is given by Equation \eqref{equ:new_usualb}.

	We compute from the definitions in Equations \eqref{equ:defnpav} and \eqref{equ:defneT}:
	\begin{align}
		\left\langle \e_\cS, \v^{(\balpha,\bbeta)}_{\pi,\TT}\right\rangle &= \left\langle \sum_{i=1}^k i \e_{S_i}, \sum_{j=1}^{d+1} \alpha_j \e_{\pi^{-1}(j)} + \sum_{j=1}^d \beta_{\TT,j} \ff_{j}^\pi\right\rangle \nonumber \\
		&=  \left\langle \sum_{i=1}^k i\e_{S_i}, \sum_{j=1}^{d+1} \alpha_j \e_{\pi^{-1}(j)}\right\rangle + \left\langle \sum_{i=1}^k i \e_{S_i},  \sum_{j=1}^d \beta_{\TT,j} \ff_{j}^\pi\right\rangle. \label{eq:twosums}
	\end{align}
	We will show that the two terms in \eqref{eq:twosums} are equal to the two terms in \eqref{equ:new_usualb}. 
	
	Note that the leaves of the partition labeled tree $(\cS(\pi),\TT)$ are labeled by $\cS(\pi)$ $=$ $( \{\pi^{-1}(1)\},$ $\{\pi^{-1}(2)\},$ $\ldots,$ $\{\pi^{-1}(d+1)\})$ from left to right. Since $(\cS(\pi),\TT) \le_{KP} (\cS, \TT_\cS)$, by the definition of the covering relation of the Kapranov poset, the followings are true:
	\begin{enumerate}[label=(\roman*)]
		\item	$\cS(\pi)$ refines $\cS$.  Hence, 
		\begin{equation}
			\pi^{-1}(j) \in S_i \text{ if and only if } t_{i-1}+1 \le j \le t_i.
			\label{eq:jinSi}
		\end{equation}
		\item For each $1 \le i \le k,$ there exists a subtree $\TT_i$ of $\TT$ such that the leaves of $\TT_i$ are labeled by $\{ \pi^{-1}(j) \ : \ t_{i-1}+1 \le j \le t_i\}.$ Thus, $\II_\TT(\TT_i)= \{ j \ : \ t_{i-1}+1 \le j \le t_i-1\}.$
	\end{enumerate}
	Clearly, by \eqref{eq:jinSi}, the first term in \eqref{eq:twosums} is equal to the first term in \eqref{equ:new_usualb}. 
	
	Next, applying \eqref{eq:jinSi} again, we obtain that
	$\ds \left\langle \sum_i i\e_{S_i},\ff_{j}^\pi\right\rangle$ is $1$ if $j \in \Type(\cS)$ and is $0$ otherwise.
	Therefore, the second term in \eqref{eq:twosums} is equal to $\ds \sum_{j\in\Type(\cS)}\beta_{\TT,j}$. By condition (ii) above, the set of labels for  internal vertices of $\TT$ that do not appear in any of $\TT_1, \dots, \TT_k$ is exactly $\Type(\cS).$
	Therefore,
	\[
	\sum_{j\in\Type(\cS)}\beta_{\TT,j} = \sum_{j=1}^d \beta_{\TT,j} - \sum_{i=1}^k \sum_{j\in \II_\TT(\TT_i)} \beta_{\TT,j}. 
	\]
	By Lemma \ref{lem:sumval}, we have that 
	\[ \sum_{j=1}^d \beta_{\TT,j} = \sum_{j=1}^d \beta_j \quad \text{and} \quad \sum_{j\in \II_\TT(\TT_i)} \beta_{\TT,j}=\sum_{j=1}^{|S_i|} \beta_j.\]
	Therefore, we conclude that the second term in \eqref{eq:twosums} is equal to the second term in \eqref{equ:new_usualb}, completing the proof.
\end{proof}

Notice in particular that the set of facets of $\PA(\balpha,\bbeta)$ and $\Perm(\balpha,\bbeta)$ are in bijection; in contrast to Corollary \ref{cor:remove}, the nested permuto-associahedron $\PA(\balpha,\bbeta)$ cannot be obtained by removing facets from a nested permutohedron.
We also remark that we do not have a Minkowski sum decomposition as in Corollary \ref{cor:minkowski} for $\PA(\balpha,\bbeta)$.

\section{Comparison to previous work}\label{sec:comparison}

In this section, we highlight some differences and similarities between our realization, Reiner-Ziegler's and Gaiffi's.
  We start by noting a common similarity among all three realizations:
  All constructions can be obtained by symmetrizing an embedding of an $(d-1)$-associahedron in $\R^{d+1}$.
The constructions and proofs are different insomuch as they use different associahedra.

\subsection{Comparison to Reiner-Ziegler's realization}

Reiner and Ziegler's paper \cite{reiner_ziegler} has two distinct parts.
In the first part, they prove that the \emph{dual} of Kapranov's poset can be realized as the face poset of a CW-ball \cite[Theorem 1]{reiner_ziegler}, which they called the \emph{sphericity} theorem. 
Contrary to Kapranov's realization as a CW-ball, Reiner and Ziegler's approach is purely combinatorial.
In the second part of \cite{reiner_ziegler}, they provide the first polytopal realization of Kapranov's poset \cite[Theorem 2]{reiner_ziegler} using methods that independent from the first part.

Interestingly, our approach turns out to be more related to the proof of Reiner and Ziegler's sphericity theorem.
The approach in \cite{reiner_ziegler} to prove sphericity is to glue together cells of the CW-ball arising from the second barycentric subdivision of a simplex. 
See the first row of \cite[Figure 5]{reiner_ziegler}.
The cells of the second barycentric subdivision of a simplex are in natural bijection with the cones of the nested braid fan~\cite[Section 6]{def-cone}.
The fan $\Xi_d$ is a coarsening of the nested braid fan and it groups nested braid cones in the same way that Reiner and Ziegler glue the cells.
In this sense, the present work completes the discussion started in \cite[Section 2]{reiner_ziegler} by showing that the proposed gluing results in a polytope, not just a topological ball.

Below we highlight the differences between our realization and Reiner-Ziegler's:
\begin{enumerate}[leftmargin=*]
	\item Whereas we use Loday's associahedra, 
	  they use a secondary polytope of a specific cyclic polygon.
	\item The resulting realizations in \cite{reiner_ziegler} and  Section \ref{sec:permasso} have different normal fans.
	  In fact, the sets of rays of these two normal fans are different, although for both cases a ray is constructed for each ordered set partition: 
	For each $\cS=(S_1,\dots,S_k) \in \OO_{d+1}$, we associate to $\cS$ a ray spanned by $\ee_\cS=\sum_{i=1}^k i\ee_{S_i}\in W_d$, and Reiner and Ziegler associate to $\cS$ a ray spanned by $\sum_{i=1}^k (t_i+t_{i-1})\ee_{S_i}\in W_d$, where $(t_1,\dots,t_{k-1})$ is the type of $\cS$ and by convention $t_0=0$ and $t_k=d+1$.
	It follows from above descriptions of rays that the two normal fans are not even linearly equivalent.
	\item Reiner and Ziegler's construction is surprisingly \emph{inscribable} (all vertices lie on a sphere), whereas ours never is.
	Indeed any nested permuto-associahedron $\PA(\balpha,\bbeta)$ of dimension greater than one will have a Loday pentagon as a face and these pentagons are never inscribable.
\end{enumerate}
We remark that Reiner and Ziegler also extended their construction to both type-$B$ and type-$D$, and showed that the permuto-associahedron (which correspond to type-$A$) arises as a facet of their type-$B$ version.

\subsection{Comparison to Gaiffi} 

The second realization of the permuto-associahedron was given by Gaiffi \cite{gaiffi}.
His construction is quite general: he constructed permutonestohedra for any nestohedron \cite[Section 7]{bible} of which the associahedron is an example.
Furthermore, he does it for general root systems.
When the root system is of type-$A$ and the nestohedron is the Stasheff-Shnider associahedron, his construction becomes a realization of the permuto-associahedron.
We remark that when the root system is of type-$B$, Gaiffi's version is different from Reiner and Ziegler's.

Even though Gaiffi's approach and ours are manifestly different, a closer inspection reveals that the rays for the normal fan of his construction are generated by the vectors $\ee_{\cS}$ we defined in Section \ref{sec:nested} for ordered set partitions $\cS$, and the rays form the maximal cones in the fan in exactly the same we as in our realization. It follows that Gaiffi's and ours permuto-associahedra have the same normal fan. This is not surprising, as Gaiffi's starting point is Stasheff-Shnider's construction for associahedra and we start with Loday's construction, and Stasheff-Shnider's and Loday's associahedra have the same normal fan. 

Below we list a few more differences between the approaches in our realization in this article and Gaiffi's \cite{gaiffi}:
\begin{enumerate}
  \item Gaiffi starts by choosing a \emph{suitable} list $\varepsilon_1<\dots<\varepsilon_d$ of real numbers.
    We use an \emph{appropriate} pair of $(\balpha,\bbeta)\in\R^{d+1}\times\R^d$. 
    Gaiffi's definition of suitable requires 
that each one is sufficiently larger than the previous one; see \cite[Definition 3.1]{gaiffi}.
Our definition of appropriateness is a bit more flexible: By Remark \ref{rem:appropriate}, as long as $\balpha\in\R^{d+1}$ is fixed, any increasing sequence $\bbeta\in\R^d$ is appropriate if every entry is smaller than a global constant. So it turns out that our flexibility in the choice of $\balpha$ allows us to relax the conditions on the choice of $\bbeta$.

\item Gaiffi describes his realization by providing inequality description, and then describing vertices as intersections of $d$ facets (even though they will eventually be contained in more facets).
  However, he did not provide explicit description for vertex coordinates of his permuto-associahedra. We start by explicitly constructing vertices and the normal fan of our permuto-associahedra, and then provide explicity inequality description as a consequence of our method. 

\item For each ordered set partition $\cS=(S_1,\dots,S_k)$ with type $(t_1, t_2, \dots, t_{k-1})$, as mentioned above Gaiffi and we both associate to a same normal vector $\ee_\cS.$ The corresponding inequality in the inequality description of Gaiffi's permuto-associahedron is
  \begin{equation}\label{eq:gaiffi}
	\langle \ee_{\cS},\x\rangle \leq \varepsilon_d-(\varepsilon_{|S_1|}+\dots+\varepsilon_{|S_k|}),
\end{equation}
and ours is (by Theorem \ref{thm:ineqs_permasso})
\begin{equation}\label{eq:us}
  \langle \ee_{\cS},\x\rangle \leq \left( \sum_{i=1}^{k} i \sum_{j=t_{i-1}+1}^{t_i} \alpha_j\right) + \left( \sum_{j=1}^d \beta_j -\sum_{i=1}^k \sum_{j=1}^{|S_i|}\beta_j\right).
\end{equation}
The right hand sides of both inequalities depend on sizes of blocks in $\cS$. However, \eqref{eq:us} depends on the \emph{orders} of the blocks in $\cS$ but \eqref{eq:gaiffi} does not.

\end{enumerate}

\begin{ques}
What is the set of all three-dimensional permuto-associahedra that arise both in Gaiffi's and in our construction?
It seems that the areas of the pentagonal faces relative to the areas of the square faces behave differently in both constructions.
Gaiffi's realization in dimension $3$ is depicted in the left hand side of \cite[Figure 5]{gaiffi} where the pentagons are large in comparison to the little square faces, whereas in Figure \ref{fig:permasso} our pentagons are small with respect to the permutohedron, and indeed in our constructions they can be arbitrarily small compare to the other faces.
\end{ques}

There are more relations to be explored.
Because Gaiffi's permuto-associahedron has the same normal fan as ours, which is a {generalized nested permutohedron}, maybe all of Gaiffi's permutonestohedra are {generalized nested permutohedra} as well, in other words, that their normal fans are all coarsening of the nested braid fan \cite{def-cone}. 
Coincidentally, what we call nested permutohedron is what Gaiffi calls \emph{permutopermutohedron}.

\bibliography{biblio}
\bibliographystyle{plain}

\end{document}

%% file: tikz/nested3.tex
\begin{tikzpicture}%
	[x={(0.178656cm, 0.726606cm)},
	y={(-0.806737cm, 0.494173cm)},
	z={(0.563256cm, 0.477323cm)},
	scale=0.12000,
	back/.style={loosely dotted, thin},
	edge/.style={color=black, thick},
	facet/.style={fill=white,fill opacity=0.800000},
	vertex/.style={inner sep=1pt,circle,draw=black!25!black,fill=black!75!black,thick,anchor=base}]
%
%
\coordinate (-18, -7, -1) at (-18, -7, -1);
\coordinate (-18, -7, 1) at (-18, -7, 1);
\coordinate (-18, -1, -7) at (-18, -1, -7);
\coordinate (-18, -1, 7) at (-18, -1, 7);
\coordinate (-18, 1, -7) at (-18, 1, -7);
\coordinate (-18, 1, 7) at (-18, 1, 7);
\coordinate (-18, 7, -1) at (-18, 7, -1);
\coordinate (-18, 7, 1) at (-18, 7, 1);
\coordinate (-17, -9, -2) at (-17, -9, -2);
\coordinate (-17, -9, 2) at (-17, -9, 2);
\coordinate (-17, -2, -9) at (-17, -2, -9);
\coordinate (-17, -2, 9) at (-17, -2, 9);
\coordinate (-17, 2, -9) at (-17, 2, -9);
\coordinate (-17, 2, 9) at (-17, 2, 9);
\coordinate (-17, 9, -2) at (-17, 9, -2);
\coordinate (-17, 9, 2) at (-17, 9, 2);
\coordinate (-16, -11, -1) at (-16, -11, -1);
\coordinate (-16, -11, 1) at (-16, -11, 1);
\coordinate (-16, -1, -11) at (-16, -1, -11);
\coordinate (-16, -1, 11) at (-16, -1, 11);
\coordinate (-16, 1, -11) at (-16, 1, -11);
\coordinate (-16, 1, 11) at (-16, 1, 11);
\coordinate (-16, 11, -1) at (-16, 11, -1);
\coordinate (-16, 11, 1) at (-16, 11, 1);
\coordinate (-11, -16, -1) at (-11, -16, -1);
\coordinate (-11, -16, 1) at (-11, -16, 1);
\coordinate (-11, -1, -16) at (-11, -1, -16);
\coordinate (-11, -1, 16) at (-11, -1, 16);
\coordinate (-11, 1, -16) at (-11, 1, -16);
\coordinate (-11, 1, 16) at (-11, 1, 16);
\coordinate (-11, 16, -1) at (-11, 16, -1);
\coordinate (-11, 16, 1) at (-11, 16, 1);
\coordinate (-9, -17, -2) at (-9, -17, -2);
\coordinate (-9, -17, 2) at (-9, -17, 2);
\coordinate (-9, -2, -17) at (-9, -2, -17);
\coordinate (-9, -2, 17) at (-9, -2, 17);
\coordinate (-9, 2, -17) at (-9, 2, -17);
\coordinate (-9, 2, 17) at (-9, 2, 17);
\coordinate (-9, 17, -2) at (-9, 17, -2);
\coordinate (-9, 17, 2) at (-9, 17, 2);
\coordinate (-7, -18, -1) at (-7, -18, -1);
\coordinate (-7, -18, 1) at (-7, -18, 1);
\coordinate (-7, -1, -18) at (-7, -1, -18);
\coordinate (-7, -1, 18) at (-7, -1, 18);
\coordinate (-7, 1, -18) at (-7, 1, -18);
\coordinate (-7, 1, 18) at (-7, 1, 18);
\coordinate (-7, 18, -1) at (-7, 18, -1);
\coordinate (-7, 18, 1) at (-7, 18, 1);
\coordinate (-2, -17, -9) at (-2, -17, -9);
\coordinate (-2, -17, 9) at (-2, -17, 9);
\coordinate (-2, -9, -17) at (-2, -9, -17);
\coordinate (-2, -9, 17) at (-2, -9, 17);
\coordinate (-2, 9, -17) at (-2, 9, -17);
\coordinate (-2, 9, 17) at (-2, 9, 17);
\coordinate (-2, 17, -9) at (-2, 17, -9);
\coordinate (-2, 17, 9) at (-2, 17, 9);
\coordinate (-1, -18, -7) at (-1, -18, -7);
\coordinate (-1, -18, 7) at (-1, -18, 7);
\coordinate (-1, -16, -11) at (-1, -16, -11);
\coordinate (-1, -16, 11) at (-1, -16, 11);
\coordinate (-1, -11, -16) at (-1, -11, -16);
\coordinate (-1, -11, 16) at (-1, -11, 16);
\coordinate (-1, -7, -18) at (-1, -7, -18);
\coordinate (-1, -7, 18) at (-1, -7, 18);
\coordinate (-1, 7, -18) at (-1, 7, -18);
\coordinate (-1, 7, 18) at (-1, 7, 18);
\coordinate (-1, 11, -16) at (-1, 11, -16);
\coordinate (-1, 11, 16) at (-1, 11, 16);
\coordinate (-1, 16, -11) at (-1, 16, -11);
\coordinate (-1, 16, 11) at (-1, 16, 11);
\coordinate (-1, 18, -7) at (-1, 18, -7);
\coordinate (-1, 18, 7) at (-1, 18, 7);
\coordinate (1, -18, -7) at (1, -18, -7);
\coordinate (1, -18, 7) at (1, -18, 7);
\coordinate (1, -16, -11) at (1, -16, -11);
\coordinate (1, -16, 11) at (1, -16, 11);
\coordinate (1, -11, -16) at (1, -11, -16);
\coordinate (1, -11, 16) at (1, -11, 16);
\coordinate (1, -7, -18) at (1, -7, -18);
\coordinate (1, -7, 18) at (1, -7, 18);
\coordinate (1, 7, -18) at (1, 7, -18);
\coordinate (1, 7, 18) at (1, 7, 18);
\coordinate (1, 11, -16) at (1, 11, -16);
\coordinate (1, 11, 16) at (1, 11, 16);
\coordinate (1, 16, -11) at (1, 16, -11);
\coordinate (1, 16, 11) at (1, 16, 11);
\coordinate (1, 18, -7) at (1, 18, -7);
\coordinate (1, 18, 7) at (1, 18, 7);
\coordinate (2, -17, -9) at (2, -17, -9);
\coordinate (2, -17, 9) at (2, -17, 9);
\coordinate (2, -9, -17) at (2, -9, -17);
\coordinate (2, -9, 17) at (2, -9, 17);
\coordinate (2, 9, -17) at (2, 9, -17);
\coordinate (2, 9, 17) at (2, 9, 17);
\coordinate (2, 17, -9) at (2, 17, -9);
\coordinate (2, 17, 9) at (2, 17, 9);
\coordinate (7, -18, -1) at (7, -18, -1);
\coordinate (7, -18, 1) at (7, -18, 1);
\coordinate (7, -1, -18) at (7, -1, -18);
\coordinate (7, -1, 18) at (7, -1, 18);
\coordinate (7, 1, -18) at (7, 1, -18);
\coordinate (7, 1, 18) at (7, 1, 18);
\coordinate (7, 18, -1) at (7, 18, -1);
\coordinate (7, 18, 1) at (7, 18, 1);
\coordinate (9, -17, -2) at (9, -17, -2);
\coordinate (9, -17, 2) at (9, -17, 2);
\coordinate (9, -2, -17) at (9, -2, -17);
\coordinate (9, -2, 17) at (9, -2, 17);
\coordinate (9, 2, -17) at (9, 2, -17);
\coordinate (9, 2, 17) at (9, 2, 17);
\coordinate (9, 17, -2) at (9, 17, -2);
\coordinate (9, 17, 2) at (9, 17, 2);
\coordinate (11, -16, -1) at (11, -16, -1);
\coordinate (11, -16, 1) at (11, -16, 1);
\coordinate (11, -1, -16) at (11, -1, -16);
\coordinate (11, -1, 16) at (11, -1, 16);
\coordinate (11, 1, -16) at (11, 1, -16);
\coordinate (11, 1, 16) at (11, 1, 16);
\coordinate (11, 16, -1) at (11, 16, -1);
\coordinate (11, 16, 1) at (11, 16, 1);
\coordinate (16, -11, -1) at (16, -11, -1);
\coordinate (16, -11, 1) at (16, -11, 1);
\coordinate (16, -1, -11) at (16, -1, -11);
\coordinate (16, -1, 11) at (16, -1, 11);
\coordinate (16, 1, -11) at (16, 1, -11);
\coordinate (16, 1, 11) at (16, 1, 11);
\coordinate (16, 11, -1) at (16, 11, -1);
\coordinate (16, 11, 1) at (16, 11, 1);
\coordinate (17, -9, -2) at (17, -9, -2);
\coordinate (17, -9, 2) at (17, -9, 2);
\coordinate (17, -2, -9) at (17, -2, -9);
\coordinate (17, -2, 9) at (17, -2, 9);
\coordinate (17, 2, -9) at (17, 2, -9);
\coordinate (17, 2, 9) at (17, 2, 9);
\coordinate (17, 9, -2) at (17, 9, -2);
\coordinate (17, 9, 2) at (17, 9, 2);
\coordinate (18, -7, -1) at (18, -7, -1);
\coordinate (18, -7, 1) at (18, -7, 1);
\coordinate (18, -1, -7) at (18, -1, -7);
\coordinate (18, -1, 7) at (18, -1, 7);
\coordinate (18, 1, -7) at (18, 1, -7);
\coordinate (18, 1, 7) at (18, 1, 7);
\coordinate (18, 7, -1) at (18, 7, -1);
\coordinate (18, 7, 1) at (18, 7, 1);
\draw[edge,back] (-16, -1, -11) -- (-11, -1, -16);
\draw[edge,back] (-11, -16, -1) -- (-9, -17, -2);
\draw[edge,back] (-11, -1, -16) -- (-11, 1, -16);
\draw[edge,back] (-11, -1, -16) -- (-9, -2, -17);
\draw[edge,back] (-9, -17, -2) -- (-7, -18, -1);
\draw[edge,back] (-9, -17, -2) -- (-2, -17, -9);
\draw[edge,back] (-9, -2, -17) -- (-7, -1, -18);
\draw[edge,back] (-9, -2, -17) -- (-2, -9, -17);
\draw[edge,back] (-9, 2, -17) -- (-7, 1, -18);
\draw[edge,back] (-7, -18, -1) -- (-7, -18, 1);
\draw[edge,back] (-7, -18, -1) -- (-1, -18, -7);
\draw[edge,back] (-7, -1, -18) -- (-7, 1, -18);
\draw[edge,back] (-7, -1, -18) -- (-1, -7, -18);
\draw[edge,back] (-7, 1, -18) -- (-1, 7, -18);
\draw[edge,back] (-2, -17, -9) -- (-1, -18, -7);
\draw[edge,back] (-2, -17, -9) -- (-1, -16, -11);
\draw[edge,back] (-2, -9, -17) -- (-1, -11, -16);
\draw[edge,back] (-2, -9, -17) -- (-1, -7, -18);
\draw[edge,back] (-2, 9, -17) -- (-1, 7, -18);
\draw[edge,back] (-1, -18, -7) -- (1, -18, -7);
\draw[edge,back] (-1, -16, -11) -- (-1, -11, -16);
\draw[edge,back] (-1, -16, -11) -- (1, -16, -11);
\draw[edge,back] (-1, -11, -16) -- (1, -11, -16);
\draw[edge,back] (-1, -7, -18) -- (1, -7, -18);
\draw[edge,back] (-1, 7, -18) -- (1, 7, -18);
\draw[edge,back] (-1, 11, -16) -- (1, 11, -16);
\draw[edge,back] (-1, 16, -11) -- (1, 16, -11);
\draw[edge,back] (1, -18, -7) -- (2, -17, -9);
\draw[edge,back] (1, -18, -7) -- (7, -18, -1);
\draw[edge,back] (1, -18, 7) -- (7, -18, 1);
\draw[edge,back] (1, -16, -11) -- (1, -11, -16);
\draw[edge,back] (1, -16, -11) -- (2, -17, -9);
\draw[edge,back] (1, -11, -16) -- (2, -9, -17);
\draw[edge,back] (1, -7, -18) -- (2, -9, -17);
\draw[edge,back] (1, -7, -18) -- (7, -1, -18);
\draw[edge,back] (1, 7, -18) -- (2, 9, -17);
\draw[edge,back] (1, 7, -18) -- (7, 1, -18);
\draw[edge,back] (1, 11, -16) -- (1, 16, -11);
\draw[edge,back] (1, 11, -16) -- (2, 9, -17);
\draw[edge,back] (1, 16, -11) -- (2, 17, -9);
\draw[edge,back] (1, 18, -7) -- (2, 17, -9);
\draw[edge,back] (2, -17, -9) -- (9, -17, -2);
\draw[edge,back] (2, -17, 9) -- (9, -17, 2);
\draw[edge,back] (2, -9, -17) -- (9, -2, -17);
\draw[edge,back] (2, 9, -17) -- (9, 2, -17);
\draw[edge,back] (2, 17, -9) -- (9, 17, -2);
\draw[edge,back] (7, -18, -1) -- (7, -18, 1);
\draw[edge,back] (7, -18, -1) -- (9, -17, -2);
\draw[edge,back] (7, -18, 1) -- (9, -17, 2);
\draw[edge,back] (7, -1, -18) -- (7, 1, -18);
\draw[edge,back] (7, -1, -18) -- (9, -2, -17);
\draw[edge,back] (7, 1, -18) -- (9, 2, -17);
\draw[edge,back] (7, 18, -1) -- (9, 17, -2);
\draw[edge,back] (9, -17, -2) -- (11, -16, -1);
\draw[edge,back] (9, -17, 2) -- (11, -16, 1);
\draw[edge,back] (9, -2, -17) -- (11, -1, -16);
\draw[edge,back] (9, 2, -17) -- (11, 1, -16);
\draw[edge,back] (9, 17, -2) -- (11, 16, -1);
\draw[edge,back] (11, -16, -1) -- (11, -16, 1);
\draw[edge,back] (11, -16, -1) -- (16, -11, -1);
\draw[edge,back] (11, -16, 1) -- (16, -11, 1);
\draw[edge,back] (11, -1, -16) -- (11, 1, -16);
\draw[edge,back] (11, -1, -16) -- (16, -1, -11);
\draw[edge,back] (11, -1, 16) -- (16, -1, 11);
\draw[edge,back] (11, 1, -16) -- (16, 1, -11);
\draw[edge,back] (11, 16, -1) -- (11, 16, 1);
\draw[edge,back] (11, 16, -1) -- (16, 11, -1);
\draw[edge,back] (16, -11, -1) -- (16, -11, 1);
\draw[edge,back] (16, -11, -1) -- (17, -9, -2);
\draw[edge,back] (16, -11, 1) -- (17, -9, 2);
\draw[edge,back] (16, -1, -11) -- (16, 1, -11);
\draw[edge,back] (16, -1, -11) -- (17, -2, -9);
\draw[edge,back] (16, -1, 11) -- (16, 1, 11);
\draw[edge,back] (16, -1, 11) -- (17, -2, 9);
\draw[edge,back] (16, 1, -11) -- (17, 2, -9);
\draw[edge,back] (16, 11, -1) -- (16, 11, 1);
\draw[edge,back] (16, 11, -1) -- (17, 9, -2);
\draw[edge,back] (17, -9, -2) -- (17, -2, -9);
\draw[edge,back] (17, -9, -2) -- (18, -7, -1);
\draw[edge,back] (17, -9, 2) -- (17, -2, 9);
\draw[edge,back] (17, -9, 2) -- (18, -7, 1);
\draw[edge,back] (17, -2, -9) -- (18, -1, -7);
\draw[edge,back] (17, -2, 9) -- (18, -1, 7);
\draw[edge,back] (17, 2, -9) -- (17, 9, -2);
\draw[edge,back] (17, 2, -9) -- (18, 1, -7);
\draw[edge,back] (17, 2, 9) -- (18, 1, 7);
\draw[edge,back] (17, 9, -2) -- (18, 7, -1);
\draw[edge,back] (17, 9, 2) -- (18, 7, 1);
\draw[edge,back] (18, -7, -1) -- (18, -7, 1);
\draw[edge,back] (18, -7, -1) -- (18, -1, -7);
\draw[edge,back] (18, -7, 1) -- (18, -1, 7);
\draw[edge,back] (18, -1, -7) -- (18, 1, -7);
\draw[edge,back] (18, -1, 7) -- (18, 1, 7);
\draw[edge,back] (18, 1, -7) -- (18, 7, -1);
\draw[edge,back] (18, 1, 7) -- (18, 7, 1);
\draw[edge,back] (18, 7, -1) -- (18, 7, 1);
\node[vertex] at (18, -7, -1)     {};
\node[vertex] at (18, -7, 1)     {};
\node[vertex] at (18, -1, -7)     {};
\node[vertex] at (18, -1, 7)     {};
\node[vertex] at (18, 1, -7)     {};
\node[vertex] at (18, 1, 7)     {};
\node[vertex] at (18, 7, -1)     {};
\node[vertex] at (18, 7, 1)     {};
\node[vertex] at (16, 11, -1)     {};
\node[vertex] at (17, 9, -2)     {};
\node[vertex] at (16, -1, -11)     {};
\node[vertex] at (16, 1, -11)     {};
\node[vertex] at (17, -2, -9)     {};
\node[vertex] at (17, 2, -9)     {};
\node[vertex] at (-11, -1, -16)     {};
\node[vertex] at (-9, -17, -2)     {};
\node[vertex] at (-9, -2, -17)     {};
\node[vertex] at (-2, -17, -9)     {};
\node[vertex] at (-2, -9, -17)     {};
\node[vertex] at (-1, -16, -11)     {};
\node[vertex] at (-1, -11, -16)     {};
\node[vertex] at (-7, -18, -1)     {};
\node[vertex] at (-7, -1, -18)     {};
\node[vertex] at (-1, -7, -18)     {};
\node[vertex] at (-7, 1, -18)     {};
\node[vertex] at (-1, 7, -18)     {};
\node[vertex] at (-1, -18, -7)     {};
\node[vertex] at (1, 16, -11)     {};
\node[vertex] at (2, 17, -9)     {};
\node[vertex] at (1, -11, -16)     {};
\node[vertex] at (1, -7, -18)     {};
\node[vertex] at (2, -9, -17)     {};
\node[vertex] at (1, -18, -7)     {};
\node[vertex] at (1, -16, -11)     {};
\node[vertex] at (2, -17, -9)     {};
\node[vertex] at (1, 7, -18)     {};
\node[vertex] at (1, 11, -16)     {};
\node[vertex] at (2, 9, -17)     {};
\node[vertex] at (7, -18, -1)     {};
\node[vertex] at (7, -18, 1)     {};
\node[vertex] at (9, -17, 2)     {};
\node[vertex] at (11, -16, 1)     {};
\node[vertex] at (16, -11, 1)     {};
\node[vertex] at (16, -1, 11)     {};
\node[vertex] at (17, -9, 2)     {};
\node[vertex] at (17, -2, 9)     {};
\node[vertex] at (9, -17, -2)     {};
\node[vertex] at (7, -1, -18)     {};
\node[vertex] at (9, -2, -17)     {};
\node[vertex] at (7, 1, -18)     {};
\node[vertex] at (9, 2, -17)     {};
\node[vertex] at (11, -16, -1)     {};
\node[vertex] at (9, 17, -2)     {};
\node[vertex] at (11, -1, -16)     {};
\node[vertex] at (11, 1, -16)     {};
\node[vertex] at (16, -11, -1)     {};
\node[vertex] at (17, -9, -2)     {};
\node[vertex] at (11, 16, -1)     {};
\fill[facet] (9, 17, 2) -- (2, 17, 9) -- (1, 18, 7) -- (7, 18, 1) -- cycle {};
\fill[facet,fill=red] (-16, -11, 1) -- (-17, -9, 2) -- (-18, -7, 1) -- (-18, -7, -1) -- (-17, -9, -2) -- (-16, -11, -1) -- cycle {};
\fill[facet] (-17, -2, 9) -- (-18, -1, 7) -- (-18, -7, 1) -- (-17, -9, 2) -- cycle {};
\fill[facet] (-17, 9, -2) -- (-18, 7, -1) -- (-18, 1, -7) -- (-17, 2, -9) -- cycle {};
\fill[facet,fill=red] (-16, 1, 11) -- (-17, 2, 9) -- (-18, 1, 7) -- (-18, -1, 7) -- (-17, -2, 9) -- (-16, -1, 11) -- cycle {};
\fill[facet,fill=red] (-16, 1, -11) -- (-17, 2, -9) -- (-18, 1, -7) -- (-18, -1, -7) -- (-17, -2, -9) -- (-16, -1, -11) -- cycle {};
\fill[facet,fill=red] (2, -17, 9) -- (1, -18, 7) -- (-1, -18, 7) -- (-2, -17, 9) -- (-1, -16, 11) -- (1, -16, 11) -- cycle {};
\fill[facet] (-17, 9, 2) -- (-18, 7, 1) -- (-18, 1, 7) -- (-17, 2, 9) -- cycle {};
\fill[facet] (-18, 7, 1) -- (-18, 1, 7) -- (-18, -1, 7) -- (-18, -7, 1) -- (-18, -7, -1) -- (-18, -1, -7) -- (-18, 1, -7) -- (-18, 7, -1) -- cycle {};
\fill[facet] (-11, 1, 16) -- (-16, 1, 11) -- (-16, -1, 11) -- (-11, -1, 16) -- cycle {};
\fill[facet,fill=red] (-16, 11, 1) -- (-17, 9, 2) -- (-18, 7, 1) -- (-18, 7, -1) -- (-17, 9, -2) -- (-16, 11, -1) -- cycle {};
\fill[facet] (-11, -16, 1) -- (-16, -11, 1) -- (-16, -11, -1) -- (-11, -16, -1) -- cycle {};
\fill[facet] (-1, -18, 7) -- (-7, -18, 1) -- (-9, -17, 2) -- (-2, -17, 9) -- cycle {};
\fill[facet] (-11, 16, 1) -- (-16, 11, 1) -- (-16, 11, -1) -- (-11, 16, -1) -- cycle {};
\fill[facet,fill=red] (-7, 18, 1) -- (-9, 17, 2) -- (-11, 16, 1) -- (-11, 16, -1) -- (-9, 17, -2) -- (-7, 18, -1) -- cycle {};
\fill[facet,fill=red] (-7, 1, 18) -- (-9, 2, 17) -- (-11, 1, 16) -- (-11, -1, 16) -- (-9, -2, 17) -- (-7, -1, 18) -- cycle {};
\fill[facet] (-1, 16, 11) -- (-2, 17, 9) -- (-9, 17, 2) -- (-11, 16, 1) -- (-16, 11, 1) -- (-17, 9, 2) -- (-17, 2, 9) -- (-16, 1, 11) -- (-11, 1, 16) -- (-9, 2, 17) -- (-2, 9, 17) -- (-1, 11, 16) -- cycle {};
\fill[facet] (-1, 18, -7) -- (-7, 18, -1) -- (-9, 17, -2) -- (-2, 17, -9) -- cycle {};
\fill[facet] (-1, -11, 16) -- (-2, -9, 17) -- (-9, -2, 17) -- (-11, -1, 16) -- (-16, -1, 11) -- (-17, -2, 9) -- (-17, -9, 2) -- (-16, -11, 1) -- (-11, -16, 1) -- (-9, -17, 2) -- (-2, -17, 9) -- (-1, -16, 11) -- cycle {};
\fill[facet] (-1, -7, 18) -- (-7, -1, 18) -- (-9, -2, 17) -- (-2, -9, 17) -- cycle {};
\fill[facet] (-1, 7, 18) -- (-7, 1, 18) -- (-9, 2, 17) -- (-2, 9, 17) -- cycle {};
\fill[facet,fill=red] (2, 9, 17) -- (1, 7, 18) -- (-1, 7, 18) -- (-2, 9, 17) -- (-1, 11, 16) -- (1, 11, 16) -- cycle {};
\fill[facet] (-1, 16, -11) -- (-2, 17, -9) -- (-9, 17, -2) -- (-11, 16, -1) -- (-16, 11, -1) -- (-17, 9, -2) -- (-17, 2, -9) -- (-16, 1, -11) -- (-11, 1, -16) -- (-9, 2, -17) -- (-2, 9, -17) -- (-1, 11, -16) -- cycle {};
\fill[facet] (-1, 18, 7) -- (-7, 18, 1) -- (-9, 17, 2) -- (-2, 17, 9) -- cycle {};
\fill[facet] (1, -11, 16) -- (-1, -11, 16) -- (-1, -16, 11) -- (1, -16, 11) -- cycle {};
\fill[facet] (-17, -2, -9) -- (-18, -1, -7) -- (-18, -7, -1) -- (-17, -9, -2) -- cycle {};
\fill[facet,fill=red] (2, -9, 17) -- (1, -11, 16) -- (-1, -11, 16) -- (-2, -9, 17) -- (-1, -7, 18) -- (1, -7, 18) -- cycle {};
\fill[facet] (7, 1, 18) -- (1, 7, 18) -- (-1, 7, 18) -- (-7, 1, 18) -- (-7, -1, 18) -- (-1, -7, 18) -- (1, -7, 18) -- (7, -1, 18) -- cycle {};
\fill[facet] (1, 16, 11) -- (-1, 16, 11) -- (-1, 11, 16) -- (1, 11, 16) -- cycle {};
\fill[facet] (9, -2, 17) -- (2, -9, 17) -- (1, -7, 18) -- (7, -1, 18) -- cycle {};
\fill[facet] (9, 2, 17) -- (2, 9, 17) -- (1, 7, 18) -- (7, 1, 18) -- cycle {};
\fill[facet,fill=red] (2, 17, 9) -- (1, 16, 11) -- (-1, 16, 11) -- (-2, 17, 9) -- (-1, 18, 7) -- (1, 18, 7) -- cycle {};
\fill[facet] (7, 18, 1) -- (1, 18, 7) -- (-1, 18, 7) -- (-7, 18, 1) -- (-7, 18, -1) -- (-1, 18, -7) -- (1, 18, -7) -- (7, 18, -1) -- cycle {};
\fill[facet,fill=red] (11, 1, 16) -- (9, 2, 17) -- (7, 1, 18) -- (7, -1, 18) -- (9, -2, 17) -- (11, -1, 16) -- cycle {};
\fill[facet] (17, 9, 2) -- (16, 11, 1) -- (11, 16, 1) -- (9, 17, 2) -- (2, 17, 9) -- (1, 16, 11) -- (1, 11, 16) -- (2, 9, 17) -- (9, 2, 17) -- (11, 1, 16) -- (16, 1, 11) -- (17, 2, 9) -- cycle {};
\draw[edge] (-18, -7, -1) -- (-18, -7, 1);
\draw[edge] (-18, -7, -1) -- (-18, -1, -7);
\draw[edge] (-18, -7, -1) -- (-17, -9, -2);
\draw[edge] (-18, -7, 1) -- (-18, -1, 7);
\draw[edge] (-18, -7, 1) -- (-17, -9, 2);
\draw[edge] (-18, -1, -7) -- (-18, 1, -7);
\draw[edge] (-18, -1, -7) -- (-17, -2, -9);
\draw[edge] (-18, -1, 7) -- (-18, 1, 7);
\draw[edge] (-18, -1, 7) -- (-17, -2, 9);
\draw[edge] (-18, 1, -7) -- (-18, 7, -1);
\draw[edge] (-18, 1, -7) -- (-17, 2, -9);
\draw[edge] (-18, 1, 7) -- (-18, 7, 1);
\draw[edge] (-18, 1, 7) -- (-17, 2, 9);
\draw[edge] (-18, 7, -1) -- (-18, 7, 1);
\draw[edge] (-18, 7, -1) -- (-17, 9, -2);
\draw[edge] (-18, 7, 1) -- (-17, 9, 2);
\draw[edge] (-17, -9, -2) -- (-17, -2, -9);
\draw[edge] (-17, -9, -2) -- (-16, -11, -1);
\draw[edge] (-17, -9, 2) -- (-17, -2, 9);
\draw[edge] (-17, -9, 2) -- (-16, -11, 1);
\draw[edge] (-17, -2, -9) -- (-16, -1, -11);
\draw[edge] (-17, -2, 9) -- (-16, -1, 11);
\draw[edge] (-17, 2, -9) -- (-17, 9, -2);
\draw[edge] (-17, 2, -9) -- (-16, 1, -11);
\draw[edge] (-17, 2, 9) -- (-17, 9, 2);
\draw[edge] (-17, 2, 9) -- (-16, 1, 11);
\draw[edge] (-17, 9, -2) -- (-16, 11, -1);
\draw[edge] (-17, 9, 2) -- (-16, 11, 1);
\draw[edge] (-16, -11, -1) -- (-16, -11, 1);
\draw[edge] (-16, -11, -1) -- (-11, -16, -1);
\draw[edge] (-16, -11, 1) -- (-11, -16, 1);
\draw[edge] (-16, -1, -11) -- (-16, 1, -11);
\draw[edge] (-16, -1, 11) -- (-16, 1, 11);
\draw[edge] (-16, -1, 11) -- (-11, -1, 16);
\draw[edge] (-16, 1, -11) -- (-11, 1, -16);
\draw[edge] (-16, 1, 11) -- (-11, 1, 16);
\draw[edge] (-16, 11, -1) -- (-16, 11, 1);
\draw[edge] (-16, 11, -1) -- (-11, 16, -1);
\draw[edge] (-16, 11, 1) -- (-11, 16, 1);
\draw[edge] (-11, -16, -1) -- (-11, -16, 1);
\draw[edge] (-11, -16, 1) -- (-9, -17, 2);
\draw[edge] (-11, -1, 16) -- (-11, 1, 16);
\draw[edge] (-11, -1, 16) -- (-9, -2, 17);
\draw[edge] (-11, 1, -16) -- (-9, 2, -17);
\draw[edge] (-11, 1, 16) -- (-9, 2, 17);
\draw[edge] (-11, 16, -1) -- (-11, 16, 1);
\draw[edge] (-11, 16, -1) -- (-9, 17, -2);
\draw[edge] (-11, 16, 1) -- (-9, 17, 2);
\draw[edge] (-9, -17, 2) -- (-7, -18, 1);
\draw[edge] (-9, -17, 2) -- (-2, -17, 9);
\draw[edge] (-9, -2, 17) -- (-7, -1, 18);
\draw[edge] (-9, -2, 17) -- (-2, -9, 17);
\draw[edge] (-9, 2, -17) -- (-2, 9, -17);
\draw[edge] (-9, 2, 17) -- (-7, 1, 18);
\draw[edge] (-9, 2, 17) -- (-2, 9, 17);
\draw[edge] (-9, 17, -2) -- (-7, 18, -1);
\draw[edge] (-9, 17, -2) -- (-2, 17, -9);
\draw[edge] (-9, 17, 2) -- (-7, 18, 1);
\draw[edge] (-9, 17, 2) -- (-2, 17, 9);
\draw[edge] (-7, -18, 1) -- (-1, -18, 7);
\draw[edge] (-7, -1, 18) -- (-7, 1, 18);
\draw[edge] (-7, -1, 18) -- (-1, -7, 18);
\draw[edge] (-7, 1, 18) -- (-1, 7, 18);
\draw[edge] (-7, 18, -1) -- (-7, 18, 1);
\draw[edge] (-7, 18, -1) -- (-1, 18, -7);
\draw[edge] (-7, 18, 1) -- (-1, 18, 7);
\draw[edge] (-2, -17, 9) -- (-1, -18, 7);
\draw[edge] (-2, -17, 9) -- (-1, -16, 11);
\draw[edge] (-2, -9, 17) -- (-1, -11, 16);
\draw[edge] (-2, -9, 17) -- (-1, -7, 18);
\draw[edge] (-2, 9, -17) -- (-1, 11, -16);
\draw[edge] (-2, 9, 17) -- (-1, 7, 18);
\draw[edge] (-2, 9, 17) -- (-1, 11, 16);
\draw[edge] (-2, 17, -9) -- (-1, 16, -11);
\draw[edge] (-2, 17, -9) -- (-1, 18, -7);
\draw[edge] (-2, 17, 9) -- (-1, 16, 11);
\draw[edge] (-2, 17, 9) -- (-1, 18, 7);
\draw[edge] (-1, -18, 7) -- (1, -18, 7);
\draw[edge] (-1, -16, 11) -- (-1, -11, 16);
\draw[edge] (-1, -16, 11) -- (1, -16, 11);
\draw[edge] (-1, -11, 16) -- (1, -11, 16);
\draw[edge] (-1, -7, 18) -- (1, -7, 18);
\draw[edge] (-1, 7, 18) -- (1, 7, 18);
\draw[edge] (-1, 11, -16) -- (-1, 16, -11);
\draw[edge] (-1, 11, 16) -- (-1, 16, 11);
\draw[edge] (-1, 11, 16) -- (1, 11, 16);
\draw[edge] (-1, 16, 11) -- (1, 16, 11);
\draw[edge] (-1, 18, -7) -- (1, 18, -7);
\draw[edge] (-1, 18, 7) -- (1, 18, 7);
\draw[edge] (1, -18, 7) -- (2, -17, 9);
\draw[edge] (1, -16, 11) -- (1, -11, 16);
\draw[edge] (1, -16, 11) -- (2, -17, 9);
\draw[edge] (1, -11, 16) -- (2, -9, 17);
\draw[edge] (1, -7, 18) -- (2, -9, 17);
\draw[edge] (1, -7, 18) -- (7, -1, 18);
\draw[edge] (1, 7, 18) -- (2, 9, 17);
\draw[edge] (1, 7, 18) -- (7, 1, 18);
\draw[edge] (1, 11, 16) -- (1, 16, 11);
\draw[edge] (1, 11, 16) -- (2, 9, 17);
\draw[edge] (1, 16, 11) -- (2, 17, 9);
\draw[edge] (1, 18, -7) -- (7, 18, -1);
\draw[edge] (1, 18, 7) -- (2, 17, 9);
\draw[edge] (1, 18, 7) -- (7, 18, 1);
\draw[edge] (2, -9, 17) -- (9, -2, 17);
\draw[edge] (2, 9, 17) -- (9, 2, 17);
\draw[edge] (2, 17, 9) -- (9, 17, 2);
\draw[edge] (7, -1, 18) -- (7, 1, 18);
\draw[edge] (7, -1, 18) -- (9, -2, 17);
\draw[edge] (7, 1, 18) -- (9, 2, 17);
\draw[edge] (7, 18, -1) -- (7, 18, 1);
\draw[edge] (7, 18, 1) -- (9, 17, 2);
\draw[edge] (9, -2, 17) -- (11, -1, 16);
\draw[edge] (9, 2, 17) -- (11, 1, 16);
\draw[edge] (9, 17, 2) -- (11, 16, 1);
\draw[edge] (11, -1, 16) -- (11, 1, 16);
\draw[edge] (11, 1, 16) -- (16, 1, 11);
\draw[edge] (11, 16, 1) -- (16, 11, 1);
\draw[edge] (16, 1, 11) -- (17, 2, 9);
\draw[edge] (16, 11, 1) -- (17, 9, 2);
\draw[edge] (17, 2, 9) -- (17, 9, 2);
\node[vertex] at (-18, -7, -1)     {};
\node[vertex] at (-18, -7, 1)     {};
\node[vertex] at (-18, -1, -7)     {};
\node[vertex] at (-18, -1, 7)     {};
\node[vertex] at (-18, 1, -7)     {};
\node[vertex] at (-18, 1, 7)     {};
\node[vertex] at (-18, 7, -1)     {};
\node[vertex] at (-18, 7, 1)     {};
\node[vertex] at (-17, -9, -2)     {};
\node[vertex] at (-17, -9, 2)     {};
\node[vertex] at (-17, -2, -9)     {};
\node[vertex] at (-17, -2, 9)     {};
\node[vertex] at (-17, 2, -9)     {};
\node[vertex] at (-17, 2, 9)     {};
\node[vertex] at (-17, 9, -2)     {};
\node[vertex] at (-17, 9, 2)     {};
\node[vertex] at (-16, -11, -1)     {};
\node[vertex] at (-16, -11, 1)     {};
\node[vertex] at (-16, -1, -11)     {};
\node[vertex] at (-16, -1, 11)     {};
\node[vertex] at (-16, 1, -11)     {};
\node[vertex] at (-16, 1, 11)     {};
\node[vertex] at (-16, 11, -1)     {};
\node[vertex] at (-16, 11, 1)     {};
\node[vertex] at (-11, -16, -1)     {};
\node[vertex] at (-11, -16, 1)     {};
\node[vertex] at (-11, -1, 16)     {};
\node[vertex] at (-11, 1, -16)     {};
\node[vertex] at (-11, 1, 16)     {};
\node[vertex] at (-11, 16, -1)     {};
\node[vertex] at (-11, 16, 1)     {};
\node[vertex] at (-9, -17, 2)     {};
\node[vertex] at (-9, -2, 17)     {};
\node[vertex] at (-9, 2, -17)     {};
\node[vertex] at (-9, 2, 17)     {};
\node[vertex] at (-9, 17, -2)     {};
\node[vertex] at (-9, 17, 2)     {};
\node[vertex] at (-7, -18, 1)     {};
\node[vertex] at (-7, -1, 18)     {};
\node[vertex] at (-7, 1, 18)     {};
\node[vertex] at (-7, 18, -1)     {};
\node[vertex] at (-7, 18, 1)     {};
\node[vertex] at (-2, -17, 9)     {};
\node[vertex] at (-2, -9, 17)     {};
\node[vertex] at (-2, 9, -17)     {};
\node[vertex] at (-2, 9, 17)     {};
\node[vertex] at (-2, 17, -9)     {};
\node[vertex] at (-2, 17, 9)     {};
\node[vertex] at (-1, -18, 7)     {};
\node[vertex] at (-1, -16, 11)     {};
\node[vertex] at (-1, -11, 16)     {};
\node[vertex] at (-1, -7, 18)     {};
\node[vertex] at (-1, 7, 18)     {};
\node[vertex] at (-1, 11, -16)     {};
\node[vertex] at (-1, 11, 16)     {};
\node[vertex] at (-1, 16, -11)     {};
\node[vertex] at (-1, 16, 11)     {};
\node[vertex] at (-1, 18, -7)     {};
\node[vertex] at (-1, 18, 7)     {};
\node[vertex] at (1, -18, 7)     {};
\node[vertex] at (1, -16, 11)     {};
\node[vertex] at (1, -11, 16)     {};
\node[vertex] at (1, -7, 18)     {};
\node[vertex] at (1, 7, 18)     {};
\node[vertex] at (1, 11, 16)     {};
\node[vertex] at (1, 16, 11)     {};
\node[vertex] at (1, 18, -7)     {};
\node[vertex] at (1, 18, 7)     {};
\node[vertex] at (2, -17, 9)     {};
\node[vertex] at (2, -9, 17)     {};
\node[vertex] at (2, 9, 17)     {};
\node[vertex] at (2, 17, 9)     {};
\node[vertex] at (7, -1, 18)     {};
\node[vertex] at (7, 1, 18)     {};
\node[vertex] at (7, 18, -1)     {};
\node[vertex] at (7, 18, 1)     {};
\node[vertex] at (9, -2, 17)     {};
\node[vertex] at (9, 2, 17)     {};
\node[vertex] at (9, 17, 2)     {};
\node[vertex] at (11, -1, 16)     {};
\node[vertex] at (11, 1, 16)     {};
\node[vertex] at (11, 16, 1)     {};
\node[vertex] at (16, 1, 11)     {};
\node[vertex] at (16, 11, 1)     {};
\node[vertex] at (17, 2, 9)     {};
\node[vertex] at (17, 9, 2)     {};
\end{tikzpicture}

%% file: tikz/permasso3.tex
\begin{tikzpicture}%
	[x={(0.178656cm, 0.726606cm)},
	y={(-0.806737cm, 0.494173cm)},
	z={(0.563256cm, 0.477323cm)},
	scale=0.1200,
	back/.style={loosely dotted, thin},
	edge/.style={color=black, thick},
	facet/.style={fill=white,fill opacity=0.800000},
	vertex/.style={inner sep=1pt,circle,draw=black!25!black,fill=black!75!black,thick,anchor=base}]
%
%
\coordinate (-19, -5, 0) at (-19, -5, 0);
\coordinate (-19, 0, -5) at (-19, 0, -5);
\coordinate (-19, 0, 5) at (-19, 0, 5);
\coordinate (-19, 5, 0) at (-19, 5, 0);
\coordinate (-17, -9, -2) at (-17, -9, -2);
\coordinate (-17, -9, 2) at (-17, -9, 2);
\coordinate (-17, -2, -9) at (-17, -2, -9);
\coordinate (-17, -2, 9) at (-17, -2, 9);
\coordinate (-17, 2, -9) at (-17, 2, -9);
\coordinate (-17, 2, 9) at (-17, 2, 9);
\coordinate (-17, 9, -2) at (-17, 9, -2);
\coordinate (-17, 9, 2) at (-17, 9, 2);
\coordinate (-16, -11, -1) at (-16, -11, -1);
\coordinate (-16, -11, 1) at (-16, -11, 1);
\coordinate (-16, -1, -11) at (-16, -1, -11);
\coordinate (-16, -1, 11) at (-16, -1, 11);
\coordinate (-16, 1, -11) at (-16, 1, -11);
\coordinate (-16, 1, 11) at (-16, 1, 11);
\coordinate (-16, 11, -1) at (-16, 11, -1);
\coordinate (-16, 11, 1) at (-16, 11, 1);
\coordinate (-11, -16, -1) at (-11, -16, -1);
\coordinate (-11, -16, 1) at (-11, -16, 1);
\coordinate (-11, -1, -16) at (-11, -1, -16);
\coordinate (-11, -1, 16) at (-11, -1, 16);
\coordinate (-11, 1, -16) at (-11, 1, -16);
\coordinate (-11, 1, 16) at (-11, 1, 16);
\coordinate (-11, 16, -1) at (-11, 16, -1);
\coordinate (-11, 16, 1) at (-11, 16, 1);
\coordinate (-9, -17, -2) at (-9, -17, -2);
\coordinate (-9, -17, 2) at (-9, -17, 2);
\coordinate (-9, -2, -17) at (-9, -2, -17);
\coordinate (-9, -2, 17) at (-9, -2, 17);
\coordinate (-9, 2, -17) at (-9, 2, -17);
\coordinate (-9, 2, 17) at (-9, 2, 17);
\coordinate (-9, 17, -2) at (-9, 17, -2);
\coordinate (-9, 17, 2) at (-9, 17, 2);
\coordinate (-5, -19, 0) at (-5, -19, 0);
\coordinate (-5, 0, -19) at (-5, 0, -19);
\coordinate (-5, 0, 19) at (-5, 0, 19);
\coordinate (-5, 19, 0) at (-5, 19, 0);
\coordinate (-2, -17, -9) at (-2, -17, -9);
\coordinate (-2, -17, 9) at (-2, -17, 9);
\coordinate (-2, -9, -17) at (-2, -9, -17);
\coordinate (-2, -9, 17) at (-2, -9, 17);
\coordinate (-2, 9, -17) at (-2, 9, -17);
\coordinate (-2, 9, 17) at (-2, 9, 17);
\coordinate (-2, 17, -9) at (-2, 17, -9);
\coordinate (-2, 17, 9) at (-2, 17, 9);
\coordinate (-1, -16, -11) at (-1, -16, -11);
\coordinate (-1, -16, 11) at (-1, -16, 11);
\coordinate (-1, -11, -16) at (-1, -11, -16);
\coordinate (-1, -11, 16) at (-1, -11, 16);
\coordinate (-1, 11, -16) at (-1, 11, -16);
\coordinate (-1, 11, 16) at (-1, 11, 16);
\coordinate (-1, 16, -11) at (-1, 16, -11);
\coordinate (-1, 16, 11) at (-1, 16, 11);
\coordinate (0, -19, -5) at (0, -19, -5);
\coordinate (0, -19, 5) at (0, -19, 5);
\coordinate (0, -5, -19) at (0, -5, -19);
\coordinate (0, -5, 19) at (0, -5, 19);
\coordinate (0, 5, -19) at (0, 5, -19);
\coordinate (0, 5, 19) at (0, 5, 19);
\coordinate (0, 19, -5) at (0, 19, -5);
\coordinate (0, 19, 5) at (0, 19, 5);
\coordinate (1, -16, -11) at (1, -16, -11);
\coordinate (1, -16, 11) at (1, -16, 11);
\coordinate (1, -11, -16) at (1, -11, -16);
\coordinate (1, -11, 16) at (1, -11, 16);
\coordinate (1, 11, -16) at (1, 11, -16);
\coordinate (1, 11, 16) at (1, 11, 16);
\coordinate (1, 16, -11) at (1, 16, -11);
\coordinate (1, 16, 11) at (1, 16, 11);
\coordinate (2, -17, -9) at (2, -17, -9);
\coordinate (2, -17, 9) at (2, -17, 9);
\coordinate (2, -9, -17) at (2, -9, -17);
\coordinate (2, -9, 17) at (2, -9, 17);
\coordinate (2, 9, -17) at (2, 9, -17);
\coordinate (2, 9, 17) at (2, 9, 17);
\coordinate (2, 17, -9) at (2, 17, -9);
\coordinate (2, 17, 9) at (2, 17, 9);
\coordinate (5, -19, 0) at (5, -19, 0);
\coordinate (5, 0, -19) at (5, 0, -19);
\coordinate (5, 0, 19) at (5, 0, 19);
\coordinate (5, 19, 0) at (5, 19, 0);
\coordinate (9, -17, -2) at (9, -17, -2);
\coordinate (9, -17, 2) at (9, -17, 2);
\coordinate (9, -2, -17) at (9, -2, -17);
\coordinate (9, -2, 17) at (9, -2, 17);
\coordinate (9, 2, -17) at (9, 2, -17);
\coordinate (9, 2, 17) at (9, 2, 17);
\coordinate (9, 17, -2) at (9, 17, -2);
\coordinate (9, 17, 2) at (9, 17, 2);
\coordinate (11, -16, -1) at (11, -16, -1);
\coordinate (11, -16, 1) at (11, -16, 1);
\coordinate (11, -1, -16) at (11, -1, -16);
\coordinate (11, -1, 16) at (11, -1, 16);
\coordinate (11, 1, -16) at (11, 1, -16);
\coordinate (11, 1, 16) at (11, 1, 16);
\coordinate (11, 16, -1) at (11, 16, -1);
\coordinate (11, 16, 1) at (11, 16, 1);
\coordinate (16, -11, -1) at (16, -11, -1);
\coordinate (16, -11, 1) at (16, -11, 1);
\coordinate (16, -1, -11) at (16, -1, -11);
\coordinate (16, -1, 11) at (16, -1, 11);
\coordinate (16, 1, -11) at (16, 1, -11);
\coordinate (16, 1, 11) at (16, 1, 11);
\coordinate (16, 11, -1) at (16, 11, -1);
\coordinate (16, 11, 1) at (16, 11, 1);
\coordinate (17, -9, -2) at (17, -9, -2);
\coordinate (17, -9, 2) at (17, -9, 2);
\coordinate (17, -2, -9) at (17, -2, -9);
\coordinate (17, -2, 9) at (17, -2, 9);
\coordinate (17, 2, -9) at (17, 2, -9);
\coordinate (17, 2, 9) at (17, 2, 9);
\coordinate (17, 9, -2) at (17, 9, -2);
\coordinate (17, 9, 2) at (17, 9, 2);
\coordinate (19, -5, 0) at (19, -5, 0);
\coordinate (19, 0, -5) at (19, 0, -5);
\coordinate (19, 0, 5) at (19, 0, 5);
\coordinate (19, 5, 0) at (19, 5, 0);
\draw[edge,back] (-16, -1, -11) -- (-11, -1, -16);
\draw[edge,back] (-11, -16, -1) -- (-9, -17, -2);
\draw[edge,back] (-11, -1, -16) -- (-11, 1, -16);
\draw[edge,back] (-11, -1, -16) -- (-9, -2, -17);
\draw[edge,back] (-9, -17, -2) -- (-5, -19, 0);
\draw[edge,back] (-9, -17, -2) -- (-2, -17, -9);
\draw[edge,back] (-9, -2, -17) -- (-5, 0, -19);
\draw[edge,back] (-9, -2, -17) -- (-2, -9, -17);
\draw[edge,back] (-9, 2, -17) -- (-5, 0, -19);
\draw[edge,back] (-5, -19, 0) -- (0, -19, -5);
\draw[edge,back] (-5, 0, -19) -- (0, -5, -19);
\draw[edge,back] (-5, 0, -19) -- (0, 5, -19);
\draw[edge,back] (-2, -17, -9) -- (-1, -16, -11);
\draw[edge,back] (-2, -17, -9) -- (0, -19, -5);
\draw[edge,back] (-2, -9, -17) -- (-1, -11, -16);
\draw[edge,back] (-2, -9, -17) -- (0, -5, -19);
\draw[edge,back] (-2, 9, -17) -- (0, 5, -19);
\draw[edge,back] (-1, -16, -11) -- (-1, -11, -16);
\draw[edge,back] (-1, -16, -11) -- (1, -16, -11);
\draw[edge,back] (-1, -11, -16) -- (1, -11, -16);
\draw[edge,back] (-1, 11, -16) -- (1, 11, -16);
\draw[edge,back] (-1, 16, -11) -- (1, 16, -11);
\draw[edge,back] (0, -19, -5) -- (2, -17, -9);
\draw[edge,back] (0, -19, -5) -- (5, -19, 0);
\draw[edge,back] (0, -19, 5) -- (5, -19, 0);
\draw[edge,back] (0, -5, -19) -- (2, -9, -17);
\draw[edge,back] (0, -5, -19) -- (5, 0, -19);
\draw[edge,back] (0, 5, -19) -- (2, 9, -17);
\draw[edge,back] (0, 5, -19) -- (5, 0, -19);
\draw[edge,back] (0, 19, -5) -- (2, 17, -9);
\draw[edge,back] (1, -16, -11) -- (1, -11, -16);
\draw[edge,back] (1, -16, -11) -- (2, -17, -9);
\draw[edge,back] (1, -11, -16) -- (2, -9, -17);
\draw[edge,back] (1, 11, -16) -- (1, 16, -11);
\draw[edge,back] (1, 11, -16) -- (2, 9, -17);
\draw[edge,back] (1, 16, -11) -- (2, 17, -9);
\draw[edge,back] (2, -17, -9) -- (9, -17, -2);
\draw[edge,back] (2, -17, 9) -- (9, -17, 2);
\draw[edge,back] (2, -9, -17) -- (9, -2, -17);
\draw[edge,back] (2, 9, -17) -- (9, 2, -17);
\draw[edge,back] (2, 17, -9) -- (9, 17, -2);
\draw[edge,back] (5, -19, 0) -- (9, -17, -2);
\draw[edge,back] (5, -19, 0) -- (9, -17, 2);
\draw[edge,back] (5, 0, -19) -- (9, -2, -17);
\draw[edge,back] (5, 0, -19) -- (9, 2, -17);
\draw[edge,back] (5, 19, 0) -- (9, 17, -2);
\draw[edge,back] (9, -17, -2) -- (11, -16, -1);
\draw[edge,back] (9, -17, 2) -- (11, -16, 1);
\draw[edge,back] (9, -2, -17) -- (11, -1, -16);
\draw[edge,back] (9, 2, -17) -- (11, 1, -16);
\draw[edge,back] (9, 17, -2) -- (11, 16, -1);
\draw[edge,back] (11, -16, -1) -- (11, -16, 1);
\draw[edge,back] (11, -16, -1) -- (16, -11, -1);
\draw[edge,back] (11, -16, 1) -- (16, -11, 1);
\draw[edge,back] (11, -1, -16) -- (11, 1, -16);
\draw[edge,back] (11, -1, -16) -- (16, -1, -11);
\draw[edge,back] (11, -1, 16) -- (16, -1, 11);
\draw[edge,back] (11, 1, -16) -- (16, 1, -11);
\draw[edge,back] (11, 16, -1) -- (11, 16, 1);
\draw[edge,back] (11, 16, -1) -- (16, 11, -1);
\draw[edge,back] (16, -11, -1) -- (16, -11, 1);
\draw[edge,back] (16, -11, -1) -- (17, -9, -2);
\draw[edge,back] (16, -11, 1) -- (17, -9, 2);
\draw[edge,back] (16, -1, -11) -- (16, 1, -11);
\draw[edge,back] (16, -1, -11) -- (17, -2, -9);
\draw[edge,back] (16, -1, 11) -- (16, 1, 11);
\draw[edge,back] (16, -1, 11) -- (17, -2, 9);
\draw[edge,back] (16, 1, -11) -- (17, 2, -9);
\draw[edge,back] (16, 11, -1) -- (16, 11, 1);
\draw[edge,back] (16, 11, -1) -- (17, 9, -2);
\draw[edge,back] (17, -9, -2) -- (17, -2, -9);
\draw[edge,back] (17, -9, -2) -- (19, -5, 0);
\draw[edge,back] (17, -9, 2) -- (17, -2, 9);
\draw[edge,back] (17, -9, 2) -- (19, -5, 0);
\draw[edge,back] (17, -2, -9) -- (19, 0, -5);
\draw[edge,back] (17, -2, 9) -- (19, 0, 5);
\draw[edge,back] (17, 2, -9) -- (17, 9, -2);
\draw[edge,back] (17, 2, -9) -- (19, 0, -5);
\draw[edge,back] (17, 2, 9) -- (19, 0, 5);
\draw[edge,back] (17, 9, -2) -- (19, 5, 0);
\draw[edge,back] (17, 9, 2) -- (19, 5, 0);
\draw[edge,back] (19, -5, 0) -- (19, 0, -5);
\draw[edge,back] (19, -5, 0) -- (19, 0, 5);
\draw[edge,back] (19, 0, -5) -- (19, 5, 0);
\draw[edge,back] (19, 0, 5) -- (19, 5, 0);
\node[vertex] at (19, 0, 5)     {};
\node[vertex] at (19, 5, 0)     {};
\node[vertex] at (16, 11, -1)     {};
\node[vertex] at (17, 9, -2)     {};
\node[vertex] at (17, 2, -9)     {};
\node[vertex] at (19, 0, -5)     {};
\node[vertex] at (19, -5, 0)     {};
\node[vertex] at (-11, -1, -16)     {};
\node[vertex] at (-9, -17, -2)     {};
\node[vertex] at (-9, -2, -17)     {};
\node[vertex] at (-2, -17, -9)     {};
\node[vertex] at (-2, -9, -17)     {};
\node[vertex] at (-1, -16, -11)     {};
\node[vertex] at (-1, -11, -16)     {};
\node[vertex] at (0, -19, -5)     {};
\node[vertex] at (-5, 0, -19)     {};
\node[vertex] at (0, -5, -19)     {};
\node[vertex] at (1, 16, -11)     {};
\node[vertex] at (2, 17, -9)     {};
\node[vertex] at (2, -9, -17)     {};
\node[vertex] at (5, 0, -19)     {};
\node[vertex] at (9, -2, -17)     {};
\node[vertex] at (0, 5, -19)     {};
\node[vertex] at (1, 11, -16)     {};
\node[vertex] at (1, -11, -16)     {};
\node[vertex] at (5, -19, 0)     {};
\node[vertex] at (1, -16, -11)     {};
\node[vertex] at (2, -17, -9)     {};
\node[vertex] at (9, 17, -2)     {};
\node[vertex] at (9, -17, 2)     {};
\node[vertex] at (11, -16, 1)     {};
\node[vertex] at (16, -11, 1)     {};
\node[vertex] at (16, -1, 11)     {};
\node[vertex] at (17, -9, 2)     {};
\node[vertex] at (17, -2, 9)     {};
\node[vertex] at (2, 9, -17)     {};
\node[vertex] at (9, -17, -2)     {};
\node[vertex] at (9, 2, -17)     {};
\node[vertex] at (11, -16, -1)     {};
\node[vertex] at (11, -1, -16)     {};
\node[vertex] at (11, 1, -16)     {};
\node[vertex] at (16, -11, -1)     {};
\node[vertex] at (16, -1, -11)     {};
\node[vertex] at (17, -9, -2)     {};
\node[vertex] at (17, -2, -9)     {};
\node[vertex] at (11, 16, -1)     {};
\node[vertex] at (16, 1, -11)     {};
\fill[facet] (0, 5, 19) -- (-5, 0, 19) -- (-9, 2, 17) -- (-2, 9, 17) -- cycle {};
\fill[facet] (-19, 5, 0) -- (-19, 0, -5) -- (-19, -5, 0) -- (-19, 0, 5) -- cycle {};
\fill[facet] (-17, -2, 9) -- (-19, 0, 5) -- (-19, -5, 0) -- (-17, -9, 2) -- cycle {};
\fill[facet,fill=red] (-16, 1, -11) -- (-17, 2, -9) -- (-19, 0, -5) -- (-17, -2, -9) -- (-16, -1, -11) -- cycle {};
\fill[facet] (-17, 9, -2) -- (-19, 5, 0) -- (-19, 0, -5) -- (-17, 2, -9) -- cycle {};
\fill[facet,fill=red] (-16, 1, 11) -- (-17, 2, 9) -- (-19, 0, 5) -- (-17, -2, 9) -- (-16, -1, 11) -- cycle {};
\fill[facet] (-17, 9, 2) -- (-19, 5, 0) -- (-19, 0, 5) -- (-17, 2, 9) -- cycle {};
\fill[facet] (-17, -2, -9) -- (-19, 0, -5) -- (-19, -5, 0) -- (-17, -9, -2) -- cycle {};
\fill[facet,fill=red] (-16, -11, 1) -- (-17, -9, 2) -- (-19, -5, 0) -- (-17, -9, -2) -- (-16, -11, -1) -- cycle {};
\fill[facet,fill=red] (-16, 11, 1) -- (-17, 9, 2) -- (-19, 5, 0) -- (-17, 9, -2) -- (-16, 11, -1) -- cycle {};
\fill[facet] (0, 19, -5) -- (-5, 19, 0) -- (-9, 17, -2) -- (-2, 17, -9) -- cycle {};
\fill[facet] (-11, -16, 1) -- (-16, -11, 1) -- (-16, -11, -1) -- (-11, -16, -1) -- cycle {};
\fill[facet] (-11, 1, 16) -- (-16, 1, 11) -- (-16, -1, 11) -- (-11, -1, 16) -- cycle {};
\fill[facet] (-11, 16, 1) -- (-16, 11, 1) -- (-16, 11, -1) -- (-11, 16, -1) -- cycle {};
\fill[facet] (-1, -11, 16) -- (-2, -9, 17) -- (-9, -2, 17) -- (-11, -1, 16) -- (-16, -1, 11) -- (-17, -2, 9) -- (-17, -9, 2) -- (-16, -11, 1) -- (-11, -16, 1) -- (-9, -17, 2) -- (-2, -17, 9) -- (-1, -16, 11) -- cycle {};
\fill[facet,fill=red] (-5, 0, 19) -- (-9, -2, 17) -- (-11, -1, 16) -- (-11, 1, 16) -- (-9, 2, 17) -- cycle {};
\fill[facet,fill=red] (2, -17, 9) -- (0, -19, 5) -- (-2, -17, 9) -- (-1, -16, 11) -- (1, -16, 11) -- cycle {};
\fill[facet,fill=red] (-5, 19, 0) -- (-9, 17, -2) -- (-11, 16, -1) -- (-11, 16, 1) -- (-9, 17, 2) -- cycle {};
\fill[facet] (0, -19, 5) -- (-5, -19, 0) -- (-9, -17, 2) -- (-2, -17, 9) -- cycle {};
\fill[facet,fill=red] (2, 9, 17) -- (0, 5, 19) -- (-2, 9, 17) -- (-1, 11, 16) -- (1, 11, 16) -- cycle {};
\fill[facet] (-1, 16, -11) -- (-2, 17, -9) -- (-9, 17, -2) -- (-11, 16, -1) -- (-16, 11, -1) -- (-17, 9, -2) -- (-17, 2, -9) -- (-16, 1, -11) -- (-11, 1, -16) -- (-9, 2, -17) -- (-2, 9, -17) -- (-1, 11, -16) -- cycle {};
\fill[facet] (5, 0, 19) -- (0, -5, 19) -- (-5, 0, 19) -- (0, 5, 19) -- cycle {};
\fill[facet] (-1, 16, 11) -- (-2, 17, 9) -- (-9, 17, 2) -- (-11, 16, 1) -- (-16, 11, 1) -- (-17, 9, 2) -- (-17, 2, 9) -- (-16, 1, 11) -- (-11, 1, 16) -- (-9, 2, 17) -- (-2, 9, 17) -- (-1, 11, 16) -- cycle {};
\fill[facet] (0, 19, 5) -- (-5, 19, 0) -- (-9, 17, 2) -- (-2, 17, 9) -- cycle {};
\fill[facet] (9, 2, 17) -- (2, 9, 17) -- (0, 5, 19) -- (5, 0, 19) -- cycle {};
\fill[facet] (0, -5, 19) -- (-5, 0, 19) -- (-9, -2, 17) -- (-2, -9, 17) -- cycle {};
\fill[facet] (1, -11, 16) -- (-1, -11, 16) -- (-1, -16, 11) -- (1, -16, 11) -- cycle {};
\fill[facet,fill=red] (2, -9, 17) -- (0, -5, 19) -- (-2, -9, 17) -- (-1, -11, 16) -- (1, -11, 16) -- cycle {};
\fill[facet] (1, 16, 11) -- (-1, 16, 11) -- (-1, 11, 16) -- (1, 11, 16) -- cycle {};
\fill[facet] (9, -2, 17) -- (2, -9, 17) -- (0, -5, 19) -- (5, 0, 19) -- cycle {};
\fill[facet] (5, 19, 0) -- (0, 19, -5) -- (-5, 19, 0) -- (0, 19, 5) -- cycle {};
\fill[facet,fill=red] (2, 17, 9) -- (0, 19, 5) -- (-2, 17, 9) -- (-1, 16, 11) -- (1, 16, 11) -- cycle {};
\fill[facet] (9, 17, 2) -- (2, 17, 9) -- (0, 19, 5) -- (5, 19, 0) -- cycle {};
\fill[facet,fill=red] (11, 1, 16) -- (9, 2, 17) -- (5, 0, 19) -- (9, -2, 17) -- (11, -1, 16) -- cycle {};
\fill[facet] (17, 9, 2) -- (16, 11, 1) -- (11, 16, 1) -- (9, 17, 2) -- (2, 17, 9) -- (1, 16, 11) -- (1, 11, 16) -- (2, 9, 17) -- (9, 2, 17) -- (11, 1, 16) -- (16, 1, 11) -- (17, 2, 9) -- cycle {};
\draw[edge] (-19, -5, 0) -- (-19, 0, -5);
\draw[edge] (-19, -5, 0) -- (-19, 0, 5);
\draw[edge] (-19, -5, 0) -- (-17, -9, -2);
\draw[edge] (-19, -5, 0) -- (-17, -9, 2);
\draw[edge] (-19, 0, -5) -- (-19, 5, 0);
\draw[edge] (-19, 0, -5) -- (-17, -2, -9);
\draw[edge] (-19, 0, -5) -- (-17, 2, -9);
\draw[edge] (-19, 0, 5) -- (-19, 5, 0);
\draw[edge] (-19, 0, 5) -- (-17, -2, 9);
\draw[edge] (-19, 0, 5) -- (-17, 2, 9);
\draw[edge] (-19, 5, 0) -- (-17, 9, -2);
\draw[edge] (-19, 5, 0) -- (-17, 9, 2);
\draw[edge] (-17, -9, -2) -- (-17, -2, -9);
\draw[edge] (-17, -9, -2) -- (-16, -11, -1);
\draw[edge] (-17, -9, 2) -- (-17, -2, 9);
\draw[edge] (-17, -9, 2) -- (-16, -11, 1);
\draw[edge] (-17, -2, -9) -- (-16, -1, -11);
\draw[edge] (-17, -2, 9) -- (-16, -1, 11);
\draw[edge] (-17, 2, -9) -- (-17, 9, -2);
\draw[edge] (-17, 2, -9) -- (-16, 1, -11);
\draw[edge] (-17, 2, 9) -- (-17, 9, 2);
\draw[edge] (-17, 2, 9) -- (-16, 1, 11);
\draw[edge] (-17, 9, -2) -- (-16, 11, -1);
\draw[edge] (-17, 9, 2) -- (-16, 11, 1);
\draw[edge] (-16, -11, -1) -- (-16, -11, 1);
\draw[edge] (-16, -11, -1) -- (-11, -16, -1);
\draw[edge] (-16, -11, 1) -- (-11, -16, 1);
\draw[edge] (-16, -1, -11) -- (-16, 1, -11);
\draw[edge] (-16, -1, 11) -- (-16, 1, 11);
\draw[edge] (-16, -1, 11) -- (-11, -1, 16);
\draw[edge] (-16, 1, -11) -- (-11, 1, -16);
\draw[edge] (-16, 1, 11) -- (-11, 1, 16);
\draw[edge] (-16, 11, -1) -- (-16, 11, 1);
\draw[edge] (-16, 11, -1) -- (-11, 16, -1);
\draw[edge] (-16, 11, 1) -- (-11, 16, 1);
\draw[edge] (-11, -16, -1) -- (-11, -16, 1);
\draw[edge] (-11, -16, 1) -- (-9, -17, 2);
\draw[edge] (-11, -1, 16) -- (-11, 1, 16);
\draw[edge] (-11, -1, 16) -- (-9, -2, 17);
\draw[edge] (-11, 1, -16) -- (-9, 2, -17);
\draw[edge] (-11, 1, 16) -- (-9, 2, 17);
\draw[edge] (-11, 16, -1) -- (-11, 16, 1);
\draw[edge] (-11, 16, -1) -- (-9, 17, -2);
\draw[edge] (-11, 16, 1) -- (-9, 17, 2);
\draw[edge] (-9, -17, 2) -- (-5, -19, 0);
\draw[edge] (-9, -17, 2) -- (-2, -17, 9);
\draw[edge] (-9, -2, 17) -- (-5, 0, 19);
\draw[edge] (-9, -2, 17) -- (-2, -9, 17);
\draw[edge] (-9, 2, -17) -- (-2, 9, -17);
\draw[edge] (-9, 2, 17) -- (-5, 0, 19);
\draw[edge] (-9, 2, 17) -- (-2, 9, 17);
\draw[edge] (-9, 17, -2) -- (-5, 19, 0);
\draw[edge] (-9, 17, -2) -- (-2, 17, -9);
\draw[edge] (-9, 17, 2) -- (-5, 19, 0);
\draw[edge] (-9, 17, 2) -- (-2, 17, 9);
\draw[edge] (-5, -19, 0) -- (0, -19, 5);
\draw[edge] (-5, 0, 19) -- (0, -5, 19);
\draw[edge] (-5, 0, 19) -- (0, 5, 19);
\draw[edge] (-5, 19, 0) -- (0, 19, -5);
\draw[edge] (-5, 19, 0) -- (0, 19, 5);
\draw[edge] (-2, -17, 9) -- (-1, -16, 11);
\draw[edge] (-2, -17, 9) -- (0, -19, 5);
\draw[edge] (-2, -9, 17) -- (-1, -11, 16);
\draw[edge] (-2, -9, 17) -- (0, -5, 19);
\draw[edge] (-2, 9, -17) -- (-1, 11, -16);
\draw[edge] (-2, 9, 17) -- (-1, 11, 16);
\draw[edge] (-2, 9, 17) -- (0, 5, 19);
\draw[edge] (-2, 17, -9) -- (-1, 16, -11);
\draw[edge] (-2, 17, -9) -- (0, 19, -5);
\draw[edge] (-2, 17, 9) -- (-1, 16, 11);
\draw[edge] (-2, 17, 9) -- (0, 19, 5);
\draw[edge] (-1, -16, 11) -- (-1, -11, 16);
\draw[edge] (-1, -16, 11) -- (1, -16, 11);
\draw[edge] (-1, -11, 16) -- (1, -11, 16);
\draw[edge] (-1, 11, -16) -- (-1, 16, -11);
\draw[edge] (-1, 11, 16) -- (-1, 16, 11);
\draw[edge] (-1, 11, 16) -- (1, 11, 16);
\draw[edge] (-1, 16, 11) -- (1, 16, 11);
\draw[edge] (0, -19, 5) -- (2, -17, 9);
\draw[edge] (0, -5, 19) -- (2, -9, 17);
\draw[edge] (0, -5, 19) -- (5, 0, 19);
\draw[edge] (0, 5, 19) -- (2, 9, 17);
\draw[edge] (0, 5, 19) -- (5, 0, 19);
\draw[edge] (0, 19, -5) -- (5, 19, 0);
\draw[edge] (0, 19, 5) -- (2, 17, 9);
\draw[edge] (0, 19, 5) -- (5, 19, 0);
\draw[edge] (1, -16, 11) -- (1, -11, 16);
\draw[edge] (1, -16, 11) -- (2, -17, 9);
\draw[edge] (1, -11, 16) -- (2, -9, 17);
\draw[edge] (1, 11, 16) -- (1, 16, 11);
\draw[edge] (1, 11, 16) -- (2, 9, 17);
\draw[edge] (1, 16, 11) -- (2, 17, 9);
\draw[edge] (2, -9, 17) -- (9, -2, 17);
\draw[edge] (2, 9, 17) -- (9, 2, 17);
\draw[edge] (2, 17, 9) -- (9, 17, 2);
\draw[edge] (5, 0, 19) -- (9, -2, 17);
\draw[edge] (5, 0, 19) -- (9, 2, 17);
\draw[edge] (5, 19, 0) -- (9, 17, 2);
\draw[edge] (9, -2, 17) -- (11, -1, 16);
\draw[edge] (9, 2, 17) -- (11, 1, 16);
\draw[edge] (9, 17, 2) -- (11, 16, 1);
\draw[edge] (11, -1, 16) -- (11, 1, 16);
\draw[edge] (11, 1, 16) -- (16, 1, 11);
\draw[edge] (11, 16, 1) -- (16, 11, 1);
\draw[edge] (16, 1, 11) -- (17, 2, 9);
\draw[edge] (16, 11, 1) -- (17, 9, 2);
\draw[edge] (17, 2, 9) -- (17, 9, 2);

\node[vertex] at (-19, -5, 0)     {};
\node[vertex] at (-19, 0, -5)     {};
\node[vertex] at (-19, 0, 5)     {};
\node[vertex] at (-19, 5, 0)     {};
\node[vertex] at (-17, -9, -2)     {};
\node[vertex] at (-17, -9, 2)     {};
\node[vertex] at (-17, -2, -9)     {};
\node[vertex] at (-17, -2, 9)     {};
\node[vertex] at (-17, 2, -9)     {};
\node[vertex] at (-17, 2, 9)     {};
\node[vertex] at (-17, 9, -2)     {};
\node[vertex] at (-17, 9, 2)     {};
\node[vertex] at (-16, -11, -1)     {};
\node[vertex] at (-16, -11, 1)     {};
\node[vertex] at (-16, -1, -11)     {};
\node[vertex] at (-16, -1, 11)     {};
\node[vertex] at (-16, 1, -11)     {};
\node[vertex] at (-16, 1, 11)     {};
\node[vertex] at (-16, 11, -1)     {};
\node[vertex] at (-16, 11, 1)     {};
\node[vertex] at (-11, -16, -1)     {};
\node[vertex] at (-11, -16, 1)     {};
\node[vertex] at (-11, -1, 16)     {};
\node[vertex] at (-11, 1, -16)     {};
\node[vertex] at (-11, 1, 16)     {};
\node[vertex] at (-11, 16, -1)     {};
\node[vertex] at (-11, 16, 1)     {};
\node[vertex] at (-9, -17, 2)     {};
\node[vertex] at (-9, -2, 17)     {};
\node[vertex] at (-9, 2, -17)     {};
\node[vertex] at (-9, 2, 17)     {};
\node[vertex] at (-9, 17, -2)     {};
\node[vertex] at (-9, 17, 2)     {};
\node[vertex] at (-5, -19, 0)     {};
\node[vertex] at (-5, 0, 19)     {};
\node[vertex] at (-5, 19, 0)     {};
\node[vertex] at (-2, -17, 9)     {};
\node[vertex] at (-2, -9, 17)     {};
\node[vertex] at (-2, 9, -17)     {};
\node[vertex] at (-2, 9, 17)     {};
\node[vertex] at (-2, 17, -9)     {};
\node[vertex] at (-2, 17, 9)     {};
\node[vertex] at (-1, -16, 11)     {};
\node[vertex] at (-1, -11, 16)     {};
\node[vertex] at (-1, 11, -16)     {};
\node[vertex] at (-1, 11, 16)     {};
\node[vertex] at (-1, 16, -11)     {};
\node[vertex] at (-1, 16, 11)     {};
\node[vertex] at (0, -19, 5)     {};
\node[vertex] at (0, -5, 19)     {};
\node[vertex] at (0, 5, 19)     {};
\node[vertex] at (0, 19, -5)     {};
\node[vertex] at (0, 19, 5)     {};
\node[vertex] at (1, -16, 11)     {};
\node[vertex] at (1, -11, 16)     {};
\node[vertex] at (1, 11, 16)     {};
\node[vertex] at (1, 16, 11)     {};
\node[vertex] at (2, -17, 9)     {};
\node[vertex] at (2, -9, 17)     {};
\node[vertex] at (2, 9, 17)     {};
\node[vertex] at (2, 17, 9)     {};
\node[vertex] at (5, 0, 19)     {};
\node[vertex] at (5, 19, 0)     {};
\node[vertex] at (9, -2, 17)     {};
\node[vertex] at (9, 2, 17)     {};
\node[vertex] at (9, 17, 2)     {};
\node[vertex] at (11, -1, 16)     {};
\node[vertex] at (11, 1, 16)     {};
\node[vertex] at (11, 16, 1)     {};
\node[vertex] at (16, 1, 11)     {};
\node[vertex] at (16, 11, 1)     {};
\node[vertex] at (17, 2, 9)     {};
\node[vertex] at (17, 9, 2)     {};
\end{tikzpicture}